\documentclass[12pt]{amsart}
\usepackage{amsmath,amsthm,amssymb,verbatim,xcolor,microtype,graphicx,aliascnt}
\usepackage[shortlabels]{enumitem}
\usepackage[bookmarksopen=true,bookmarksopenlevel=1,bookmarksdepth=2,linktoc=page,colorlinks,linkcolor={red!70!black},citecolor={red!80!black},urlcolor={blue!80!black},pdfpagemode=UseOutlines,pdfstartview={fitH}]{hyperref}   
\usepackage[T1]{fontenc}
\usepackage[utf8]{inputenc}
\usepackage[english]{babel}\def\hyph{-\penalty0\hskip0pt\relax} 
\usepackage[top=3.0cm,bottom=3.0cm,left=3.0cm,right=3.0cm]{geometry}
\usepackage{tikz}
\usepackage{tikz-cd}
\usepackage[mathcal]{euscript}                               
\usepackage{cleveref}
\usetikzlibrary{patterns,snakes}

\usepackage{newtxtext,newtxmath}                              
\usepackage{bm}

\usepackage[colorinlistoftodos,bordercolor=orange,backgroundcolor=orange!20,linecolor=orange,textsize=footnotesize,textwidth=23mm]{todonotes} \makeatletter \providecommand \@dotsep{5} \def\listtodoname{List of Todos} \def\listoftodos{\@starttoc{tdo}\listtodoname} \makeatother 
\makeatletter\newcommand{\mylabel}[2]{#2\def\@currentlabel{#2}\label{#1}}\makeatother                 

\allowdisplaybreaks 

\usepackage{etoolbox}
\makeatletter
\patchcmd{\@startsection}{\@afterindenttrue}{\@afterindentfalse}{}{}             
\patchcmd{\part}{\bfseries}{\bfseries}{}{}
\patchcmd{\section}{\scshape}{\bfseries}{}{}\renewcommand{\@secnumfont}{\bfseries} 
\patchcmd{\@settitle}{\uppercasenonmath\@title}{\large}{}{}
\patchcmd{\@setauthors}{\MakeUppercase}{}{}{}
\addto{\captionsenglish}{} 
\addto{\captionsenglish}{} 
\addto{\captionsenglish}{} 
\makeatother

\usepackage{fancyhdr}

\addto\extrasenglish{}  
\theoremstyle{plain}
\newtheorem{thm}{Theorem}[section] 
\newaliascnt{lemma}{thm}\newtheorem{lemma}[lemma]{Lemma}\aliascntresetthe{lemma}
\newaliascnt{cor}{thm}\newtheorem{cor}[cor]{Corollary}\aliascntresetthe{cor}
\newaliascnt{prop}{thm}\newtheorem{prop}[prop]{Proposition}\aliascntresetthe{prop}
\newtheorem{thmA}{Theorem} 
\newaliascnt{propA}{thmA}\newtheorem{propA}[propA]{Proposition}\aliascntresetthe{propA}
\newaliascnt{lemmaA}{thmA}\aliascntresetthe{lemmaA}

\newtheorem*{prop*}{Proposition}
\newtheorem*{thm*}{Theorem}
\newtheorem*{lem*}{Lemma}
\newtheorem*{cor*}{Corollary}
\def\equationautorefname~#1\null{Equation~(#1)\null}

\theoremstyle{definition}
\newaliascnt{df}{thm}\newtheorem{df}[df]{Definition}\aliascntresetthe{df}
\newaliascnt{rem}{thm}\newtheorem{rem}[rem]{Remark}\aliascntresetthe{rem}
\newaliascnt{ex}{thm}\newtheorem{ex}[ex]{Example}\aliascntresetthe{ex}
\newaliascnt{conj}{thm}\aliascntresetthe{conj}
\newaliascnt{problem}{thm}\aliascntresetthe{problem}

\newtheorem*{df*}{Definition}
\newtheorem*{ex*}{Example}
\newtheorem*{rem*}{Remark}
\newtheorem*{question*}{Question}
\newtheorem*{problem*}{Problem}

\theoremstyle{remark}

\DeclareRobustCommand{\gobblefour}[5]{}    

\DeclareMathOperator{\upD}{D}   
\DeclareMathOperator{\upLin}{Lin}   
\DeclareMathOperator{\upR}{R}   
\DeclareMathOperator{\Gr}{Gr}
\DeclareMathOperator{\Dr}{Dr}
\DeclareMathOperator{\PolyGr}{PolyGr}
\DeclareMathOperator{\PolyDr}{PolyDr}
\DeclareMathOperator{\Pl}{Pl}
\DeclareMathOperator{\pl}{pl}
\DeclareMathOperator{\ulineGr}{\underline{Gr}}

\DeclareMathOperator{\Hom}{Hom}

\DeclareMathOperator{\Tracts}{Tracts}

\DeclareMathOperator{\Sets}{Sets}
\DeclareMathOperator{\Maps}{Maps}
\DeclareMathOperator{\Stab}{Stab}
\DeclareMathOperator{\sign}{sign}

\DeclareMathOperator{\conv}{conv}
\DeclareMathOperator{\rk}{rk}
\DeclareMathOperator{\colim}{colim\,}
\DeclareMathOperator{\hypersum}{\,\raisebox{-2.2pt}{\larger[2]{$\boxplus$}}\,}
\DeclareMathOperator{\Up}{\Pi^{\uparrow}}

\newcommand\D{{\mathbb D}}

\newcommand\F{{\mathbb F}}
\newcommand\G{{\mathbb G}}
\renewcommand\H{{\mathbb H}}

\newcommand\K{{\mathbb K}}

\newcommand\N{{\mathbb N}}

\renewcommand\P{{\mathbb P}}

\newcommand\R{{\mathbb R}}
\renewcommand\S{{\mathbb S}}
\newcommand\T{{\mathbb T}}
\newcommand\U{{\mathbb U}}

\newcommand\Z{{\mathbb Z}}

\newcommand\bJ{{\mathbf J}}
\newcommand\bN{{\mathbf N}}

\newcommand\cB{{\mathcal B}}

\newcommand\cF{{\mathcal F}}

\newcommand\cP{{\mathcal P}}

\newcommand\fR{{\mathfrak R}}

\newcommand\br{{\textup{\bf{r}}}}
\newcommand\balpha{{\bm{\alpha}}}
\newcommand\bbeta{{\bm{\beta}}}
\newcommand\bgamma{{\bm{\gamma}}}

\newcommand\biota{{\bm{\iota}}}
\newcommand\bchi{{\bm{\chi}}}
\newcommand\brho{{\bm{\rho}}}
\newcommand\bsigma{{\bm{\sigma}}}

\renewcommand{\epsilon}{\varepsilon}

\newcommand\0{{\mathbf{0}}}
\newcommand\1{{\mathbf{1}}}

\newcommand\rbar{{\bm\bar{r}}}

\newcommand\Jbar{{\bm\bar{J}}}
\newcommand\Funpm{{\F_1^\pm}}

\newcommand\ev{\textup{ev}}

\renewcommand\max{\textup{max}}
\renewcommand{\min}{\textup{min}}
\renewcommand\geq{\geqslant}
\renewcommand\leq{\leqslant}

\newcommand{\gen}[1]{\langle #1 \rangle}
\newcommand{\norm}[1]{|#1|}

\newcommand{\pastgen}[2]{#1/\!\!/\gen{#2}}
\newcommand{\cross}[5]{\mathchoice{\scalebox{1.3}{$\big[$}\,\raisebox{1pt}{$\begin{matrix}{\scalebox{0.9}{$#1$}}\hspace{-5pt}&{\scalebox{0.9}{$#2$}}\\[-2pt]{\scalebox{0.9}{$#3$}}\hspace{-5pt}&{{\scalebox{0.9}{$#4$}}}\end{matrix}$}\,\scalebox{1.3}{$\big]$}_{#5}}{\big[\begin{smallmatrix}{#1}&{#2}\\{#3}&{#4}\end{smallmatrix}\big]_{#5}}{}{}}   
\newcommand{\crossinv}[5]{\mathchoice{\scalebox{1.3}{$\big[$}\,\raisebox{1pt}{$\begin{matrix}{\scalebox{0.9}{$#1$}}\hspace{-5pt}&{\scalebox{0.9}{$#2$}}\\[-2pt]{\scalebox{0.9}{$#3$}}\hspace{-5pt}&{{\scalebox{0.9}{$#4$}}}\end{matrix}$}\,\scalebox{1.3}{$\big]$}_{#5}^{-1}}{\big[\begin{smallmatrix}{#1}&{#2}\\{#3}&{#4}\end{smallmatrix}\big]_{#5}^{-1}}{}{}}   
\renewcommand{\setminus}{\backslash}
\newcommand{\minor}[2]{\setminus #1 / #2}
\newcommand{\hyperplus}{\mathrel{\,\raisebox{-1.1pt}{\larger[-0]{$\boxplus$}}\,}}
\renewcommand\emptyset\varnothing
\makeatletter\def\smallunderbrace#1{\mathop{\vtop{\m@th\ialign{##\crcr $\hfil\displaystyle{#1}\hfil$\crcr \noalign{\kern3\p@\nointerlineskip} \tiny\upbracefill\crcr\noalign{\kern3\p@}}}}\limits}\makeatother

\title{Representation theory for polymatroids}

\author{Matthew Baker}
\address{\rm Matthew Baker, Georgia Institute of Technology}
\email{mbaker@math.gatech.edu}

\author{June Huh}
\address{\rm June Huh, Princeton University and Korea Institute for Advanced Study}
\email{huh@princeton.edu}

\author{Donggyu Kim}
\address{\rm Donggyu Kim, Georgia Institute of Technology}
\email{donggyu@gatech.edu}

\author{Mario Kummer}
\address{\rm Mario Kummer, Technische Universit\"at Dresden}
\email{mario.kummer@tu-dresden.de}

\author{Oliver Lorscheid}
\address{\rm Oliver Lorscheid, University of Groningen}
\email{olorscheid@rug.nl}

\hypersetup{pdftitle={Representation theory for polymatroids},pdfauthor={Matthew Baker, June Huh, Donggyu Kim, Mario Kummer and Oliver Lorscheid}}

\begin{document}


\begin{abstract}
 We develop a theory of representations of (discrete) polymatroids over tracts in terms of Pl\"ucker coordinates and suitable Pl\"ucker relations. As special cases, we recover polymatroids themselves as polymatroid representations over the Krasner hyperfield $\K$ and M-convex functions as polymatroid representations over the tropical hyperfield $\T_0$. 
 
 We introduce and study several useful operations for polymatroid representations, such as translation and refined notions of minors and duality which have better properties than the existing definitions; for example, deletion and contraction become dual operations (up to translation) in our setting. We also prove an \emph{idempotency principle} which asserts that polymatroids which are not translates of matroids are representable only over tracts that are idempotent in a certain specific sense (in particular $-1 = 1$).
 
 The space of all representations of a polymatroid $J$, which we call the \emph{thin Schubert cell} of $J$, is represented by an algebraic object called the {\emph{universal tract}} of $J$. When we restrict to just the $3$-term Pl\"ucker relations, we obtain the \emph{weak thin Schubert cell}, and passing to torus orbits yields the \emph{realization space}. These are represented by the \emph{universal pasture} and the \emph{foundation} of $J$, respectively. We exhibit a canonical bijection between the universal tract and the universal pasture, which is new even in the case of matroids, and we show that the foundation of a polymatroid is generated by \emph{cross ratios}. We also describe a complete list of multiplicative relations between cross ratios.
 
 Thin Schubert cells and realization spaces are canonically embedded in certain tori. Over idempotent tracts, we show that thin Schubert cells contain a canonical torus orbit and split naturally as a product of the realization space with this distinguished torus.
\end{abstract}

\maketitle

\begingroup
  \setlength{\parskip}{0pt}
  \begin{small}
  \tableofcontents
  \end{small}
\endgroup



\section*{Introduction}
\label{introduction}

Polymatroids were originally introduced in 1970 by Jack Edmonds
\cite{Edmonds70} in the context of combinatorial optimization theory. Kazuo Murota demonstrated the importance of discrete polymatroids\footnote{In the body of this paper, we will be concerned solely with discrete polymatroids and so we omit the modifier ``discrete'', calling these objects simply ``polymatroids''. However, for the purposes of this historical introduction, we distinguish between polymatroids in the sense of Edmonds (which are certain real-valued submodular functions) with discrete polymatroids (a.k.a. M-convex sets or integer polymatroids) in the sense of Murota et.~al.} to discrete convex analysis, via the equivalent notion of M-convex sets, and introduced M-convex functions as valuated generalizations. Much of this theory is developed in his book ``Discrete Convex Analysis'' \cite{Murota03}. 

Discrete polymatroids are treated more systematically from a combinatorial point of view analogous to matroid theory in Herzog and Hibi’s article \cite{Herzog-Hibi02}. As a sampling of the many papers generalizing results from matroid theory to discrete polymatroids, we mention \cite{Bernardi-Kalman-Postnikov22,Oxley-Semple-Whittle16,Oxley-Whittle93a,Oxley-Whittle93}.

Postnikov \cite{Postnikov09} investigated a class of polytopes called ``generalized permutohedra'', which arise as deformations of the standard permutohedron. These polytopes were subsequently recognized to be equivalent (up to translation) to base polytopes of polymatroids by an extension of the classical result by Gelfand, Goresky, MacPherson, and Serganova \cite{Gelfand-Goresky-MacPherson-Serganova87}.
Since polymatroids are cryptomorphically equivalent to their base polytopes, generalized permutohedra can be viewed as another combinatorial representation of polymatroids.

In recent years, there has been an explosion of interest in polymatroids, motivated in part by Petter Br\"and\'en's discovery \cite{Branden07} that the support of a homogeneous multivariate stable polynomial is an M-convex set, that is, the set of bases of a discrete polymatroid. This generalized an earlier result of Choe, Oxley, Sokal, and Wagner \cite{Choe-Oxley-Sokal-Wagner04} showing that the support of a \emph{multi-affine} homogeneous stable polynomial is the set of bases of a matroid, and was generalized by Br\"and\'en and Huh, who showed that the support of a \emph{Lorentzian} polynomial is always M-convex \cite[Theorem 2.25]{Branden-Huh20}. The latter result is particularly interesting because, in fact,   a subset of $\N^n$ is M-convex if and only if its normalized generating polynomial is a Lorentzian polynomial \cite[Theorem 3.10]{Branden-Huh20}.

Lorentzian polynomials have a close connection to combinatorial Hodge theory, and Hodge theory for matroids in the sense of Adiprasito--Huh--Katz \cite{Adiprasito-Huh-Katz18} was recently extended to discrete polymatroids by Pagaria and Pezzoli \cite{Pagaria-Pezzoli23}. Other papers exploring recently discovered connections between discrete polymatroids and combinatorial Hodge theory include the work of Crowley--Huh--Larson--Simpson--Wang \cite{Crowley-Huh-Larson-Simpson-Wang22}, Crowley--Simpson--Wang \cite{Crowley-Simpson-Wang24}, and Eur--Larson \cite{Eur-Larson24}.

The relevance of, and interest in, discrete polymatroids is by no means confined to combinatorial Hodge theory. As a sampling of some other interesting applications of discrete polymatroids, we mention:
\begin{itemize}\itemsep 5pt
\item Knutson and Tao's proofs of Horn's conjecture and the saturation conjecture \cite{Knutson-Tao99} implicitly employ M-convex functions, which correspond to $\mathbb{T}_0$-representations of discrete polymatroids in our terminology (where $\mathbb{T}_0$ denotes the tropical hyperfield). The ``hives'' that provide the technical foundation for their work are precisely $\mathbb{T}_0$-representations of $\Delta^r_3$. We discuss this example in greater detail in \autoref{subsection: hives}.
\item Br\"and\'en \cite{Branden11} disproved a conjecture of Helton and Vinnikov, that any real zero polynomial admits a certain determinantal representation, by studying the discrete polymatroid which Gurvits had previously associated to a hyperbolic polynomial.
\item Amini and Esteves showed that the tropicalization of linear series on an algebraic curve gives rise to certain families of tilings of vector spaces by discrete polymatroids \cite{Amini-Esteves24}.
\item Farr\`as, Mart\'i--Farr\'e, and Padr\'o provide applications of discrete polymatroids to cryptographic secret-sharing schemes \cite{Farras-Marti-Farre-Padro12}.
\item Discrete polymatroids arise naturally from the Klyaschko datum of a (framed) toric vector bundle, see \cite{Khan-Maclagan24}.
\end{itemize}

\subsection*{Motivation for the present paper} 
Recently some intriguing links between Lorentzian polynomials and representations of discrete polymatroids over tracts in the sense of this paper were discovered in \cite{BHKL1,BHKL2}. For example, for every discrete polymatroid $J$, the dimension of the space $L_J$ of Lorentzian polynomials with support $J$ is equal to the rank of the finitely generated abelian group $\widehat T_J^\times$ defined in \autoref{subsection: the universal tract} below.\footnote{As we will see, $\widehat T_J^\times$ is the multiplicative group of the ``extended universal tract'' of $J$, which is characterized by the property that ${\rm Hom}(\widehat T_J,F)$ is equal to the set of $F$-representations of $J$ for every tract $F$.} Moreover, $L_J$ is always contained in the space $\upR_J(\T_2)$ of representations of $J$ over the generalized triangular hyperfield $\T_2$ defined in \autoref{subsection: more examples of tracts} below.

These observations served as the primary motivation for working out the ``foundational'' results described in the present paper. Indeed, in order to precisely formulate and explore such a connection, one needs a rigorous development of the representation theory of discrete polymatroids over tracts, generalizing Baker and Bowler's theory \cite{Baker-Bowler19} of matroids with coefficients and Baker and Lorscheid's subsequent work \cite{Baker-Lorscheid20,Baker-Lorscheid21,Baker-Lorscheid-Zhang24} on foundations of matroids. This is what we systematically set out to do below.

\subsection*{A word of warning}
Linear representations of a discrete polymatroid $J$ over a field $F$, as studied in \cite{Farras-Marti-Farre-Padro12,Oxley-Whittle93}, differ from the $F$-representations of $J$ defined in this paper.
In the present work, we are primarily concerned with generalizing M-convex functions in the sense of Murota by viewing them as $\T_0$-representations; there does not appear to be a systematic way to generalize both this point of view and the traditional notion of linear representations over fields. 

\subsection*{Content outline} 
Before we describe the contents of this text in detail, we provide the following overview of our results.

\begin{enumerate}
 \item[\textbf{Part 1.}] 
 We introduce and study several operations for polymatroids: duality and translates (\autoref{subsection: duality and translates}), embedded minors (\autoref{subsection: embedded minors}), which refine previous notions of polymatroid minors, permutation and extension of variables (\autoref{subsection: permutation and extension of variables}), and direct sums (\autoref{subsection: direct sums}).
 \item[\textbf{Part 2.}]
 We exhibit a characterization of polymatroids in terms of Pl\"ucker relations (\autoref{thm: Pluecker relations for polymatroids}), which leads to the notion of a \emph{representation} of a polymatroid $J$ over a tract $F$ (\autoref{subsection: representations over tracts}). A feature of central importance is the \emph{idempotency principle for proper polymatroids} (\autoref{prop: idempotency principle}). M-convex functions are essentially the same as polymatroid representations over the tropical hyperfield (\autoref{subsection: M-convex functions as representations over the tropical hyperfield}).
 \item[\textbf{Part 3.}]
 We investigate thin Schubert cells of polymatroids and show that they are represented by the \emph{universal tract} (\autoref{prop: universal property of the universal tract}). In an analogous way, the weak thin Schubert cell, defined by the $3$-term Pl\"ucker relations, is represented by the \emph{universal pasture} (\autoref{prop: universal property of the universal pasture}). The \emph{comparison theorem} (\autoref{thm: bijection between the universal pasture and the universal tract}) establishes a bijection between the universal pasture and the universal tract . The realization space consists of polymatroid representations modulo rescaling and is represented by the \emph{foundation} (\autoref{prop: universal property of the foundation}). The foundation is generated by cross ratios, for which we establish a complete list of relations (\autoref{thm: generators and relations for the foundation}). We extend the polymatroid operations from Part 1 to thin Schubert cells and realization spaces (\autoref{section: representations of embedded minors and duals}).
 \item[\textbf{Part 4.}]
 We establish and study canonical embeddings of representations spaces, thin Schubert cells, and realization spaces into tori (\autoref{section: the torus embedding of the representation space}, \autoref{section: the Plucker embedding for thin Schubert cells}, and \autoref{section: the canonical torus embedding for realization spaces}, respectively). In the idempotent case, we study lineality spaces (\autoref{subsection: the lineality space}) and a decomposition of thin Schubert cells into a product of the realization space with a torus orbit (\autoref{subsection: decomposition of the representation space}).
\end{enumerate}

\subsection*{Polymatroids}
Discrete polymatroids appear in different cryptomorphic disguises in the literature: as a ``rank function'' $\rk\colon 2^{[n]}\to\N$ (where $2^{[n]}$ is the power set of $[n]=\{1,\dotsc,n\}$), as an integral polytope in $\R_{\geq0}^n$, as a collection of ``independent vectors'' in $\N^n$, and as a collection of ``bases'' contained in the dilated discrete simplex
\[
 \Delta^r_n \ = \ \big\{\alpha\in\N^n \, \big| \, \norm\alpha=r \big\}
\]
for some $r\geq0$ (the \emph{rank of the polymatroid}) where $\norm\alpha=\alpha_1+\dotsc+\alpha_n$ (``M-convex set''). See \autoref{section: polymatrois and M-convex sets} for an overview of these different descriptions of polymatroids, and of the relations between them. 

\begin{quote}
 \textbf{Convention:} We use the terms \emph{polymatroid} and \emph{M-convex set} interchangeably in this paper as synonyms for ``discrete polymatroid''. In particular, we omit the modifier ``discrete'' except when we wish to make comparisons to the literature.
\end{quote}

\begin{ex*}
 Matroids are polymatroids in a natural way: a matroid $M$ defines the polymatroid $J=\{ \sum_{i\in B} \epsilon_i \mid B\text{ is a basis of $M$}\}$, where $\epsilon_i$ is the $i$-th standard basis vector of $\N^n$. We say that a polymatroid $J$ is a \emph{matroid} if it of this form.
\end{ex*}

For the purposes of the introduction, it suffices to understand polymatroids using a novel characterization in terms of Pl\"ucker relations in the Krasner hyperfield (\autoref{thmA}). In order to properly formulate this result, we first introduce the concept of a {\em tract}.

\subsection*{Tracts}
Tracts were introduced by Baker and Bowler in \cite{Baker-Bowler19} as a generalization of fields over which one can still develop a satisfying theory of matroid representations. Examples of matroid representations over a tract include matroids themselves, oriented matroids, valuated matroids, and linear subspaces of a vector space.

In this paper, we develop a theory of \emph{polymatroid representations} over a tract.

\subsection*{Definition and examples of tracts}
A \emph{tract} is a multiplicatively written commutative monoid $F$ with an absorbing element $0\in F$ such that $F^\times=F-\{0\}$ forms an abelian group (the \emph{unit group}), together with a subset $N_F$ (the \emph{null set}) of the group semiring $F^+=\N[F^\times]$ that satisfies the following properties:
\begin{enumerate}\itemsep 5pt
 \item $N_F$ is an \emph{ideal} of $F^+$, that is,\ $0\in N_F$, $N_F+N_F=N_F$ and $F\cdot N_F=N_F$;
 \item there is an element $-1\in F$ such that for all $a,b\in F$, one has $a+b\in N_F$ if and only if $b=(-1)\cdot a$.
\end{enumerate}
We write $-a=(-1)\cdot a$ and call this element the \emph{additive inverse of $a$}. We write $a-b$ for the element $a+(-b)$ of $F^+$.

Examples of tracts are fields $F$, whose addition gets replaced by the null set 
\[\textstyle
 N_F \ = \ \big\{ \sum a_i \, \big| \, \sum a_i=0 \text{ as elements of }F\},
\]
the \emph{Krasner hyperfield} $\K=\{0,1\}$, with null set 
\[
 N_\K \ =\ \N-\{1\} \ = \ \{0, \ 1+1,\ 1+1+1,\ 1+1+1+1, \dotsc \},
\]
and the \emph{tropical hyperfield} $\T_0=\R_{\geq0}$, with null set 
\[\textstyle
 N_{\T_0} \ = \ \big\{ \sum a_i \, \big| \, \text{the maximum appears twice among }a_1,\dotsc,a_n\in\R_{\geq0}\big\}.
\]
Note that $-1=1$ in $\K$ and in $\T_0$. See \autoref{section: tracts} for a more comprehensive introduction to tracts.

\subsection*{Pl\"ucker relations for polymatroids}
The fundamental insight that leads to our notion of polymatroid representations over tracts is the following result, which is \autoref{thm: Pluecker relations for polymatroids}. 

We equip $\N^n$ with the component-wise partial order, where $\alpha\leq\beta$ if and only if $\alpha_i\leq\beta_i$ for all $i\in[n]$. For a subset $J\subseteq\Delta^r_n$, we define the infimum $\delta^-_J=\inf\, J$ and the supremum $\delta^+_J=\sup\, J$. These are the elements of $\N^n$ whose components are given by 
\[
 \delta^-_{J,i} \ = \ \min\{ \alpha_i \mid \alpha\in J\} \ \ \text{and} \ \ \delta^+_{J,i} \ = \ \max\{ \alpha_i \mid \alpha\in J\}.
\]
We write $\chi_J\colon\Delta^r_n\to\K$ for the characteristic function of $J$, defined by $\chi_J(\alpha)=1$ if and only if $\alpha\in J$.

\begin{thmA}\label{thmA}
A nonempty subset $J\subseteq\Delta^r_n$ is a polymatroid if and only if $\chi_J$ satisfies the Pl\"ucker relations
 \[
  \sum_{k=0}^s \ \chi_J(\alpha-\epsilon_{i_k}+\epsilon_{i_0}+\dotsb+\epsilon_{i_s})  \cdot  \chi_J(\alpha+\epsilon_{i_k}+\epsilon_{j_2}+\dotsb+\epsilon_{j_s})  \in  N_\K
 \]
 for any $s=2,\dotsc,r$, any $\alpha\in\Delta^{r-s}_n$, and any $i_0,\dotsc,i_s,j_2,\dotsc,j_s\in [n]$ such that
 \[
  \delta^-_J \ \leq \ \alpha \qquad \text{and} \qquad \alpha+\epsilon_{i_0}+\dotsb+\epsilon_{i_s}+\epsilon_{j_2}+\dotsb+\epsilon_{j_s} \ \leq \ \delta^+_J.
 \]
\end{thmA}

\subsection*{Polymatroid representations}

The generalization of the Pl\"ucker relations in \autoref{thmA} from $\K$ to an arbitrary tract $F$ involves a delicate choice of signs, unless $-1=1$ in $F$. 
It will turn out, \emph{a posteriori}, that a polymatroid which is not ``essentially'' a matroid is only representable over tracts in which $-1 = 1$, see \autoref{propC}.
Thus the reader can ignore the power of $-1$ in the following definition if he/she wishes, without much loss of generality.
The \emph{support} of a function $\rho\colon\Delta^r_n\to F$ is the set of $\alpha\in\Delta^r_n$ such that $\rho(\alpha)\neq0$. 

\begin{df*}
 Let $J\subseteq\Delta^r_n$ be a polymatroid and let $F$ be a tract. A \emph{strong Grassmann--Pl{\"u}cker representation of $J$ over $F$} is a map $\rho\colon\Delta^r_n\to F$, whose support is $J$, such that $\rho$ satisfies the \emph{Pl\"ucker relations}
 \begin{equation*}
  \sum_{k=0}^s \ (-1)^{k+\sigma(k)}\cdot \rho(\alpha -\epsilon_{i_k}+\epsilon_{i_0}+\dotsb+\epsilon_{i_s}) \cdot \rho(\alpha +\epsilon_{i_k}+\epsilon_{j_2}+\dotsb +\epsilon_{j_s}) \in   N_{F}
 \end{equation*}
 for any $2\leq s\leq r$, any $\alpha\in\Delta^{r-s}_n$, any $1\leq i_0\leq \dotsc\leq i_s\leq n$ and $1\leq j_2\leq\dotsc\leq j_s\leq n$ such that 
 \[
  \delta^-_J \ \leq \ \alpha \qquad \text{and} \qquad \alpha+\epsilon_{i_0}+\dotsb+\epsilon_{i_s}+\epsilon_{j_2}+\dotsb+\epsilon_{j_s} \ \leq \ \delta^+_J.
 \]
 where $\sigma(k)$ is the number of $k\in\{2,\ldots,s\}$ with $i_k<j_s$.
\end{df*}

\begin{rem*}
This definition turns out to be equivalent to \autoref{df: polymatroid representation}, which is stated using a slightly different formalism; see \autoref{subsection: simplified description of idempotent polymatroid representations} for details.
\end{rem*}

\begin{quote}
 \textbf{Convention:} Unless otherwise noted, we will use the term ``representation'' as shorthand for ``strong Grassmann--Pl{\"u}cker representation'' throughout this paper. We will also frequently refer to (strong Grassmann--Pl\"ucker) representations of $J$ over $F$ as ``$F$-representations of $J$''.
\end{quote}

\begin{ex*}
 Let $J\subseteq\Delta^r_n$ be a polymatroid. The characteristic function $\chi_J\colon\Delta^r_n\to\K$ of $J$ is the unique $\K$-representation of $J$. This establishes an equivalence between polymatroids and their (unique) $\K$-representations.
\end{ex*}

Another example of central interest are $\T_0$-representations, which are essentially the same as M-convex functions in the sense of Murota (\cite{Murota03}); cf.\ \autoref{subsection: M-convex functions as representations over the tropical hyperfield} for a definition. The following is \autoref{prop: M-convex functions are strong T_0-representations}.

\begin{propA}\label{propB}
 Let $\rho\colon\Delta^r_n\to \T_0$ be a map with support $J$. Then $-\log\rho\colon\Delta^r_n\to\R\cup\{\infty\}$ is an M-convex function if and only if $J$ is a polymatroid and $\rho$ is a $\T_0$-representation of $J$.
\end{propA}
 
\begin{ex*}[Hives]
 Hives are combinatorial gadgets that were introduced by Knutson and Tao in \cite{Knutson-Tao99} in their celebrated proof of the {\em saturation conjecture} for Littlewood--Richardson coefficients, which implies \emph{Horn's conjecture} on the possible eigenvalues of a sum of Hermitian matrices. As pointed out by Br\"and\'en in \cite{Branden10}, hives are naturally in bijection with $M$-concave functions (which are the negatives of M-convex functions) having support $\Delta^r_3$, and thus with $\T_0$-representations of $\Delta^r_3$. 
 
 Due to their explicit combinatorial nature, hives are easier to understand than $\T_0$\hyph representations in general, and we encourage the interested reader to have a look at \autoref{subsection: hives} before continuing with the introduction.
\end{ex*}

\subsection*{The idempotency principle for proper polymatroids}

Let $\delta^-_J=\inf\, J$. The \emph{reduction of $J$} is $\Jbar=\{\alpha\in\N^r\mid \alpha+\delta^-_J\in J\}$. We say that $J$ is a \emph{proper polymatroid} if $\Jbar$ is not a matroid. A tract $F$ is \emph{idempotent} if $-1=1$ and $1+1+1\in N_F$ (in other words, if $F$ is a $\K$-algebra). A tract $F$ is \emph{near-idempotent} if $-1=1$ and $1+1+x\in N_F$ for some $x\in F^\times$. The following is \autoref{prop: idempotency principle}.

\begin{propA}\label{propC}
 Let $J$ be a proper polymatroid and suppose there exists a representation of $J$ over $F$. Then $F$ is near-idempotent. If $\delta^+_{J,i}-\delta^-_{J,i}\geq3$ for some $i\in[n]$, then $F$ is idempotent.
\end{propA}

In other words, a polymatroid which is representable over some tract that is not near-idempotent must be a translate of a matroid. 
Note that a field (considered in the natural way as a tract) is never near-idempotent. Consequently, a polymatroid which is not the translate of a matroid is not representable over any field.

Since the representation theories of translates are essentially equivalent (cf.\ \autoref{thm: foundations of embedded minors}), and since the representation theory of matroids in the sense of this paper is well-understood:
\begin{center}
 {\it \textbf{We assume for the rest of the introduction that all tracts are near-idempotent.}}
\end{center}
 
In particular, we assume that $-1=1$, which allows us to suppress all signs, thus simplifying various expressions.

\subsection*{Discrepancy with other concepts of polymatroid representations}

Just as matroids can be viewed as combinatorial abstractions of hyperplane arrangements, polymatroids can be viewed as combinatorial abstractions of subspace arrangements. More precisely, for any field $K$, let $V_1,\ldots,V_n$ be $K$-linear subspaces of a fixed $K$-vector space $V$. Then
\[
\rk(S) := {\rm codim}_V ( \cap_{i \in S} V_i ) 
\]
is the rank function of a polymatroid $P$ on $E=[n]$.
The polymatroid $P$ is a matroid if and only if every $V_i$ is a hyperplane.
A polymatroid arising in this way is said to be \emph{linearly representable}, or \emph{realizable}, over $K$, and the subspace arrangement is called a \emph{linear representation}, or \emph{realization}, of $P$ over $K$, cf.~\cite{Oxley-Whittle93,Farras-Marti-Farre-Padro12,Crowley-Huh-Larson-Simpson-Wang22}.

As we see from the idempotency principle, the notion of (Grassmann--Pl\"ucker) polymatroid representations in this text differs from the concept of a linear representation, as a proper polymatroid has no Grassmann--Pl\"ucker representations over a field.

\subsection*{Thin Schubert cells and the universal tract}
The last two parts of the paper are dedicated to the study of the spaces of all polymatroid representations. We consider several variations of such spaces (weak and strong representation spaces, weak and strong thin Schubert cells, and realization spaces); for the purpose of this introduction, we first discuss our main results in the context of thin Schubert cells, and turn to the other spaces afterwards.

The \emph{representation space of $J$ over $F$} is the set $\upR_J(F)$ of all $F$-representations $\rho\colon\Delta^r_n\to F$ of $J$. The \emph{thin Schubert cell of $J$ over $F$} is the quotient $\Gr_J(F)=\upR_J(F)/F^\times$ by the diagonal action of $F^\times$ on $\upR_J(F)$. Both $\upR_J(F)$ and $\Gr_J(F)$ are functorial in $F$ (by composing an $F$-representation $\rho\colon\Delta^r_n\to F$ with a tract morphism $F\to F'$; cf.\ \autoref{subsection: functoriality}). The following is \autoref{prop: universal property of the universal tract}.

\begin{propA}\label{propD}
 Given a polymatroid $J$, the functor sending a tract $F$ to the thin Schubert cell $\Gr_J(F)$ is represented by a tract $T_J$, i.e.,\ there is a bijection $\Gr_J(F)\to\Hom(T_J,F)$ which is functorial in $F$. We call $T_J$ the \emph{universal tract of $J$}.
\end{propA}

The universal tract $T_J$ is given by a simple construction: up to taking degree $0$ elements, it is generated by symbols $x_\beta$ (for $\beta\in J$), and its null set is generated by Pl\"ucker relations for the $x_\beta$. This explicit description makes $T_J$ amenable to computations. On the other hand, the universal property of $T_J$ implicit in \autoref{propD} allows us to show that thin Schubert cells are functorial with respect to polymatroid embeddings (\autoref{thmE}), which we introduce in the following section.

\subsection*{Minors and polymatroid embeddings}

Let $\delta^-_J=\inf\, J$ and $\delta^+_J=\sup\, J$. The \emph{translation} of $J$ is $J+\tau=\{\alpha+\tau\mid\alpha\in J\}$, which is a polymatroid (and, in particular, contained in $\N^n$) provided that $\tau\geq-\delta^-_J$ (\autoref{lemma: translates and differences}). 

Let $\nu,\mu\in\N^n$ with $\mu+\delta^-_J\leq\alpha\leq\delta^+_J-\nu$ for some $\alpha\in J$. The \emph{deletion of $\nu$ in $J$} is
\[
 J\setminus\nu \ = \ \big\{ \alpha\in J \, \big| \, \alpha\leq\delta^+_J-\nu\}
\]
and the \emph{contraction of $\mu$ in $J$} is
\[
 J/\mu \ = \ \big\{ \alpha-\mu\in \N^n \, \big| \, \alpha\in J,\ \delta^-_J+\mu\leq\alpha\}.
\]
Both sets are polymatroids (\autoref{lemma: contraction and deletion}). 

An \emph{embedded minor of $J$} is a sequence of deletions, contractions, and translations. While translations commute with both deletions and contractions, the exchange of deletions and contractions is more subtle since these operations have an irregular effect on $\delta^-_J$ and $\delta^+_J$; cf.\ \autoref{prop: bounds on the duality vectors of minors}. Still, \autoref{prop: contraction and deletion commute} attests that for given $\nu$ and $\mu$, there are $\nu'$, $\mu'$ and $\tau'$ such that $(J/\mu)\setminus\nu=((J\setminus\nu')/\mu')+\tau'$, which allows us to represent every embedded minor as $J\minor\nu\mu+\tau=((J\setminus\nu)/\mu)+\tau$.

A \emph{minor embedding} is an inclusion of polymatroids of the form
\[
 \iota_{J\minor\nu\mu+\tau}\colon \ J\minor\nu\mu+\tau \ \longrightarrow \ J, \qquad \text{with} \qquad \iota_{J\minor\nu\mu+\tau}(\alpha) \ = \ \alpha+\mu-\tau.
\]

A permutation $\sigma\in S_n$ of $[n]$ induces a bijection $\iota_\sigma\colon\Delta^r_n\to \Delta^r_n$ by permuting the coordinates, which restricts to a bijection $\iota_\sigma\colon J\to\sigma(J)$ (called a \emph{permutation of variables}), and $\sigma(J)$ is a polymatroid (\autoref{lemma: permutation of variables}).

The embedding $\iota_n\colon\N^n\to\N^{n+1}$ into the first $n$ coordinates restricts to a bijection $\iota_n\colon J\to\iota_n(J)$ (called an \emph{extension of variables}), and $\iota_n(J)$ is a polymatroid (\autoref{lemma: extension of variables}). Its inverse bijection is called a \emph{restriction of variables}.

\begin{df*}
 A \emph{polymatroid embedding} is a composition of minor embeddings with permutations, extensions, and restrictions of variables. Two polymatroids $J$ and $J'$ are \emph{combinatorially equivalent} if there is a bijective polymatroid embedding $J\to J'$.
\end{df*}

Note that the inverse of a bijective polymatroid embedding is again a polymatroid embedding.
The following summarizes \autoref{thm: foundations of embedded minors}, \autoref{prop: foundations of coordinate inclusions}, and \autoref{prop: foundations of permutations}.

\begin{thmA}\label{thmE}
 Let $\iota\colon J\to J'$ be a polymatroid embedding and let $F$ be a tract. Precomposing an $F$-representation $\rho$ of $J'$ with $\iota$ yields a map $\iota^*\colon\Gr_{J'}(F)\to\Gr_J(F)$. If $\iota$ is bijective, then so is $\iota^\ast$.
\end{thmA}

As a consequence of \autoref{thmE}, it makes sense to talk about (embedded) minors of polymatroid representations, which are defined as $\rho\minor\nu\mu+\tau=\iota^\ast_{J\minor\nu\mu+\tau}(\rho)$.

\begin{rem*}
 Embedded minors in the sense of this text correspond to polymatroid truncations that appear in the work of Br\"and\'en--Huh (see \autoref{rem: matroid truncations} for details).
\end{rem*}

\subsection*{Duality}

Whittle has shown in \cite{Whittle92} that there is no involution on polymatroids which interchanges deletion and contraction. \autoref{thmE} suggests an alternative perspective: perhaps polymatroid duality should only interchange deletion and contraction up to combinatorial equivalence. It turns out that this paradigm leads to a satisfactory notion of duality.

\begin{df*}
 The \emph{duality vector of $J$} is $\delta_J=\delta^-_J+\delta^+_J=\inf\, J+\sup\, J$. The \emph{dual of $J$} is $J^\ast=\delta_J-J$.
\end{df*}

The following summarizes \autoref{lemma: dual}, \autoref{prop: duals of contractions and deletions}, and \autoref{thm: foundations of duals}.

\begin{thmA}\label{thmF}
 Polymatroid duality satisfies the following properties:
 \begin{enumerate}
  \item $(J^\ast)^\ast=J$;
  \item $(J\setminus\nu)^\ast$ and $J^\ast/\nu$ are combinatorially equivalent (by a translation);
  \item there is a canonical bijection $\Gr_J(F)\to\Gr_{J^\ast}(F)$ that is functorial in $J$ and $F$.
 \end{enumerate}
\end{thmA}

\subsection*{Direct sums}
The \emph{direct sum} $J_1\oplus J_2$ of two polymatroids $J_1\subseteq\Delta^{r_1}_{n_1}$ and $J_2\subseteq\Delta^{r_2}_{n_2}$ is a polymatroid contained in $\Delta^{r_1+r_2}_{n_1+n_2}$ (see \autoref{subsection: direct sums}). The following is \autoref{thm: representations of direct sums}.

\begin{thmA}\label{thmG}
 There is a canonical bijection $\Gr_{J_1\oplus J_2}(F)\to\Gr_{J_1}(F)\times\Gr_{J_2}(F)$ that is functorial in $F$.
\end{thmA}

A polymatroid $J\subseteq\Delta^r_n$ is \emph{nontrivial} if $n\geq1$, and \emph{indecomposable} if $J$ is nontrivial and $J$ is not the direct sum of two nontrivial polymatroids. Every polymatroid $J$ has a unique decomposition into a direct sum $\bigoplus_{i=1}^{c(J)} J_i$ of indecomposable polymatroids $J_i$, which are unique up to combinatorial equivalence and a permutation of indices (\autoref{prop: unique decompistion into indecomposable polymatroids}). 

\subsection*{The canonical torus embedding}
The association $\rho\mapsto(\rho(\alpha))_{\alpha\in J}$ defines a canonical embedding $\upR_J(F)\to (F^\times)^J$, which we show factors through a smaller subgroup $\upD_J(F)$ of $(F^\times)^J$ defined as follows.

We say that a Pl\"ucker relation 
\begin{equation*}
 \sum_{k=0}^s \quad \rho(\alpha -\epsilon_{i_k}+\epsilon_{i_0}+\dotsb+\epsilon_{i_s}) \cdot \rho(\alpha +\epsilon_{i_k}+\epsilon_{j_2}+\dotsb +\epsilon_{j_s}) \quad \in \quad  N_{F}
\end{equation*}
is \emph{degenerate} if it has exactly two nonzero terms, say 
\[
 \rho(\beta-\epsilon_{i_k})\rho(\gamma+\epsilon_{i_k}) \qquad \text{and} \qquad \rho(\beta-\epsilon_{i_\ell})\rho(\gamma+\epsilon_{i_\ell})
\]
for $k\neq\ell$, where $\beta=\alpha+\epsilon_{i_0}+\dotsb+\epsilon_{i_s}$ and $\gamma=\alpha+\epsilon_{j_2}+\dotsb+\epsilon_{j_s}$. By the uniqueness of additive inverses, and since we assume that $-1=1$, a degenerate Pl\"ucker relation corresponds to an equality of the form 
\[
 \rho(\beta-\epsilon_{i_k})\rho(\gamma+\epsilon_{i_k}) \ = \ \rho(\beta-\epsilon_{i_\ell})\rho(\gamma+\epsilon_{i_\ell}).
\]
The \emph{degeneracy locus of $J$ over $F$} is the subgroup
\[
 \upD_J(F) \ = \ \big\{ \rho\in (F^\times)^J \, \big| \, \rho\text{ satisfies all degenerate Pl\"ucker relations}\big\}
\]
of $(F^\times)^J$, and it contains the image of the embedding $\upR_J(F)\to(F^\times)^J$.

\subsection*{The Pl\"ucker embedding and the Polygrassmannian}
The map $[\rho]\mapsto[\rho(\alpha)]_{\alpha\in J}$ defines a canonical embedding $\Gr_J(F)\to(F^\times)^J/F^\times$, which can be considered as a stratum of the projective space $\P(F^{\Delta^r_n})=\big(F^{\Delta^r_n}-\{0\}\big)/F^\times$ (cf.\ \autoref{section: the Plucker embedding for thin Schubert cells} for details). Composing these two inclusions yields the Pl\"ucker embedding $\Gr_J(F)\to\P(F^{\Delta^r_n})$.

The union of the images of $\Gr_J(F)$ in $\P(F^{\Delta^r_n})$ for all polymatroids $J\subseteq\Delta^r_n$ defines the \emph{Polygrassmannian $\PolyGr(r,n)(F)$}, which is in general larger than the Grassmannian; see \autoref{subsection: the Polygrassmannian for idempotent fusion tracts} for details.

\subsection*{Weak thin Schubert cells and the universal pasture}

A nontrivial result in matroid theory is that for many tracts of interest, including fields, $\K$, and $\T_0$, the thin Schubert cell $\Gr_J(F)$ of a matroid $J$ is cut out by just the \emph{$3$-term Pl\"ucker relations}, which are of the form
\begin{multline*}
 \rho(\alpha+\epsilon_j+\epsilon_k) \cdot \rho(\alpha +\epsilon_i+\epsilon_l) \ - \ \rho(\alpha +\epsilon_i+\epsilon_k) \cdot \rho(\alpha +\epsilon_j+\epsilon_l) \\ + \ \rho(\alpha +\epsilon_i+\epsilon_j) \cdot \rho(\alpha +\epsilon_k+\epsilon_l) \quad \in \quad N_F
\end{multline*}
for $\alpha\in\Delta^{r-2}_n$ and $1\leq i\leq j\leq k\leq l\leq n$ with $\delta^-_J\leq\alpha$ and $\alpha+\epsilon_i+\epsilon_j+\epsilon_k+\epsilon_l\leq\delta_J^+$ (where we can ignore the sign if $F$ is near-idempotent).

Maps $\rho\colon\Delta^r_n\to F$ with support $J$ that satisfy the $3$-term Pl\"ucker relations are called \emph{weak $F$-representations of $J$}, a notion that makes sense for all polymatroids. The \emph{weak thin Schubert cell $\Gr^w_J(F)$} is the space of $F^\times$-classes of weak $F$-representations of $J$.

\begin{propA}\label{propH}
 Given a polymatroid $J$, the functor sending a tract $F$ to the weak thin Schubert cell $\Gr^w_J(F)$ is represented by a tract $P_J$, i.e.,\ there is a bijection $\Gr^w_J(F)\to\Hom(P_J,F)$ which is functorial in $F$. We call $P_J$ the \emph{universal pasture of $J$}. 
\end{propA}

Evidently, every (strong) $F$-representation is a weak $F$-representation, which yields an inclusion $\Gr_J(F)\subseteq\Gr^w_J(F)$, or, equivalently, a surjection $P_J\to T_J$. The following non-obvious result is \autoref{thm: bijection between the universal pasture and the universal tract}:

\begin{thmA}\label{thmI}
 The canonical morphism $P_J\to T_J$ is a bijection.
\end{thmA}

We deduce from this result that the functoriality of $\Gr_J(F)$ described in \autoref{thmE} also holds for $\Gr^w_J(F)$, and that $\Gr^w_J(F)$ is functorial with respect to duality (\autoref{thmF}) and direct sums (\autoref{thmG}). \autoref{thmI} also implies that $\Gr^w_J(F)$ is contained in the degeneracy locus $\upD_J(F)$, and shows that the Tutte group of a matroid $M$ (which coincides with $P_M^\times$) is canonically isomorphic to $T_M^\times$ (a result which was known by Wenzel \cite[p.\ 39]{Wenzel1996}).

\subsection*{Excellent tracts}
In general, the inclusion $\Gr_J(F)\subseteq\Gr^w_J(F)$ is not an equality. We call a tract $F$ \emph{excellent} if this inclusion is an equality for all polymatroids $J$.

Excellent tracts are closely related to perfect tracts $F$ (see \cite{Baker-Bowler19} for the definition), for which the equality $\Gr_J(F)=\Gr^w_J(F)$ holds whenever $J$ is a matroid  (see \cite[Thm.\ 3.46]{Baker-Bowler19}). Note that every field, the Krasner hyperfield $\K$, and the tropical hyperfield $\T$ are perfect.

A tract $F$ is \emph{degenerate} if $N_F$ contains every formal sum $\sum a_i$ with at least $3$ nonzero terms $a_1,a_2,a_3\in F^\times$. The following summarizes our present state of knowledge about excellent tracts (see \autoref{cor: perfect tracts are excellent} and \autoref{cor: degenerate tracts are excellent}).

\begin{thmA}\label{thmJ}
 \begin{enumerate}
 \item[]
 \item Every perfect tract is excellent. 
 \item Every degenerate tract is excellent. 
 \end{enumerate}
\end{thmA}

\subsection*{Realization spaces and the foundation}
The \emph{torus} $T(F)=(F^\times)^n$ acts on $\Gr^w_J(F)$ through the formula $(t.\rho)(\alpha)=(t_1^{\alpha_1}\dotsb t_n^{\alpha_n})\cdot\rho(\alpha)$. The \emph{realization space of $J$ over $F$} is the set $\ulineGr^w_J(F)$ of $T(F)$-orbits of this action.

\begin{propA}\label{propK}
 For every polymatroid $J$, the functor sending a tract $F$ to the realization space $\ulineGr^w_J(F)$ is represented by a tract $F_J$, i.e.,\ there is a bijection $\ulineGr^w_J(F)\to\Hom(F_J,F)$ which is functorial in $F$. We call $F_J$ the \emph{foundation of $J$}.  
\end{propA}

The foundation of a matroid has proven to be a valuable tool for the study of matroid representations (see, for example, \cite{Baker-Lorscheid20,Baker-Lorscheid21,Baker-Lorscheid-Zhang24}), in part because one knows an explicit presentation for it in terms of generators and relations, with the generators being certain canonical elements called \emph{cross ratios}.
This presentation is closely connected to Tutte's homotopy theory (\cite{Tutte58a,Tutte58b}; see also \cite{Baker-Jin-Lorscheid25} for a ``modern'' account). 

We succeed in generalizing this presentation to polymatroids (for the definition of cross ratios and a precise formulation, see \autoref{thm: generators and relations for the foundation}).

\begin{thmA}\label{thmL}
 The foundation $F_J$ of $J$ is generated by the cross ratios for $J$. All types of relations between cross ratios that hold for matroids also hold for all polymatroids, and such relations together with $1=-1$ if $J$ is proper form a complete list of relations between cross ratios for polymatroids.
\end{thmA}

The realization space $\ulineGr_J(F)$ satisfies further properties similar to the thin Schubert cells: it is functorial with respect to polymatroid embeddings (\autoref{thm: foundations of embedded minors}, \autoref{prop: foundations of coordinate inclusions}, \autoref{prop: foundations of permutations}), duality (\autoref{thm: foundations of duals}), and direct sums (\autoref{thm: representations of direct sums}).

\subsection*{The lineality space}
If $F$ is idempotent, the characteristic map $\chi_J\colon\Delta^r_n\to F$ of $J$ is an $F$-representation. The \emph{lineality space of $J$ over $F$} is the torus orbit $\upLin_J(F)=T(F).\chi_J$. The following is \autoref{thm: decomposition of the representation space}.

\begin{thmA}\label{thmM}
 Let $c(J)$ be the number of indecomposable components of $J\subseteq\Delta^r_n$. Then $P_J\simeq F_J(x_1,\dotsc,x_{n-c(J)})$ is a free algebra over $F_J$, and this isomorphism defines a bijection
 \[
  \Gr^w_J(F) \ \simeq \ \ulineGr^w_J(F) \times (F^\times)^{n-c(J)}
 \]
 which is functorial in $F$. If $F$ is idempotent, then $\upLin_J(F)\simeq(F^\times)^{n-c(J)}$.
\end{thmA}

\subsection*{Acknowledgements} 
We thank Ari Pomeranz, and Noah Solomon for helpful remarks on an early draft of this text. We thank the Korea Institute for Advanced Study in Seoul and the Institute for Advanced Study in Princeton for hosting us during our collaboration. M.B.\ was partially supported by NSF grant DMS-2154224 and a Simons Fellowship in Mathematics (1037306, Baker). J.H.\ is partially supported by the Oswald Veblen Fund and the Simons Investigator Grant. M.K.\ was partially supported by DFG grant 502861109. O.L.\ was partially supported by NSF grant DMS-1926686 and by the Institute for Advanced Study.

\part{Polymatroids: duality and embedded minors}
\section{Polymatroids and M-convex sets}
\label{section: polymatrois and M-convex sets}

In this section, we review the concepts of (discrete and integral) polymatroids and their relation to M-convex sets. 

We will consider the nonnegative orthant $\R_{\geq0}^n$ in $\R^n$, together with the partial order $\leq$ defined by $\alpha\leq\beta$ iff $\alpha_i\leq\beta_i$ for all $i=1,\dotsc,n$. We define $\norm\alpha=\alpha_1+\dotsb+\alpha_n$, and we let $\epsilon_i$ denote the $i$-th standard basis vector of $\R^n$. 

\subsection{Polymatroids}
\label{subsection: polymatroids}

Like matroids, polymatroids have several equivalent (``cryptomorphic'') characterizations, for example in terms of rank functions and polytopes. For the purposes of this paper, the following approach seems the most economical.

A \emph{polymatroid on $[n]$} is a nonempty compact subset $\cP$ of $\R_{\geq0}^n$ that satisfies the following two axioms:
\begin{enumerate}[label={(P\arabic*)}]
 \item\label{poly1} if $\beta\in\cP$ and $\alpha\leq\beta$, then $\alpha\in\cP$;
 \item\label{poly2} if $\alpha,\beta\in\cP$ and $\norm\alpha<\norm\beta$, then there exist $i\in[n]$ and $r\in[0,\ \beta_i-\alpha_i]$ such that $\alpha+r\epsilon_i\in\cP$.
\end{enumerate}
It follows from these axioms that $\cP$ is a convex polytope in $\R_{\geq0}^n$. 

An \emph{integral polymatroid on $[n]$} is a polymatroid $\cP$ on $[n]$ whose vertices have integer coordinates, i.e., $\cP$ is the convex hull $\conv(S)$ of a finite subset $S$ of $\N^n$.

A \emph{discrete polymatroid on $[n]$} is a nonempty finite subset $P$ of $\N^n$ that satisfies the following two axioms:
\begin{enumerate}[label={(DP\arabic*)}]
 \item\label{dpoly1} if $\beta\in P$ and $\alpha\leq\beta$, then $\alpha\in P$;
 \item\label{dpoly2} if $\alpha,\beta\in P$ and $\norm\alpha<\norm\beta$, then there exists $i\in[n]$ such that $\alpha_i<\beta_i$ and $\alpha+\epsilon_i\in P$.
\end{enumerate}

The elements of a discrete polymatroid $P$ are called its \emph{independent vectors}, and the maximal elements (w.r.t.\ $\leq$) are called its \emph{bases}. Axiom \ref{dpoly1} implies that a discrete polymatroid is determined by its set of bases, and axiom \ref{dpoly2} implies that $\norm \beta = \norm{\beta'}$ for any two bases $\beta,\beta'$ of $P$.

Taking the convex hull of a discrete polymatroid $P$, considered as a subset of $\R_{\geq0}^n$, yields an integral polymatroid $\cP=\conv(P)$. By \cite[Thm.\ 3.4]{Herzog-Hibi02}, we recover $P$ as $\cP\cap\N^n$. As a consequence of \cite[Theorem 4.15]{Murota03}, this establishes a bijective correspondence between discrete and integral polymatroids.

\subsection{M-convex sets}
\label{subsection: M-convex sets}

Let $\Delta^r_n=\{\alpha\in\N^n \mid \alpha_1+\dotsb+\alpha_n=r\}$. An \emph{M-convex set} of rank $r$ on $[n]$ is a nonempty subset $J$ of $\Delta^r_n$ such that for all $\alpha,\beta\in J$ and every $i\in[n]$ with $\alpha_i<\beta_i$, there exists an $j\in[n]$ with $\alpha_j > \beta_j$ such that $J$ contains both $\alpha+\epsilon_i-\epsilon_j$ and $\beta-\epsilon_i+\epsilon_j$.

It follows directly from the definition that for an M-convex set $J\subseteq\Delta^r_n$, the subset $P_J=\{\alpha\in\N^n\mid \alpha\leq\beta\text{ for some }\beta\in J\}$ of $\N^n$ is a discrete polymatroid. \emph{A priori,} the (symmetric) exchange axiom for M-convex sets seems stronger than the exchange axiom \ref{dpoly2} for discrete polymatroids, since the symmetric exchange axiom requires both $\alpha+\epsilon_i-\epsilon_j$ and $\beta-\epsilon_i+\epsilon_j$ to be in $J$. However, it follows from \cite[Thm.\ 2.7]{Vladoiu06} that every discrete polymatroid is of the form $P_J$ for some M-convex set $J$, which establishes a bijective correspondence between M-convex sets $J$ and discrete polymatroids $P_J$.

\begin{ex}\label{ex: matroids as M-convex sets}
 A matroid $M$ of rank $r$ on $[n]$ can be considered as the M-convex subset 
 \[\textstyle
  J \ = \ \{\sum_{i\in B} \epsilon_i\mid B\text{ is a basis of }M\}
 \]
 of $\Delta^r_n$. In this sense, we consider matroids as particular kinds of M-convex sets. For simplicity, we say that an M-convex set $J$ \emph{is a matroid} if it stems from a matroid in the above sense. Note that an M-convex set $J$ is a matroid if and only if $J\subseteq\{0,1\}^n$.
 
 Another class of examples of M-convex sets are the sets $J=\Delta^r_n$, which are not matroids for $r\geq2$. In particular, there are M-convex sets whose rank $r$ is bigger than $n$, which does not occur for matroids.
\end{ex}

\subsection{Rank functions}
\label{subsection: rank functions}

Let $\cP\subseteq\R_{\geq0}^n$ be a polymatroid and let $2^{[n]}$ denote the power set of $[n]$. The \emph{rank function of $\cP$} is the function $\br\colon 2^{[n]}\to\R_{\geq0}$ with values
\[\textstyle
 \br(S) \ = \ \max\, \{ \, \alpha_S \mid \alpha\in\cP \, \}, 
\]
for $S\subseteq [n]$, where $\alpha_S=\sum_{i\in S}\alpha_i$ (with $\alpha_\emptyset=0$). The polymatroid $\cP$ is characterized by $\br$ through the formula
\[\textstyle
 \cP \ = \ \big\{ \, \alpha\in\R_{\geq0} \ \big| \ \alpha_S\leq \br(S)\text{ for all }S\subseteq[n] \, \big\}.
\]

By \cite[Proposition 1.2]{Herzog-Hibi02}, a function $\br\colon2^{[n]}\to\R_{\geq0}$ is a rank function of a polymatroid if and only if it is \emph{normalized}, i.e., $\br(\emptyset)=0$, \emph{non-decreasing}, i.e., $\br(S)\leq \br(T)$ whenever $S\subseteq T$, and \emph{submodular}, i.e., $\br(S)+\br(T)\leq \br(S\cup T)+\br(S\cap T)$ for all $S,T\subseteq [n]$.

The polymatroid $\cP$ is integral if and only if the image of its rank function $\br\colon2^{[n]}\to\R_{\geq0}$ is contained in $\N$. If $\cP$ is the integral polymatroid corresponding to an M-convex set $J$, then its rank function is given by the formula 
\[\textstyle
 \br(S) \ = \ \max\, \{ \, \alpha_S\mid \alpha\in J \, \}
\]
for $S\subseteq[n]$.

\subsection{Base polytopes and generalized permutohedra}
\label{subsection: permutohedra}

An \emph{(integral) generalized permutohedron} is a polytope $\cB \subseteq \R^n$ such that all vertices of $\cB$ belong to $\Z^n$ and all edges of $\cB$ are parallel to $\epsilon_i - \epsilon_j$ for some $i \neq j$.

Generalized permutohedra are closely related to polymatroids. Given an integral polymatroid $\cP$ on $[n]$ with associated M-convex set $J$, its \emph{base polytope $\cB$} is defined as the convex hull of $J$, considered as elements of $\N^n \subseteq \R^n$. The base polytope $\cB$ characterizes the matroid since $J=\cB\cap\N^n$.

The following theorem is proved, for example, in \cite{Derksen-Fink10} and \cite{Postnikov-Reiner-Williams08}. It establishes a bijection (modulo translations) between discrete polymatroids and generalized permutohedra, and generalizes a well-known polytopal characterization of matroids due to Gelfand and Serganova \cite{Gelfand-Serganova87}.

\begin{thm*}
A polytope $\cB \subseteq \R^n$ is the base polytope of an integral polymatroid if and only if it is a generalized permutohedron and lies in the nonnegative orthant $\R_{\geq 0}^n$.
\end{thm*}

\subsection{Terminological convention for this paper}
\label{subsection: terminological convention}

Similar to the usage of the term ``matroid,'' which might refer to different characterizations---in terms of independent sets, bases, or a host of other quantities---we consider a polymatroid as an abstract mathematical object, which we typically describe in terms of its bases. Moreover, we assume from this point on, as a standing assumption whenever the term `polymatroid' appears:
\begin{center}\it \textbf{All polymatroids are discrete.}\end{center}

This means that we can describe a polymatroid in terms of its associated M-convex set $J$, which is our principal perspective on polymatroids.

\subsection{The natural matroid}
Let $J \subseteq \Delta^r_n$ be a polymatroid with rank function $\br\colon 2^{[n]} \to \N$. We let $E_1,\dots,E_n$ be disjoint sets with $\norm{E_i} = \br(i)$ and let $E := E_1 \cup \dots \cup E_n$. The projection $\theta\colon \Z^E \to \Z^n$ is given by $\theta(\epsilon_j) = \epsilon_i$ for $j \in E_i$. The \emph{natural matroid} of $J$ is $N_J := \theta^{-1} \left( J \right) \cap \{0,1\}^E$. The exchange axiom for $N_J$ follows immediately from that for $J$, and thus the natural matroid $N_J$ is indeed a matroid. 

\begin{ex}
  For $J = \Delta^2_2$, the natural matroid $N_J$ is the uniform matroid $U_{2,4}$.
\end{ex}

Natural matroids are quite useful for studying polymatroids. In particular, Crowley--Huh--Larson--Simpson--Wang \cite{Crowley-Huh-Larson-Simpson-Wang22} discovered connections between polymatroids and combinatorial Hodge theory by making use of natural matroids (which they refer to as \emph{minimal multisymmetric lifts}). In this paper, we utilize natural matroids to establish several useful properties of representations of polymatroids (see \autoref{subsection: the up operator} and \autoref{subsection: generators and relations for the foundation}). For basic properties of natural matroids, see~\cite{Bonin-Chun-Fife23} and the references therein.

\section{Duality and embedded minors}
\label{section: duality and minors}

In \cite{Whittle92}, Whittle introduces deletion and contraction operations for polymatroids, and discusses the existence and non-existence of a duality operation for polymatroids which interchanges these two operations, as is the case for matroids. The executive summary is that only polymatroids of a special shape allow for such a duality (cf.\ \autoref{rem: comparison with duality by Whittle}).

In this section, we bypass the limitations which Whittle encountered by introducing a duality operation which only interchanges deletion and contraction ``up to translation,'' leading to a more satisfactory theory. We also refine Whittle's notion of polymatroid minors.

\subsection{Duality and translation}
\label{subsection: duality and translates}

For the rest of this section, we fix an M-convex set $J\subseteq\Delta^r_n$. Let $\gamma = (\gamma_1,\ldots,\gamma_n) \in\Z^n$. We define $\norm\gamma=\gamma_1+\dotsb+\gamma_n$ and
\[
 J \ + \ \gamma \ = \ \{ \alpha+\gamma\mid \alpha\in J\}
 \qquad \text{and} \qquad 
 \gamma \ - \ J \ = \ \{ \gamma-\alpha\mid \alpha\in J\} 
\]
as subsets of $\Z^n$.

\begin{lemma}\label{lemma: translates and differences}
 If $J+\gamma\subseteq\N^n$ (resp.\ $\gamma-J\subseteq\N^n$), then $J+\gamma$ (resp.\ $\gamma-J$) is M-convex of rank $\norm\gamma+r$ (resp.\ $\norm\gamma-r$).
\end{lemma}

\begin{proof}
 Consider $\alpha+\gamma,\ \beta+\gamma\in J+\gamma$ with $\alpha,\beta\in J$ and $i\in[n]$ such that $(\alpha+\gamma)_i>(\beta+\gamma)_i$. Then $\alpha_i>\beta_i$, and by the M-convexity of $J$, there is a $j\in[n]$ such that $\alpha_j<\beta_j$ and $\alpha-\epsilon_i+\epsilon_j,\beta+\epsilon_i-\epsilon_j\in J$. Thus $(\alpha+\gamma)_j<(\beta+\gamma)_j$ and $\alpha+\gamma-\epsilon_i+\epsilon_j,\beta+\gamma+\epsilon_i-\epsilon_j\in J+\gamma$, which shows that $J+\gamma$ is M-convex. It is clear that the rank of $J+\gamma$ is $\norm{\gamma}+r$.
 
 The claim for $\gamma-J$ is proven by the same argument, but with the appropriate signs reversed.
\end{proof}

We consider the partial order on $\Z^n$ defined by $\alpha\leq\beta$ iff $\alpha_i\leq\beta_i$ for all $i\in[n]$. Note that every finite subset $S$ of $\Z^n$ has a greatest lower bound $\delta_S^-=\inf\, S$ and a least upper bound $\delta_S^+=\sup\, S$, whose respective coefficients are given by
\[
 \delta^-_{S,i} \ = \ \min\{\alpha_i\mid\alpha\in S\} \qquad \text{and} \qquad \delta^+_{S,i} \ = \ \max\{\alpha_i\mid\alpha\in S\}.
\]

\begin{df}
 The \emph{duality vector of $J$} is $\delta_J=\delta_J^-+\delta_J^+$. The \emph{dual of $J$} is $J^\ast=\delta_J-J$.
\end{df}

Before we discuss examples (see \autoref{rem: duals of loops} and \autoref{ex: duals} at the end of this section), we discuss several properties of polymatroid duality and compare it to matroid duality.

\begin{lemma}\label{lemma: dual}
 The dual $J^\ast$ of $J$ is M-convex of rank $\norm{\delta_J}-r$ with duality vector $\delta_{J^\ast}=\delta_J$ and dual $J^{\ast\ast}=J$.
\end{lemma}

\begin{proof}
 By \autoref{lemma: translates and differences}, $J^\ast$ is M-convex of rank $\norm{\delta_J}-r$. The equality $\delta_{J^\ast}=\delta_J$ follows from
 \[
  \delta^-_{J^\ast,i} \ = \ \min\{ \delta_{J,i}-\alpha_i\mid \alpha\in J\} \ = \ \delta^-_{J,i} \ + \ \delta^+_{J,i} \ - \ \underbrace{\max\{\alpha_i\mid \alpha\in J\}}_{=\delta^+_{J,i}} \ = \ \delta^-_{J,i}
 \]
 and
 \[
  \delta^+_{J^\ast,i} \ = \ \max\{ \delta_{J,i}-\alpha_i\mid \alpha\in J\} \ = \ \delta^-_{J,i} \ + \ \delta^+_{J,i} \ - \ \underbrace{\min\{\alpha_i\mid \alpha\in J\}}_{=\delta^-_{J,i}} \ = \ \delta^+_{J,i},
 \]
 and $J^{\ast\ast}=J$ follows from the equality $\delta_{J^\ast}-(\delta_J-\alpha)=\alpha$ for $\alpha\in J$.
\end{proof}

\begin{lemma}\label{lemma: dual of a translate}
 Let $\gamma\in\Z^n$ such that $J+\gamma\subseteq\N^n$. Then $\delta_{J+\gamma}=\delta_J+2\gamma$ and $(J+\gamma)^\ast=J^\ast+\gamma$.
\end{lemma}

\begin{proof}
 The first claim follows from 
 \[
  \delta_{J+\gamma} \ = \ \delta_{J+\gamma}^- \ + \ \delta_{J+\gamma}^+ \ = \ \delta_J^- \ + \ \gamma \ + \ \delta_J^+ \ + \ \gamma \ = \ \delta_J \ + \ 2\gamma,
 \]
 and the second claim follows from
 \[
  \delta_{J+\gamma} \ - \ (\alpha+\gamma) \ = \ \delta_J \ + \ 2\gamma \ - \ \alpha \ - \ \gamma \ = \ (\delta_J-\alpha) \ + \ \gamma 
 \]
 for $\alpha\in J$, together with the observation that $\delta_{J+\gamma}-(\alpha+\gamma)\in (J+\gamma)^\ast$ and $\delta_J-\alpha\in J^\ast$.
\end{proof}

We say that an element $i\in[n]$ is \emph{isolated in $J$} if $\delta^-_{J,i}=\delta^+_{J,i}$, and that $J$ is \emph{without isolated elements} if no element of $[n]$ is isolated in $J$.

\begin{lemma}\label{lemma: isolated elements}
 Let $J$ be a matroid. Then $i\in[n]$ is isolated in $J$ if and only if $i$ is a loop or a coloop.
\end{lemma}

\begin{proof}
 If $i$ is a loop, then $\delta^-_{J,i}=\delta^+_{J,i}=0$. If $i$ is a coloop, then $\delta^-_{J,i}=\delta^+_{J,i}=1$. In both cases, $i$ is isolated in $J$. If $i$ is not a loop nor a coloop, then there are $\alpha,\beta\in J$ such that $i\in\alpha$ and $i\notin\beta$. Thus $\delta^-_{J,i}=0$ and $\delta^+_{J,i}=1$, which shows that $i$ is not isolated in $J$.
\end{proof}

\begin{prop}\label{prop: matroid duality as polymatroid duality}
 If $J$ is a matroid without isolated elements, then the matroid dual of $J$ is equal to $J^\ast$.
\end{prop}

\begin{proof}
 Since $J$ is without isolated elements, we have $\delta^-_{J,i}=0$ and $\delta^-_{J,i}=1$ for all $i\in[n]$. Thus $\delta_{J,i}=1$ for all $i$ and $(\delta_J-\alpha)_i=1-\alpha_i$. Therefore the support of $\delta_J-\alpha$ is the complement of the support of $\alpha$ (as subsets of $[n]$), which agrees with the matroid dual of $J$.
\end{proof}

\begin{rem}\label{rem: duals of loops} 
 There is a discrepancy between matroid duality and polymatroid duality in the presence of isolated elements. The prototypical example is $J=\Delta^r_1=\{r\epsilon_1\}$, which has duality vector $\delta_J=\delta^-_J+\delta^+_J=(r)+(r)=(2r)$ and dual $J^\ast=\{2r\epsilon_1-r\epsilon_1\}=\{r\epsilon_1\}=J$. In particular, the loop $U_{0,1}=\Delta^0_1$ and the coloop $U_{1,1}=\Delta^1_1$ are self-dual as polymatroids, in contrast to matroid duality, which interchanges these two matroids. 
 In any case, the matroid dual and polymatroid dual differ only by a translation.
\end{rem}

\begin{ex}\label{ex: duals}
 For $J=\Delta^r_2$, we have $\delta^-_J=0$ and $\delta_J=\delta^+_J=r\epsilon_1+r\epsilon_2$. So $J^\ast=\Delta^r_2$ is self-dual, just as in the case of $\Delta^r_1$. The situation is different for $n\geq3$. If $J=\Delta^2_3$, then $\delta^-_J=0$ and $\delta_J=\delta^+_J=2\epsilon_1+2\epsilon_2+2\epsilon_3$. Thus 
 \[
  J^\ast \ = \ \big\{2\epsilon_i+2\epsilon_j,\ 2\epsilon_i+\epsilon_j+\epsilon_k \ \big| \ \{i,j,k\}=\{1,2,3\}\big\}.
 \]
\end{ex}

\subsection{Embedded minors}
\label{subsection: embedded minors}

An \emph{embedded minor} of a matroid $M$ is a matroid of the form $M \minor JI$ with $I$ independent and $J$ coindependent, together with the data of $I$ and $J$. In this section, we extend this concept to polymatroids. 

Recall that $\delta_J=\delta^-_J+\delta^+_J$ with $\delta^-_J=\inf\,J$ and $\delta^+_J=\sup\,J$.

\begin{df}\label{df:reduction}
The \emph{reduction of $J$} is the M-convex set $\Jbar:=J-\delta^-_J$.
\end{df}

\begin{df}\label{df:independent}
 Let $\mu\in\N^n$. We say that $\mu$ is \emph{effectively independent in $J$} if $\mu\leq\alpha-\delta^-_J$ for some $\alpha\in J$. We say that $\nu$ is \emph{effectively coindependent in $J$} if $\nu\leq\delta^+_J-\alpha$ for some $\alpha\in J$.

 Let $\mu$ be effectively independent in $J$ and let $\nu$ be effectively coindependent in $J$. The \emph{contraction of $\mu$ in $J$} is
 \[
  J/\mu \ = \ \{\alpha-\mu \mid \alpha\in J\text{ and }\mu\leq\alpha-\delta^-_J\},
 \]
 and the \emph{deletion of $\nu$ in $J$} is
 \[
  J\setminus\nu \ = \ \{\alpha\mid \alpha\in J\text{ and }\nu\leq\delta^+_J-\alpha\}.
 \]
\end{df}

Note that $\mu$ is effectively independent in $J$ if and only if $\mu-\delta^-_J$ is independent in $\Jbar:=J-\delta^-_J$, and thus it differs from the usual notion of independence for polymatroids (cf.~\autoref{subsection: polymatroids}). We won't use the latter meaning of independence in this paper, however, so we will frequently omit the modifier ``effectively'' in what follows, for ease of terminology.

Contractions and deletions come with injections 
\[
 \begin{array}{cccl}
  \iota_{J/\mu}\colon & J/\mu  & \longrightarrow & J \\
                 & \alpha & \longmapsto     & \alpha+\mu
 \end{array}
 \hspace{30pt} \text{and} \hspace{30pt}
 \begin{array}{cccc}
  \iota_{J\setminus \nu}\colon & J\setminus \nu  & \longrightarrow & J \\
                          & \alpha          & \longmapsto     & \alpha
 \end{array}
\]
into $J$.

\begin{ex}\label{ex: first example of contraction and deletion}
 Consider the M-convex set $J=\Delta^3_3\setminus\{(0,3,0)\}$, which has $\delta^-_J=0$ and $\delta^+_J=(3,2,3)$. The contractions and deletions of $\epsilon_3$ and $2\epsilon_3$ are illustrated in \autoref{fig: example of contraction and deletions}, where we write $ijk$ for the tuple $(i,j,k)\in J$. More examples can be found in \autoref{ex: contractions deletions and duality vectors}.
 \begin{figure}[ht]
  \includegraphics{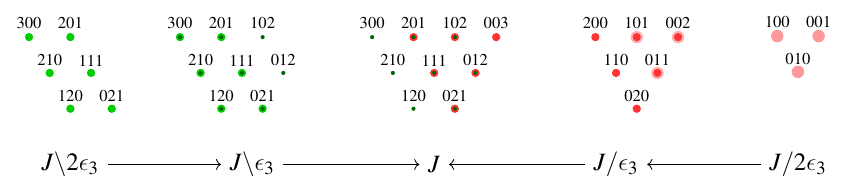}
  \caption{Some contractions and deletions of $J=\Delta^3_3\setminus\{(0,3,0)\}$}
  \label{fig: example of contraction and deletions}
 \end{figure}
\end{ex}

\begin{lemma}\label{lemma: contraction and deletion}
 Both $J/\mu$ and $J\setminus\nu$ are M-convex.
\end{lemma}

\begin{proof}
 Consider $\alpha-\mu,\beta-\mu\in J/\mu$ and $i\in[n]$ with $(\alpha-\mu)_i>(\beta-\mu)_i$. Then $\alpha,\beta\in J$ with $\mu+\delta^-_J\leq\alpha,\beta$ and $\alpha_i>\beta_i$. Since $J$ is M-convex, there is a $j\in[n]$ such that $\alpha_j<\beta_j$ and $\alpha-\epsilon_i+\epsilon_j,\ \beta+\epsilon_i-\epsilon_j\in J$. Thus also $(\alpha-\mu)_j<(\beta-\mu)_j$. Since
 \begin{align*}
  \mu_i+\delta^-_{J,i} \ \leq \ \beta_i  \ &\leq \ \alpha_i-1 \ = \ (\alpha-\epsilon_i+\epsilon_j)_i,    & \mu_i+\delta^-_{J,i} \ &\leq \ \beta_i \ < \ (\beta+\epsilon_i-\epsilon_j)_i, \\
  \mu_j+\delta^-_{J,j} \ \leq \ \alpha_j \ &\leq \ \beta_j-1  \ = \ (\beta+\epsilon_i-\epsilon_j)_j,     & \mu_j+\delta^-_{J,j} \ &\leq \ \alpha_j \ < \ (\alpha-\epsilon_i+\epsilon_j)_j, 
 \end{align*}
 we have $\alpha-\mu-\epsilon_i+\epsilon_j,\ \alpha-\mu+\epsilon_i-\epsilon_j\in J/\mu$, which shows that $J/\mu$ is M-convex.
 
 Consider $\alpha,\beta\in J\setminus\nu$ and $i\in[n]$ with $\alpha_i<\beta_i$. Since $J$ is M-convex, there is a $j\in[n]$ such that $\alpha_j>\beta_j$ and $\alpha-\epsilon_i+\epsilon_j,\ \beta+\epsilon_i-\epsilon_j\in J$. Since
 \[
  (\beta+\nu-\epsilon_i+\epsilon_j)_i \ = \ \beta_i+\nu_i+1 \ \leq \ \alpha_i+\nu_i \ \leq \ \delta^+_{J,i}, 
 \]
 \[
  (\alpha+\nu-\epsilon_i+\epsilon_j)_j \ = \ \alpha_j+\nu_j+1 \ \leq \ \beta_j+\nu_j \ \leq \ \delta^+_{J,j}, 
 \]
 \[
  (\beta+\nu+\epsilon_i-\epsilon_j)_j \ < \ \beta_j+\nu_j \ \leq \ \delta^+_{J,j},
  \quad \text{and} \quad 
  (\alpha+\nu+\epsilon_i-\epsilon_j)_i \ < \ \alpha_i+\nu_i \ \leq \ \delta^+_{J,i},
 \]
 we have $\alpha+\nu+\epsilon_i-\epsilon_j,\ \beta+\nu+\epsilon_i-\epsilon_j\in J\setminus\nu$, which shows that $M\setminus\nu$ is M-convex.
\end{proof}

\begin{lemma}\label{lemma: translations commute with contractions and deletions}
 Let $\mu$ be independent and let $\nu$ coindepdendent in $J$, and let $\tau\in\Z^n$ be such that $\tau\geq-\delta^-_J$. Then 
 \[
  (J+\tau)/\mu \ = \ (J/\mu)+\tau \qquad \text{and} \qquad (J+\tau)\setminus\nu \ = \ (J\setminus\nu)+\tau,
 \]
 with all sets M-convex.
\end{lemma}

\begin{proof}
 As subsets of $\Z^n$, we have
 \begin{multline*}
  (J+\tau)/\mu \ = \ \{\alpha+\tau-\mu \mid \alpha+\tau\in J+\tau,\ \mu+\delta^-_{J+\tau}\leq\alpha+\tau\} \\
  = \ \{\alpha-\mu+\tau \mid \alpha\in J,\ \mu+\delta^-_{J}\leq\alpha\} \ = \ (J/\mu)+\tau
 \end{multline*}
 and
 \begin{multline*}
  (J+\tau)\setminus\nu \ = \ \{\alpha+\tau \mid \alpha+\tau\in J+\tau,\ \alpha+\tau\leq\delta^+_{J+\tau}-\nu\} \\
  = \ \{\alpha+\tau \mid \alpha\in J,\ \alpha\leq\delta^+_J-\nu\} \ = \ (J\setminus\nu)+\tau.
 \end{multline*}
 Since $\tau\geq-\delta^-_J$, all of these set are contained in $\N^n$ and are thus M-convex.
\end{proof}

The following result shows that contractions and deletions behave well --- up to translation --- with respect to polymatroid duality. Note that $\mu$ is independent in $J$ if and only if it is coindependent in $J^\ast$, and vice versa.

\begin{prop}\label{prop: duals of contractions and deletions}
 Let $\mu$ be independent and let $\nu$ coindependent in $J$. Then 
 \[
  (J/\mu)^\ast \ = \ J^\ast\setminus\mu \ + \ (\delta_{J/\mu}+\mu-\delta_J) \quad \text{and} \quad (J\setminus\nu)^\ast \ = \ J^\ast/\nu \ + \ (\delta_{J\setminus\nu}+\nu-\delta_J).
 \]
\end{prop}

\begin{proof}
 These equalities follow from the direct computations
 \begin{align*}
  (J/\mu)^\ast \ &= \ \big\{ \delta_{J/\mu}-(\alpha-\mu) \ \big| \ \alpha\in J, \ \mu\leq\alpha-\delta^-_J \big\} \\
                 &= \ \big\{\delta_J-\alpha \ \big| \ \alpha\in J,\ \mu\leq\delta_J^+-(\delta_J-\alpha)\big\} \ + \ (\delta_{J/\mu}+\mu-\delta_J) \\
                 &= \ J^\ast\setminus\mu \ + \ (\delta_{J/\mu}+\mu-\delta_J) 
 \end{align*}
 and 
 \begin{align*}
  (J\setminus\nu)^\ast \ &= \ \big\{ \delta_{J\setminus\nu}-\alpha \ \big| \ \alpha\in J, \ \nu\leq\delta^+_J-\alpha \big\} \\
                         &= \ \big\{\delta_J-\alpha-\nu \ \big| \ \alpha\in J,\ \nu\leq(\delta_J-\alpha)-\delta^-_J\big\} \ + \ (\delta_{J\setminus\nu}+\nu-\delta_J) \\
                         &= \ J^\ast/\nu \ + \ (\delta_{J\setminus\nu}+\nu-\delta_J). \qedhere
 \end{align*}
\end{proof}

Note that if $\delta_{J/\mu}+\mu=\delta_J$ and $\delta_{J\setminus\nu}+\nu=\delta_J$, then the equalities from \autoref{prop: duals of contractions and deletions} simplify to $(J/\mu)^\ast=J^\ast\setminus\mu$ and $(J\setminus\nu)^\ast=J^\ast/\nu$, which resemble the analogous formulas from matroid theory.
In general, we only have the following bounds on the duality vectors of minors. \autoref{ex: contractions deletions and duality vectors} shows that these bounds are attained.

Let $\1\in\N^n$ be the all-ones vector. 
Recall that $\norm\mu=\mu_1+\dotsb+\mu_n$.

\begin{prop}\label{prop: bounds on the duality vectors of minors}
 Let $\mu=\mu_1+\mu_2$ be independent and let $\nu=\nu_1+\nu_2$ be coindependent in $J$. Then $\mu_2$ is independent in $J/\mu_1$, $\nu_2$ is coindependent in $J\setminus\nu_1$, and
 \[
  J/\mu \ = \ (J/\mu_1)/\mu_2 \qquad \text{and} \qquad J\setminus\nu \ = \ (J\setminus\nu_1)\setminus\nu_2.
 \]
 Moreover,
 \begin{align*}
  \delta^-_{J/\mu} \ &= \ \delta^-_J,             & 0 \quad &\leq \quad \delta^+_J-(\delta^+_{J/\mu}+\mu)          \hspace{-29.5pt} && \leq \quad \norm{\mu}\cdot\1 \ - \ \mu, \\ 
  \delta^+_{J\setminus\nu} \ &= \ \delta^+_J-\nu, & 0 \quad &\leq \quad \delta^-_{J\setminus \nu} \ - \ \delta^-_J \hspace{-30pt} && \leq \quad \norm{\nu}\cdot\1  \ - \ \nu,
 \end{align*}
 and thus
 \begin{align*}
  0 \quad \leq \quad \delta_J \ - \ \big( \delta_{J/\mu} \ + \ \mu \big)             \quad &\leq \quad \norm{\mu}\cdot\1 \ - \ \mu ,\\
  0 \quad \leq \quad \big( \delta_{J\setminus \nu} \ + \ \nu \big) \ - \ \delta_J \, \quad &\leq \quad \norm{\nu}\cdot\1  \ - \ \nu.
 \end{align*}
\end{prop}

 \begin{proof}
 We first treat the case of contractions, and then derive the results for deletions by duality. We begin with the bounds for $\delta^\pm_{J/\epsilon_i}$ for $i\in[n]$. 
 Fix $\alpha\in J$ with $\alpha-\epsilon_i\in J/\epsilon_i$, i.e., $\alpha_i\geq\delta^-_{J,i}+1$. 
 
 By \autoref{lemma: translations commute with contractions and deletions}, $(J-\delta^-_J)/\epsilon_i=(J/\epsilon_i)-\delta^-_J$, which allows us to assume that $\delta^-_J=0$ for simplicity. Then $\delta^-_{J/\epsilon_i}\geq0=\delta^-_J$ is evident. In order to establish the reverse inequality, consider $j\in[n]$ and $\beta\in J$ with $\beta_j=\delta^-_{J,j}$. If $\beta_i=\delta^-_{J,i}<\alpha_i$, then there is a $k\in[n]$ with $\beta_k>\alpha_k$ (and thus $k\neq i,j$) and $\beta'=\beta+\epsilon_i-\epsilon_k\in J$. Thus we can assume that $\beta_j=\delta^-_{J,j}$ and $\beta_i\geq\delta^-_{J,i}+1$, i.e., $\beta-\epsilon_i\in J/\epsilon_i$. Therefore $\delta^-_{J/\epsilon_i,j}\leq(\beta-\epsilon_i)_j\leq\beta_j=\delta^-_{J,j}$, as desired.  
 
 Since $J/\epsilon_i\subseteq J-\epsilon_i$, we have $\delta^+_{J/\epsilon_i}+\epsilon_i\leq \delta^+_J$ and thus $0 \leq\delta^+_J-(\delta^+_{J/\epsilon_i}+\epsilon_i)$. In order to establish the upper bound, choose $j\in[n]$ and $\beta\in J$ such that $\beta_j=\delta^+_{J,j}$. If $\beta_i=\delta^-_{J,i}<\alpha_i$, the M-convexity of $J$ implies that there is a $k\in[n]$ with $\beta_k>\alpha_k$ (and thus $k\neq i$, but possibly $k=j$) and $\beta'=\beta+\epsilon_i-\epsilon_k\in J$. Thus we can assume that $\beta_j\geq\delta^+_{J,j}-1$ (and $\beta_j=\delta^+_{J,j}$ if $j=i$) and $\beta_i\geq\delta^-_{J,i}+1$, i.e., $\beta-\epsilon_i\in J/\epsilon_i$. Therefore $\delta^+_{J,j}-1\leq\beta_j=(\beta-\epsilon_i)_j\leq\delta^+_{J/\epsilon_i,j}$ for $j\neq i$, and $\delta^+_{J,i}-1=\beta_i-1=(\beta-\epsilon_i)_i\leq\delta^+_{J/\epsilon_i,i}$, which establishes the desired upper bound $\delta^+_J-(\delta^+_{J/\epsilon_i}+\epsilon_i)\leq\1-\epsilon_i$. 

 Next we establish the equality $J/\mu=(J/\mu_1)/\mu_2$ for $\mu_1=\epsilon_i$, by the following direct computation:
 \begin{multline*}
  (J/\mu_1)/\mu_2 \ = \ \big\{ (\alpha-\mu_1)-\mu_2 \, \big| \, \alpha-\mu_1\in J/\mu_1,\ (\alpha-\mu_1)-\mu_2 \geq\delta^-_{J/\mu_1} \big\} \\
  = \ \big\{ \alpha-(\mu_1+\mu_2) \, \big| \, \alpha\in J,\ \alpha-(\mu_1+\mu_2) \geq\delta^-_{J} \big\} \ = \ J/\mu,
 \end{multline*}
 where we use that $\delta^-_{J/\mu_1}=\delta^-_J$ for $\mu_1=\epsilon_i$. Note that since $\mu$ is independent in $J$, there exists $\alpha\in J$ with $(\alpha-\mu_1)-\mu_2=\alpha-\mu\geq\delta^-_J=\delta^-_{J/\mu_1}$; thus, in particular, $\alpha-\mu_1\in J/{\mu_1}$ and $\mu_2$ is independent in $J/\mu_1$, as claimed.
 
 As a consequence, an induction over $s=\norm{\mu}$ for $\mu=\epsilon_{i_1}+\dotsc+\epsilon_{i_s}$ shows that $\delta^-_{J/\mu}=\delta^-_{J/(\mu-\epsilon_{i_s})}=\dotsc=\delta^-_J$ as well as $0\leq\delta^+_J-(\delta^+_{J/\mu})+\mu\leq \norm{\mu}\cdot\1 -\mu$. This in turn makes the above verification of $J/\mu=(J/\mu_1)/\mu_2$ valid for all $\mu_1$. This establishes all claims of the proposition for contractions.
 
 We proceed with deletions, and begin with the bounds for $(\delta_{J\setminus\nu}+\nu)-\delta_J$. By \autoref{prop: duals of contractions and deletions}, we have
 \[
  J\setminus\nu \ = \ (J^\ast/\nu)^\ast \ + \ (\delta_J-\delta_{J^\ast/\nu}-\nu),
 \]
 and by \autoref{lemma: dual of a translate} ($\delta_{J\setminus\nu}=\delta_{(J^\ast/\nu)^\ast}+2\delta_J-2\delta_{J^\ast/\nu}-2\nu$) and \autoref{lemma: dual} ($\delta_J=\delta_{J^\ast}$), we have
 \[
  \big( \delta_{J\setminus\nu} + \nu \big) - \delta_J \ = \ \delta_{(J^\ast\setminus\nu)^\ast} + 2\delta_J - 2\delta_{J^\ast/\nu} - 2\nu + \nu - \delta_J \ = \ \delta_{J^\ast} - \big( \delta_{J^\ast/\nu} + \nu \big).
 \]
 Applying the bounds for the contraction to $\delta_{J^\ast}-(\delta_{J^\ast/\nu}+\nu)$ yields the desired bounds for $(\delta_{J\setminus\nu}+\nu)-\delta_J$. Together with the trivial inequalities $0\leq\delta^-_{J\setminus\nu}-\delta^-_J$ and $\delta^+_{J\setminus\nu}\leq\delta^+_J-\nu$, this implies the other two desired bounds for $\delta^-_{J\setminus\nu}$ and $\delta^+_{J\setminus\nu}$.
 
 Finally, the equality $J\setminus\nu=(J\setminus\nu_1)\setminus\nu_2$, can also be deduced by duality from the previous results, using \autoref{prop: duals of contractions and deletions}:
 \begin{align*}
  \big((J\setminus\nu_1)\setminus\nu_2\big)^\ast \ &= \ (J\setminus\nu_1)^\ast/\nu_2 \ + \ \big( \delta_{(J\setminus\nu_1)\setminus\nu_2}+\nu_2-\delta_{J\setminus\nu_1}\big) \\
  &= \ (J/\nu_1)/\nu_2 \ + \ \big(\delta_{J\setminus\nu_1}+\nu_1-\delta_J\big) \ + \ \big( \delta_{(J\setminus\nu_1)\setminus\nu_2}+\nu_2-\delta_{J\setminus\nu_1}\big) \\
  &= \ J/\nu \ + \ \big( \delta_{J\setminus\nu}+\nu-\delta_{J}\big) \\
  &= \ (J\setminus\nu)^\ast. \qedhere
 \end{align*}
\end{proof}

\begin{ex}\label{ex: contractions deletions and duality vectors}
 The only independent vector and the only coindependent vector for $\Delta^r_1$ is $0$. Thus $\Delta^r_1$ does not have any proper minors.

 For $J=\Delta^r_n$ with $n\geq2$ and $r\geq1$, we have $\delta^-_J=0$ and $\delta_J=\delta^+_J=r\epsilon_1+\dotsb+r\epsilon_n$. For $\mu=s\epsilon_1$ with $0\leq s\leq r$, we find 
 \[
  \Delta^r_n/s\epsilon_1 \ = \ \Delta^{r-s}_n \qquad \text{and} \qquad \Delta^r_n \setminus s\epsilon_1 \ = \ \Delta^r_n \ \setminus \ \{(r-s+1)\epsilon_1+\alpha \mid \alpha\in\Delta^{s-1}_n\}.
 \]
 In the case of the contraction, the difference of the dimension vectors assumes the extremal value $\delta_J-(\delta_{J/\mu}+\mu)=\norm{\mu^+}\cdot\1-\mu$ from \autoref{prop: bounds on the duality vectors of minors}. In the case of the deletion, we have $\delta^-_{\Delta^r_n\setminus s\epsilon_1}=s\epsilon_2$ if $n=2$ and $\delta^-_{\Delta^r_n\setminus s\epsilon_1}=0$ for $n\geq3$. Thus we have
 \[
  \delta_{J\setminus\mu} \ = \ \begin{cases} (r-s)\epsilon_1+(r+s)\epsilon_2 & \text{if }n=2, \\ (r-s)\epsilon_1+r\epsilon_2+\dotsc+r\epsilon_n & \text{if }n\geq3, \end{cases} 
 \]
 and the difference of the dimension vectors is
 \[
  (\delta_{J\setminus\mu} \ + \ \mu) \ - \ \delta_J \ = \ \begin{cases}
                                                           \norm{\mu}\cdot\1 - \mu & \text{if }n=2, \\ 
                                                           0 & \text{if }n\geq3,
                                                          \end{cases}
 \]
 which assumes the different extremal values from \autoref{prop: bounds on the duality vectors of minors}, depending on the value of $n$.
 
 An example in which the extremal value $\delta_J-\delta_{J/\mu}-\mu=0$ appears is the M-convex set
 \[
  J \ = U^+_{2,3} := \ \big\{2\epsilon_1,\ \epsilon_1+\epsilon_2,\ \epsilon_1+\epsilon_3,\ \epsilon_2+\epsilon_3 \big\}
 \] 
 in $\Delta^2_3$. Here we have $\delta^-_J=0$ and $\delta_J=2\epsilon_1+\epsilon_2+\epsilon_3$. For $\mu=\epsilon_1$, we find $J/\epsilon_1=\Delta^1_3$ and $\delta_{J/\epsilon_1}=\epsilon_1+\epsilon_2+\epsilon_3$. Thus $\delta_J-(\delta_{J/\mu}+\mu)=0$.
\end{ex}

The exchange of order of contractions and deletions involves a translation, as the following result shows.

\begin{prop}\label{prop: contraction and deletion commute}
 Let $\mu,\nu\in\N^n$. Then the following are equivalent:
 \begin{enumerate}
  \item \label{minor1} $\delta^-_J+\mu\leq\alpha\leq\delta^+_J-\nu$ for some $\alpha\in J$;
  \item \label{minor2} $\mu$ is independent in $J$ and $\nu'=\sup\{\0,\ \nu-\delta^+_J+\delta^+_{J/\mu}+\mu\}$ is coindependent in $J/\mu$; 
  \item \label{minor3} $\nu$ is coindependent in $J$ and $\mu'=\sup\{\0,\ \delta^-_J-\delta^-_{J\setminus \nu}+\mu\}$ is independent in $J\setminus\nu$.
 \end{enumerate}
 If these conditions are met, then
 \[
  (J\setminus \nu)/\mu' \ = \ (J/\mu)\setminus\nu' \ + \ \sup\{-\mu,\ \ \delta^-_{J\setminus \nu} - \delta^-_J\}.
 \]
\end{prop}

\begin{proof}
 We begin with the equivalence of \eqref{minor1} and \eqref{minor2}. First note that $\mu+\delta^-_J\leq\alpha$ for $\alpha\in J$ means that $\mu$ is independent in $J$ and that $\beta=\alpha-\mu$ is in $J/\mu$. Thus \eqref{minor1} is equivalent to $\mu$ being independent in $J$, together with the existence of a $\beta\in J/\mu$ such that
 \[
  \beta \ \leq \ \delta^+_J \ - \ \nu \ - \ \mu \ = \ \delta^+_{J/\mu} \ - \ \big(\nu \ + \ \delta^+_{J/\mu} \ - \ \delta^+_J \ + \ \mu\big).
 \]
 Since $\beta\leq\delta^+_{J/\mu}-0$, this latter inequality turns into $\beta\leq\delta^+_{J/\mu}-\nu'$, which means that $\nu'$ is coindependent in $J/\mu$. This establishes the equivalence between \eqref{minor1} and \eqref{minor2}.
 
 We continue with the equivalence of \eqref{minor1} and \eqref{minor3}, which is proven similarly. The inequality $\alpha\leq\delta^+_J-\nu$ for $\alpha\in J$ expresses that $\nu$ is coindependent in $J$ and that $\alpha\in J\setminus\nu$. Thus \eqref{minor1} is equivalent to $\nu$ being coindependent in $J$, together with the existence of $\alpha\in J\setminus\nu$ such that
 \[
  \delta^-_{J\setminus\nu} \ + \ \big(\delta^-_J \ - \ \delta^-_{J\setminus \nu} \ + \ \mu \big) \ = \ \delta^-_J \ \ + \ \mu \leq \ \alpha.
 \]
 Since $\delta^-_{J\setminus\nu}+0\leq\alpha$, this latter inequality turns into $\delta^-_{J\setminus\nu}+\mu'\leq\alpha$, which means that $\mu'$ is independent in $J\setminus\nu$. This establishes the equivalence between \eqref{minor1} and \eqref{minor3}.
 
 We turn to the last claim, assuming \eqref{minor1}--\eqref{minor3}, which follows from the direct computation
  \begin{align*}
  (J\setminus \nu)/\mu \ &= \ \big\{ \alpha-\mu' \ \big| \ \alpha\in J,\ \mu'+\delta^-_{J\setminus\nu}\leq \alpha\leq\delta^+_J-\nu \big\} \\
  &= \ \big\{ \alpha-\mu \ \big| \ \alpha\in J,\ \mu+\delta^-_J\leq \alpha,\ \alpha-\mu\leq\delta^+_{J/\mu}-\nu' \big\} \ + \ \big(\mu-\mu'\big) \\
  &= \ (J/\mu)\setminus\nu' \ + \ \sup\{-\mu,\ \ \delta^-_{J\setminus \nu} - \delta^-_J\}. \qedhere
 \end{align*}
\end{proof}

\begin{ex}
 Note that we need to use the supremum in the definitions of $\mu'$ and $\nu'$ in \autoref{prop: contraction and deletion commute} since the tuples $\nu-\delta^+_J+\delta^+_{J/\mu}+\mu$ and $\delta^-_J-\delta^-_{J\setminus \nu}+\mu$ might have negative coefficients. An example where this happens is $J=\Delta^1_3$ with $\mu=\epsilon_1$ and $\nu=\epsilon_3$. In this case, we have $\nu-\delta^+_J+\delta^+_{J/\mu}+\mu=-\epsilon_2$. Similarly, we have $\delta^-_{J^\ast}-\delta^-_{J^\ast\setminus\nu}+\mu=-\epsilon_2$.
\end{ex}

\begin{df}
 An \emph{embedded minor of $J$} is an M-convex set of the form
 \[
  J\minor{\nu}{\mu} \ + \ \tau \ = \ (J\setminus\nu)/\mu \ + \ \tau,
 \]
 together with the \emph{minor embedding}
 \[
  \begin{array}{cccc}
   \iota_{J\minor\nu\mu+\tau}: & J\minor{\nu}{\mu}+\tau & \longrightarrow & J \\
                               & \alpha                 & \longmapsto     & \alpha+\mu-\tau, 
  \end{array}
 \]
 where $\tau\geq-\delta^-_{J\minor\nu\mu}$ and $\mu+\delta^-_{J\setminus\nu}\leq\alpha\leq\delta^+_{J}-\nu$ for some $\alpha\in J$.
\end{df}

Note that the minor embedding $\iota_{J\minor\nu\mu+\tau}$ is an injective map. Note further that $\nu$, $\mu$ and $\tau$ are uniquely determined by the minor embedding $\iota_{J\minor\nu\mu+\tau}$. In \autoref{section: representations of embedded minors and duals}, we investigate the relationship between representations of $J$ and representations of its embedded minors.

\begin{rem}[Truncations with cubes and the multi-affine part]\label{rem: matroid truncations}
 A \emph{cube in $\N^n$} is a the non-empty intersection of $\N^n$ with a product of closed intervals, i.e.\ a subset of the form
 \[
  I \ = \ I_{\beta,\gamma} \ = \ \{\alpha\in\N^n\mid \beta\leq\alpha\leq\gamma\}
 \]
 with $\beta\leq\gamma$ in $\N^n$. As observed in the proof of \cite[Lemma 2.8]{Branden-Huh20}, the intersection of an M-convex set $J$ with a cube $I$ is M-convex, provided it is non-empty. In fact, such an intersection is an embedded minor of $J$:
 \[
  J\cap I_{\beta,\gamma} \ = \ J\minor\nu\mu+\mu
 \]
 for $\nu=\sup\{\0,\ \delta^+_J-\gamma\}$ and $\mu=\sup\{\0,\ \beta-\delta^-_{J\setminus\nu}\}$. In particular, the \emph{multi-affine part of $J$} is the intersection of $J$ with the unit cube $I_{\0,\1}$, which is M-convex whenever it is not empty (cf.\ \cite[Cor.\ 3.5]{Branden-Huh20}).
\end{rem}

\subsection{Permutation and extension of variables}
\label{subsection: permutation and extension of variables}

In this section, we discuss two additional operations on an M-convex set $J\subseteq\Delta^r_n$: permutation and extension of variables.

The proofs of the following two lemmas are immediate from the defining property of M-convex sets:

\begin{lemma}\label{lemma: permutation of variables}
 Let $\sigma\in S_n$ be a permutation of $[n]$ and $\iota_\sigma\colon\Delta^r_n\to\Delta^r_n$ the bijection defined by $\iota_\sigma(\alpha_i)=\alpha_{\sigma(i)}$ for $\alpha\in\Delta^r_n$ and $i\in[n]$. Then $\iota_\sigma(J)$ is M-convex.
\end{lemma}

\begin{lemma}\label{lemma: extension of variables}
 Let $\iota_{n+1}\colon[n]\to[n+1]$ be the inclusion given by $\iota_{n+1}(i)=i$ for $i\in[n]$. Then $\iota_{n+1}(J)\subseteq\Delta^r_{n+1}$ is M-convex.
\end{lemma}

\begin{df}\label{df: type of a polymatroid}
 An M-convex set $J'$ is said to be \emph{elementary equivalent to $J$} if $J'=J+\tau$ for some $\tau\geq-\delta^-_J$, or $J'=\iota_\sigma(J)$ for some $\sigma\in S_n$, or $J'=\iota_n(J)$, or $J=\iota_n(J')$. Two M-convex sets $J$ and $J'$ are said to be \emph{combinatorially equivalent} if there exists a chain of elementary equivalences $J=J_0\sim \dotsc\sim J_\ell=J'$. In this case, we say that $J'$ is \emph{of type $J$}.

An \emph{elementary polymatroid embedding} is a map $\iota\colon J\to J'$ between polymatroids which is given by either a minor embedding $\iota_{J\minor\nu\mu+\tau}$, a permutation of variables $\iota_\sigma$, an extension of variables $\iota_n$, or the inverse $\iota_n^{-1}$ of an extension of variables.
 A \emph{polymatroid embedding} is a map $\iota\colon J\to J'$ between polymatroids that is the composition of elementary polymatroid embeddings.
\end{df}

\begin{lemma}\label{lemma: combinatoral equivalence as bijective polymatroid embedding}
 Two polymatroids $J$ and $J'$ are combinatorially equivalent if and only if there exists a bijective polymatroid embedding $J\to J'$.
\end{lemma}

\begin{proof}
 First note that each type of elementary equivalence induces a bijection $J\to J'$ between the respective M-convex sets, which is given by $\alpha\mapsto\alpha+\tau$ (translation), $\alpha\mapsto \iota_\sigma(\alpha)$ (permutation), and $\alpha\mapsto\iota_n(\alpha)$ (extension of variables), respectively. The inverse bijection is in each case also a polymatroid embedding, namely, $\alpha\mapsto\alpha-\tau$, $\alpha\mapsto \iota_{\sigma^{-1}}(\alpha)$, and $\iota_n^{-1}$, respectively. Thus, if $J$ and $J'$ are combinatorially equivalent, the composition of the defining elementary equivalences yield a bijective polymatroid embedding $J\to J'$.
 
 Conversely, we note that every elementary polymatroid embedding is injective and that
 the polymatroid embeddings $\iota_{J\setminus\epsilon_\ell}\colon J\setminus\epsilon_\ell\to J$ and $\iota_{J/\epsilon_\ell}\colon J/\epsilon_\ell\to J$ are not surjective. Thus, if $\iota\colon J\to J'$ is a bijective polymatroid embedding, it must be composed of elementary equivalences, which shows that $J$ and $J'$ are combinatorially equivalent.
\end{proof}

\subsection{Comparison with Whittle's notion of minors}
\label{subsection: comparison with Whittle's notion of minors}

In \cite{Whittle92}, Whittle introduces single-element deletions and contractions for rank functions of (discrete) polymatroids; cf.~also \cite{Oxley-Whittle93}. We recall Whittle's construction and compare it to our notions of deletion and contraction.

Let $\br\colon2^{[n]}\to\N$ be the rank function of $J$, i.e., $\br(S)=\max\,\{\alpha_S\mid\alpha\in J\}$ for $S\subseteq[n]$. For ease of notation, and without loss of generality, we describe Whittle's operations only for the element $n$. 

Let $\iota_{n-1}\colon\N^{n-1}\to\N^n$ be the inclusion of the first $n-1$ coordinates.
The \emph{deletion of $n$ in $\br$} is the function $\br\setminus n\colon2^{[n-1]}\to\N$ given by $\br\setminus n(S)=\br(\iota_{n-1}(S))$ for $S\subseteq[n-1]$. The \emph{contraction of $n$ in $\br$} is the function $\br/n\colon2^{[n-1]}\to\N$ given by $\br/n(S)=\br(\iota_{n-1}(S)\cup n)-\br(n)$ for $S\subseteq[n-1]$. 

Both $\br\setminus n$ and $\br/n$ are indeed rank functions of polymatroids, i.e., they satisfy
\[\textstyle
 \br\setminus n(S) \ = \ \max\,\{\alpha_S \mid \alpha\in J_{\setminus n} \} \qquad \text{and} \qquad \br/n(S) \ = \ \max\,\{\alpha_S \mid \alpha\in J_{/n} \}
\]
for all $S\subseteq[n-1]$ and (uniquely determined) M-convex sets $J_{\setminus n}$ and $J_{/n}$. The following result identifies these two M-convex sets (embedded into $\N^n$ via $\iota_{n-1}\colon\N^{n-1}\to\N^n$) with embedded minors of $J$ in the sense of our paper.

\begin{prop}\label{prop: comparison with minors of Whittle}
 Define $\mu=(\delta^+_{J,n}-\delta^-_{J,n})\cdot\epsilon_n$ and $\tau=\delta^-_n\cdot\epsilon_n$. Then
 \[
  \iota_{n-1}(J_{\setminus n}) \ = \ J\setminus \mu \ - \ \tau \qquad \text{and} \qquad \iota_{n-1}(J_{/n}) \ = \ J/\mu \ - \ \tau.
 \]
\end{prop}

\begin{proof}
 We only explain the proof the first identity; the proof of the second identity is analogous. (Alternatively, it can be deduced from the first identity by establishing suitable compatibilities between polymatroid duality in the sense of this paper, rank functions, and $\iota_{n-1}$.)
 
 We establish the first equality by identifying the rank function $\br_{\setminus n}$ of $\iota_{n-1}(J_{\setminus n})$ with the rank function $\br'$ of 
 \[
  J\setminus\mu \ - \ \tau \ = \ \big\{\alpha\in\Delta^{r-\delta^-_n}_n \, \big| \, \alpha_n=0,\ \alpha+\tau\in J \big\}. 
 \]
 
 We first show that $\br'(S)\leq\br_{\setminus n}(S)$ for all $S\subseteq [n]$. Note that evidently $\br_{\setminus n}(S-n)\leq\br_{\setminus n}(S)$, and that the above description of $J\setminus \mu-\tau$ yields $\br'(S-n)=\br'(S)$. Consider $\alpha\in J\setminus\mu-\tau$ with $\br'(S-n)=\alpha_{S-n}$ and let $\beta=\alpha+\tau\in J$. Then
 \[
  \br'(S) \ = \ \br'(S-n) \ = \ \alpha_{S-n} \ = \ \beta_{S-n} \ \leq \ \br_{\setminus n}(S-n) \ \leq \ \br_{\setminus n}(S),
 \]
 as desired.
 
 Next we show that, for every $S\subseteq[n-1]$, there exists an $\alpha\in J$ with $\br_{\setminus n}(S)=\alpha_S$ and $\alpha_n=\delta^-_n$. For this, choose $\alpha\in J$ with $\br_{\setminus n}(S)=\alpha_S$ so that $\alpha_n$ is minimal, and choose $\beta\in J$ with $\beta_n=\delta^-_J$. If $\alpha_n>\delta^-_n=\beta_n$, there is a $k\in[n]$ such that $\alpha'=\alpha+\epsilon_k-\epsilon_n\in J$. This element satisfies $\br_{\setminus n}(\alpha)\leq\br_{\setminus n}(\alpha')$ and $\alpha'_n<\alpha_n$, which contradicts our assumptions on $\alpha$. Thus $\alpha_n=\delta^-_J$, as desired.
 
 In order to show that $\br'(S)\geq\br_{\setminus n}(S)$, we choose $\alpha\in J$ with $\br_{\setminus n}(S-n)=\alpha_{S-n}$ and $\alpha_n=\delta^-_J$. Then $\alpha-\tau\in J\setminus\mu-\tau$, and thus
 \[
  \br'(S) \ \geq \ (\alpha-\tau)_S \ = \ \br_{\setminus n}(S),
 \]
 which completes the proof.
\end{proof}

\begin{rem}\label{rem: comparison with duality by Whittle}
 The main focus of Whittle's paper \cite{Whittle92} is on duality operations which interchange deletion and contraction,
 in the sense that $(\br\setminus n)^\ast=\br^\ast/n$. This is a stricter requirement than what the duality operation, in the sense of this paper, satisfies. Under the additional assumption that duality is an involution, \cite{Whittle92} shows that such a duality operation only exists for particular subclasses of polymatroids. The corresponding duality functions from \cite{Whittle92} do not correspond to the duality in the sense of this paper (not even up to translation).
\end{rem}

\subsection{Direct sums}
\label{subsection: direct sums}

Let $J_1\subseteq\Delta^{r_1}_{n_1}$ and $J_2\subseteq\Delta^{r_2}_{n_2}$ be M-convex sets. Let $r=r_1+r_2$ and $n=n_1+n_2$. For $\alpha_1\in J_1$ and $\alpha_2\in J_2$, we define $\alpha_1\oplus\alpha_2\in\N^n$ by
\[
 (\alpha_1\oplus\alpha_2)_i \ = \ \begin{cases}
                                   \alpha_{1,i} & \text{if }1\leq i\leq n_1, \\
                                   \alpha_{2,i-n_1} & \text{if }n_1<i\leq n. \\
                                  \end{cases}
\]
Note that $\norm{\alpha_1\oplus\alpha_2}=r_1+r_2=r$ and thus $\alpha_1\oplus\alpha_2\in\Delta^r_n$. The \emph{direct sum of $J_1$ and $J_2$} is the subset
\[
 J_1\oplus J_2 \ = \ \{\alpha_1\oplus\alpha_2 \in\Delta^r_n\mid \alpha_1\in J_1,\ \alpha_2\in J_2\}
\]
of $\Delta^r_n$, which inherits the exchange property from $J_1$ and $J_2$ and is therefore M-convex.

The direct sum of polymatroids behaves well with respect to the operations and constructions of the previous sections. The following properties are easy to verify, and will not be used elsewhere in this paper, so we omit the straightforward proofs.

\begin{prop}\label{prop: properties of the direct sum of polymatroids}
 The direct sum $J=J_1\oplus J_2$ satisfies the following properties:
 \begin{enumerate}
  \item $\delta^-_J=\delta^-_{J_1}\oplus\delta^-_{J_2}$, \ $\delta^+_J=\delta^+_{J_1}\oplus\delta^+_{J_2}$ and $\delta_J=\delta_{J_1}\oplus\delta_{J_2}$;
  \item $J^\ast=J_1^\ast\oplus J_2^\ast$;
  \item $(J_1\oplus J_2)\minor{(\nu_1\oplus\nu_2)}{(\mu_1\oplus\mu_2)}=\big(J_1\minor{\nu_1}{\mu_1}\big)\oplus\big(J_1\minor{\nu_2}{\mu_2}\big)$;
  \item $J_1\oplus J_2$ and $J_2\oplus J_1$ are elementary equivalent (by a suitable permutation);
  \item $\overline J=\overline J_1\oplus\overline J_2$, where $\overline J$ denotes the reduction of a polymatroid $J$ (cf.~\autoref{df:reduction}).
 \end{enumerate}
\end{prop}

\subsubsection{Decomposition into indecomposable summands}
\label{subsubsection: decomposition into indecomposable summands}

An M-convex set $J\subseteq\Delta^r_n$ is \emph{decomposable} if it is combinatorially equivalent to the direct sum $J_1\oplus J_2$ of two polymatroids with nonempty ground sets $[n_1]$ and $[n_2]$. If $J$ is not decomposable and its ground set is nonempty, then it is \emph{indecomposable}. If $J$ is a matroid, then it is indecomposable as an M-convex set if and only if it is connected as a matroid.

For $\alpha\in J$ and $S\subseteq[n]$, we write $\alpha_S=\sum_{i\in S}\alpha_i$; in particular, $\norm\alpha=\alpha_{[n]}$. We say that a subset $S\subseteq[n]$ is a \emph{direct summand of $J$} if there is an $r_S\in\N$ such that $\alpha_S=r_S$ for all $\alpha\in J$.

\begin{lemma}\label{lemma: characterization of indecomposable polymatroids}
 If $S\subseteq[n]$ is a direct summand of $J$ and $T=[n]-S$, then $J$ is elementary equivalent to $J_1\oplus J_2$ (by a suitable permutation of variables) for two M-convex sets $J_1$ and $J_2$ with respective ground sets $[\# S]$ and $[\# T]$. In particular, $J\subseteq\Delta^r_n$ is indecomposable if and only if for every nonempty proper subset $S\subsetneq[n]$, there are elements $\alpha,\beta\in J$ with $\alpha_S\neq\beta_S$.
\end{lemma}

\begin{proof}
 By assumption there is an $r_1\in\N$ such that $\alpha_S=r_1$ for all $\alpha\in J$. After permuting $[n]$, we can assume that $S=[n_1]\subseteq[n]$ for $n_1=\# S$. Let $T=[n_2]$ for $n_2=\# T=n-n_1$ and $r_2=r-r_1$. Then every $\alpha\in J$ can be written as $\alpha=\alpha_{J_1}\oplus\alpha_{J_2}$ for $\alpha_{J_1}\in\Delta^{r_1}_{n_1}$ and $\alpha_{J_2}\in\Delta^{r_2}_{n_2}$. Define $J_1=\{\alpha_{J_1}\mid\alpha\in J\}$ and $J_2=\{\alpha_{J_2}\mid\alpha\in J\}$. We claim that $J=J_1\oplus J_2$ is a decomposition of $J$.

 First note that $J_1$ and $J_2$ are M-convex, since the fact that $\alpha_S=r=\beta_S$ for any two $\alpha,\beta\in J$ guarantees that when we apply the exchange axiom, we substitute an $i\in S$ by a $j\in S$. This allows us to deduce the exchange axiom for $J_1$ and $J_2$ from the exchange axiom for $J$.

 It is evident that $J\subseteq J_1\oplus J_2$. Consider $\alpha_1\in J_1$ and $\alpha_2\in J_2$. We need to show that $\alpha_1\oplus\alpha_2\in J$. By the definition of $J_1$ and $J_2$, there are $\beta_1\in J_1$ and $\beta_2\in J_2$ such that $\alpha_1\oplus\beta_2, \beta_1\oplus\alpha_2\in J$. Using the fact that the exchange axiom substitutes an $i\in S$ by a $j\in S$, we can exchange the elements in the difference $\beta_1 -\alpha_1$ one by one, and by induction we obtain 
 $\alpha_1\oplus\alpha_2 \in J$ as desired. This establishes the first claim.

 We turn to the second claim. If $J$ is indecomposable and $S\subseteq[n]$ is a direct summand, then $S=\emptyset$ or $S=[n]$. Conversely, if $J$ decomposes into the direct sum $J_1\oplus J_2$ of two polymatroids with nonempty ground sets $[n_1]$ and $[n_2]$, then 
for every nonempty proper subset $S=[n_1]$ of $[n]$, $\alpha_S=\norm{\alpha_1}$ is equal to the rank $r_1$ of $J_1$ for all $\alpha=\alpha_1\oplus\alpha_2\in J$. This establishes the reverse implication.
\end{proof}

Note that $(J_1\oplus J_2)\oplus J_3=J_1\oplus(J_2\oplus J_3)$, which allows us define the direct sum
\[
 \bigoplus_{i=1}^m J_i \ = \ J_1\oplus\dotsb\oplus J_m
\]
of M-convex sets $J_1,\dotsc,J_m$ unambiguously. As in \autoref{prop: properties of the direct sum of polymatroids}, permuting the summands results in a direct sum that is elementary equivalent to $\bigoplus J_i$ (by a suitable permutation of variables).

\begin{prop}\label{prop: unique decompistion into indecomposable polymatroids}
 Let $J$ be an M-convex set. Then there are a unique positive integer $m$ and indecomposable M-convex sets $J_1,\dotsc,J_m$, which are unique up to combinatorial equivalence and a permutation of indices, such that $J$ is combinatorially equivalent to $J_1\oplus\dotsb\oplus J_m$.
\end{prop}

\begin{proof}
 By \autoref{lemma: characterization of indecomposable polymatroids}, every direct summand $S\subseteq[n]$ of $J$ induces a decomposition into $J_1\oplus J_2$ (up to permutation of variables). Using this fact repeatedly leads to a decomposition of $J$ into a finite number of indecomposable M-convex sets.
 
 The uniqueness follows from the fact that the intersection $S\cap T$ of two direct summands $S$ and $T$ of $J$ is also a direct summand of $[n]$. To see this, we prove that $\alpha_{S\cap T}=\beta_{S\cap T}$ for all $\alpha,\beta\in J$ by induction on the \emph{distance} $d(\alpha,\beta)=\frac12\cdot\sum_{i\in[n]}\norm{\alpha_i-\beta_i}$ between $\alpha$ and $\beta$.

 If $d=d(\alpha,\beta)=0$, then $\alpha_{S\cap T}=\beta_{S\cap T}$, as claimed. If $d>0$, then $\alpha_i<\beta_i$ for some $i\in[n]$, and thus there is a $j\in[n]$ such that $\alpha'=\alpha+\epsilon_i-\epsilon_j\in J$. Since $d(\alpha',\beta)=d(\alpha,\beta)-1$, the inductive hypothesis shows that $\alpha'_{S\cap T}=\beta_{S\cap T}$. Since $\alpha_S=\beta_S$ and $\alpha_T=\beta_T$, we have $i\in S$ if and only if $j\in S$, and $i\in T$ if and only if $j\in T$. Thus $\alpha_{S\cap T}=\alpha'_{S\cap T}$, and the result now follows by induction.
\end{proof}

\part{Representations of polymatroids}

\section{Tracts}
\label{section: tracts}

Tracts were introduced in \cite{Baker-Bowler19} as a streamlined way to systematically study matroids with coefficients. 
The axioms for a tract are similar to the axioms for a field, but we relax the requirement that addition is a binary operation. 
More precisely, we do not define the sum of two elements of a tract $F$, but rather declare certain formal sums $a_1 + \cdots + a_k$ of elements of $F$ to be ``null'' and the rest to be non-null.

In this section, we review the definition and basic properties of tracts and provide a number of examples.

\subsection{Definition of tracts} 

A \emph{pointed monoid} is a multiplicatively written commutative monoid $F$ with identity element $1$ and a distinguished \emph{absorbing} element $0$ that satisfies $0\cdot a=0$ for all $a\in F$. The \emph{unit group of $F$} is the group
\[
 F^\times \ := \ \{a\in F\mid ab=1\text{ for some }b\in F\}
\]
of all invertible elements in $F$. 

A \emph{pointed (abelian) group} is a pointed monoid $F$ such that $F^\times=F-\{0\}$. The \emph{ambient semiring} of a pointed group $F$ is the group semiring
\[
 F^+ \ := \ \N[F^\times].
\]

We denote its elements by $\sum n_a.a$, where $n_a\in\N$ and $n_a=0$ for all but finitely many $a\in F^\times$. We sometimes write the sum $\sum n_a.a$ as $n_1.a_1+\dotsb+n_r.a_r$ or $\sum a_i$ (where $a$ appears $n_a$ times as a summand). The pointed group $F$ embeds into $F^+$ by sending $0$ to the empty sum (which is the additive identity element of $F^+$) and $a\in F^\times$ to $a=1.a\in F^+$. 

An \emph{ideal of $F^+$} is a subset $I$ that contains $0$ and is closed under addition and under multiplication by elements of $F^+$. For a subset $S\subseteq F^+$, we denote by $\langle S\rangle$ the ideal of $F^+$ generated by $S$.

A \emph{tract}\footnote{We deviate in this text from the definition of a tract in \cite{Baker-Bowler19} by imposing the property that $N_F$ is closed under addition. What we call a tract in this text should, strictly speaking, be called an \emph{ideal tract} or \emph{idyll} (cf.\ \cite{Baker-Lorscheid21b}).} is a pointed group $F$ together with an ideal $N_F$ of $F^+$, called the \emph{null set of $F$}, such that for every $a\in F$ there is a unique $b\in F$ with $a+b\in N_F$. We write $-a$ for the unique element $b$ with $a+b\in N_F$, and call it the \emph{additive inverse of $a$}. We often write $a-b$ instead of $a+(-b)$. Typically, we denote a tract by $F$ and suppress its null set $N_F$ from the notation. 

Note that the tract axioms imply that $a\in N_F$ if and only if $a=0$; in particular $-0=0$. Furthermore, we have $(-1)^2=1$ and $a+b\in N_F$ if and only if $b=-a$.

A \emph{tract morphism} is a multiplicative map $f\colon F_1\to F_2$ with $f(0)=0$ and $f(1)=1$ such that $\sum f(a_i)\in N_{F_2}$ for all $\sum a_i\in N_{F_1}$. This defines the category $\Tracts$.

\subsection{First examples}
Every field $F$ is a tract by defining the null set of $F$ as
\[\textstyle
 N_F = \{ \sum a_i \mid \sum a_i=0\text{ in }F\}.
\]

This construction extends to \emph{partial fields} and \emph{hyperfields}. Semple and Whittle's original definition of a partial field in \cite{Semple-Whittle96} is in terms of a pointed group $F$ together with a partially defined addition $+:F\times F\dashrightarrow F$ that satisfies certain axioms. Equivalently, the partial addition can be captured in terms of a certain ring $R$ that contains $F$ as a multiplicative submonoid (\cite[Thm.\ 5.1]{Pendavingh-vanZwam10b}). By \cite{Baker-Lorscheid21b}, every partial field $F$ can be viewed as a tract by defining the null set of $F$ as
\[\textstyle
  N_F \ = \
  \{ \sum a_i \mid \sum a_i=0\text{ in }R\}.
\]

Similarly a hyperfield $F$ with hyperaddition $\hyperplus\colon F\times F\to 2^F$ (where $2^F$ is the power set of $F$) is a tract with respect to the null set
\[\textstyle
  N_F \ = \ \{ \sum a_i \mid 0\in \hypersum a_i\}.
\]
See \cite{Baker-Bowler19,Baker-Lorscheid21b} for additional details on partial fields and hyperfields.

In each case, homomorphisms are naturally tract morphisms. More to the point, these constructions define fully faithful embeddings of the categories of fields, partial fields, and hyperfields into $\Tracts$, which allows us to consider (partial / hyper-)fields as tracts by abuse of terminology.

Two concrete examples are the \emph{regular partial field} $\Funpm=\{0,1,-1\}$, whose null set is
\[
 N_{\Funpm} \ = \ \{n.1+n.(-1)\mid n\in\N\} \ = \ \{0,\quad 1-1, \quad 1+1-1-1,\dotsc\},
\]
and the \emph{Krasner hyperfield} $\K=\{0,1\}$, whose null set is
\[
 N_\K \ = \ \N-\{1\} \ = \ \{0,\quad 1+1,\quad 1+1+1,\quad 1+1+1+1,\dotsc\}.
\]
Note that $-1=1$ in $\K$. The regular partial field $\Funpm$ is an initial object of $\Tracts$, and the Krasner hyperfield $\K$ is a terminal object of $\Tracts$. 

\subsection{Subtracts}
Let $F$ be a tract. A \emph{subtract of $F$} is a pointed submonoid $A$ of $F$ with $-1\in A^\times=A-\{0\}$, which is a tract in its own right with respect to the null set
\[\textstyle
 N_A \ = \ \{ \sum a_i\in N_F \mid a_i\in A\}.
\]

\subsection{Free algebras}
\label{subsubsection: free algebras}
Let $k$ be a tract. A \emph{$k$-algebra} is a tract $F$ together with a tract morphism $\alpha_F\colon k\to F$. A $k$-linear morphism between $k$-algebras is a tract morphism $f\colon F_1\to F_2$ between $k$-algebras $F_1$ and $F_2$ such that $\alpha_{F_2}=f\circ\alpha_{F_1}$.

Let $\{x_i\}_{i\in I}$ be a set. The \emph{free $k$-algebra in $\{x_i\}$} is defined as follows: as a pointed monoid, it is
\[\textstyle
 k(x_i) \ = \ k(x_i)_{i\in I} \ = \ \big\{ a\cdot\prod x_i^{\epsilon_i} \, \big| \, a\in k,\ (\epsilon_i)\in\bigoplus_I\Z\big\} \, \big/ \, \sim,
\]
where $\sim$ is the equivalence relation generated by $0:=0\cdot\prod x_i^0\sim 0\cdot \prod x_i^{\epsilon_i}$ for any $(\epsilon_i)\in\bigoplus_I\N$. The association $a\mapsto a\cdot x_j^0$ defines an embedding of $k$ as a submonoid of $k(x_i)$, which extends by linearity to an embedding $k^+\to k(x_i)^+$. The nullset of $k(x_i)$ is the ideal $N_{k(x_i)}$ generated by the image of $N_k$ in $k(x_i)^+$. We write $a x_{i_1}^{\epsilon_{i_1}} \dotsb x_{i_r}^{\epsilon_{i_r}}$ for $a\cdot\prod x_i^{\epsilon_i}$ with $\epsilon_j=0$ for $j\notin\{i_1,\dotsc,i_r\}$. 

By construction, the inclusion $k\to k(x_i)$ is a tract morphism, which turns $k(x_i)$ into a $k$-algebra. It satisfies the expected universal property: every set-theoretic map $f_0\colon\{x_i\}\to F$ into a $k$-algebra $F$ extends uniquely to a $k$-linear morphism $f\colon k(x_i)\to F$ with $f(x_i)=f_0(x_i)$ (this is proven exactly as for pastures; cf.\ \cite[Prop.\ 2.6]{Baker-Lorscheid20}).

\subsection{Quotients}
\label{subsubsection: quotients of tracts}
Let $F$ be a tract and $S\subseteq F^+$ be a subset that does not contain any element of $F^\times$. The \emph{quotient of $F$ by $S$} is the quotient monoid
\[
 \pastgen FS \ := \ F/\sim,
\]
where $\sim$ is the equivalence relation generated by the relations $ca\sim cb$ for all $a-b\in S$ and $c\in F$, together with the null set
\[\textstyle
 N_{\pastgen FS} \ := \ \big\langle \sum [ca_i] \, \big| \, c\in F,\ \sum a_i\in N_F\cup S\big\rangle,
\]
where $[a]$ denotes the class of $a\in F$ in $\pastgen FS$.

The quotient map $\pi_S\colon F\to \pastgen FS$ is a tract morphism, which turns $\pastgen FS$ into an $F$-algebra. It satisfies the expected universal property: every tract morphism $f\colon F\to F'$ with $\sum f(a_i)\in N_{F'}$ for all $\sum a_i\in S$ factors into $f=\bar f\circ\pi_S$ for a uniquely determined morphism $\bar f\colon\pastgen FS\to F'$ (this is proven exactly as for pastures; cf.\ \cite[Prop.\ 2.6]{Baker-Lorscheid20}).

\subsection{Tensor products}
\label{subsection: tensor products}

The category of tracts is complete and cocomplete. In a nutshell, $\Funpm$ is an initial object, $\K$ is a terminal object, products are given by Cartesian products of the unit groups, equalizers are defined as the set-theoretic equalizers, and coequalizers can be constructed in terms of a quotient construction. The only subtle construction (similar to the constructions for rings) is the coproduct, or tensor product, of tracts, which is given by the following universal property.

The \emph{tensor product} of a family $\{F_i\}_{i\in I}$ of tracts is a tract $\bigotimes F_i$, together with morphisms $\iota_i:F_i\to \bigotimes F_i$ (one for each $i\in I$), such that the induced map
\[
 \Hom\Big(\bigotimes\limits_{i\in I} F_i, \, F'\Big) \ \longrightarrow \ \prod\limits_{i\in I} \ \Hom(F_i, F')
\]
is a bijection for all tracts $F'$. The construction of $\bigotimes F_i$ is analogous to the case of pastures (\cite[Lemma 2.7]{Baker-Lorscheid20}) and bands (\cite[Prop.\ 1.42]{Baker-Jin-Lorscheid25}).

\subsection{More examples}
\label{subsection: more examples of tracts}

Every tract $F$ can be written  in the form $F=\pastgen{\Funpm(x_i)}S$ by choosing a suitable set of generators $x_i$ and a suitable set $S$ of defining relations. Some examples are:
\begin{align*}
 \F_2 \ &= \ \pastgen\Funpm{1+1}                      && \text{(the field with $2$ elements)} \\
 \F_3 \ &= \ \pastgen\Funpm{1+1+1}                    && \text{(the field with $3$ elements)} \\
 \S   \ &= \ \pastgen\Funpm{1+1-1}                    && \text{(the \emph{sign hyperfield})} \\
 \U   \ &= \ \pastgen{\Funpm(x,y)}{x+y-1}             && \text{(the \emph{near regular partial field})} \\
 \D   \ &= \ \pastgen{\Funpm(z)}{z-1-1}               && \text{(the \emph{dyadic partial field})} \\
 \H   \ &= \ \pastgen{\Funpm(z)}{z^3+1,\ \ z^2-z+1} && \text{(the \emph{hexagonal partial field})} \\
 \G   \ &= \ \pastgen{\Funpm(z)}{z^2-z-1}             && \text{(the \emph{golden ratio partial field})} 
\end{align*}  

If $F$ is a pointed group, there are several general ways to define a tract structure on $F$.

We define the \emph{trivial} tract structure on $F$ by letting the null set be 
\[\textstyle
  N_F \ = \ \gen{a + a \mid a \in F^\times},
\]
so that in particular $1=-1$.
For example, this provides $\R_{\geq 0}$ with the structure of a tract. (And, from now on, when we write $\R_{\geq 0}$ as a tract, we consider it with the trivial tract structure.)

We define the \emph{degenerate} tract structure on $F$ by letting the null set be
\[\textstyle
 N_F \ = \ \big\{\ a_1+\dotsb+a_n \ \big| \ \text{$a_2=-a_1$ or at least $3$ terms are nonzero}\big\}.
\]

The \emph{tropical hyperfield} is the tract
\[\textstyle
 \T_0 \ = \ \pastgen{\R_{\geq0}}{\sum a_i\mid \text{the maximum among $a_1,\dotsc,a_n$ appears at least twice}},
\]
the \emph{triangular hyperfield} is the tract
\[\textstyle
 \T_1 \ = \ \pastgen{\R_{\geq0}}{\sum a_i\mid \text{$a_1,\dotsc,a_n$ are the side lengths of a (possibly degenerate) polygon}},
\]
and the \emph{degenerate triangular hyperfield} is the tract
\[\textstyle
 \T_\infty \ = \ \pastgen{\R_{\geq0}}{\sum a_i\mid \text{the maximum appears twice or at least $3$ terms are nonzero}}.
\]
Note that $\T_\infty$ is equal to the pointed group $F=\R_{\geq0}$ endowed with the degenerate tract structure.

The tracts $\T_0, \T_1$, and $\T_\infty$ play a major role in the forthcoming papers \cite{BHKL1} and \cite{BHKL2}, where  a continuous family of tracts $\T_q$ (for $q\in[0,\infty]$)  that interpolates between them (a process which Viro calls {\em Litvinov--Maslov dequantization} in \cite{Viro10}) is considered.
Concretely, for $q>0$ the generalized triangular hyperfield $\T_q$ is defined as
\[\textstyle
 \T_q \ = \ \pastgen{\R_{\geq0}}{\sum a_i\mid \text{$a_1^{1/q},\dotsc,a_n^{1/q}$ are the side lengths of a (possibly degenerate) polygon}}.
\]

We also define the \emph{discrete tropical hyperfield} $\T_0^\Z$ as the subtract of $\T_0$ corresponding to the pointed submonoid $e^\Z \cup \{ 0 \}$ of $\T_0=\R_{\geq 0}$.

\section{Representations of polymatroids}
\label{section: representations of polymatroids}

In this section, we extend the notions of strong and weak matroid representations over tracts (cf.\ \cite{Baker-Bowler19,Baker-Lorscheid20}) to polymatroids, using a novel characterization of polymatroids in terms of Pl\"ucker relations.

\subsection{Pl\"ucker relations for polymatroids}
\label{subsection: Plucker relations for polymatroids}

We consider the characteristic function of a subset $J\subseteq\Delta^r_n$, defined by $\chi_J(\alpha)=1$ if $\alpha\in J$ and $\chi_J(\alpha)=0$ if not, as a function
\[
 \chi_J\colon \ \Delta^r_n \ \longrightarrow \ \K
\]
into the Krasner hyperfield $\K=\pastgen{\Funpm}{1+1,1+1+1}$, which has elements $0$ and $1=-1$.
\begin{thm}\label{thm: Pluecker relations for polymatroids}
 A subset $J\subseteq\Delta^r_n$ is M-convex if and only if the characteristic function $\chi_J\colon\Delta^r_n\to\K$ of $J$ satisfies the \emph{Pl\"ucker relations}
 \[
  \sum_{k=0}^s \ \chi_J(\alpha+\epsilon_{i_0}+\dotsb\widehat{\epsilon_{i_k}}\dotsb+\epsilon_{i_s}) \ \cdot \ \chi_J(\alpha+\epsilon_{i_k}+\epsilon_{j_2}+\dotsb+\epsilon_{j_s}) \ \in \ N_\K
 \]
 for all $s\in\{2,\dotsc,r\}$, all $\alpha\in\Delta^{r-s}_n$ with $\delta^-_J=\inf\,J\leq\alpha$, and all $i_0,\dotsc,i_s,j_2,\dotsc,j_s\in [n]$ such that $\alpha+\epsilon_{i_0}+\dotsb+\epsilon_{i_s}+\epsilon_{j_2}+\dotsb+\epsilon_{j_s}\leq\delta^+_J=\sup\,J$.
\end{thm}

\begin{proof}
 Assume that $J$ is M-convex. Since $N_\K=\N-\{1\}$, it suffices to show, for $\alpha\in\Delta_n^{r-s}$ with $s\in\{2,\ldots,r\}$ and $i_0,\dotsc,i_s,j_2,\dotsc,j_s\in [n]$, that either all terms in 
 \[
  \sum_{k=0}^s \ \chi_J(\alpha+\epsilon_{i_0}+\dotsb\widehat{\epsilon_{i_k}}\dotsb+\epsilon_{i_s}) \ \cdot \ \chi_J(\alpha+\epsilon_{i_k}+\epsilon_{j_2}+\dotsb+\epsilon_{j_s})
 \]
 are zero or that at least two terms are nonzero. Assume that the sum contains a nonzero term $\chi_J(\alpha+\epsilon_{i_0}+\dotsb\widehat{\epsilon_{i_k}}\dotsb+\epsilon_{i_s})\chi_J(\alpha+\epsilon_{i_k}+\epsilon_{j_2}+\dotsb+\epsilon_{j_s})$. Then we define $i=i_k$, $\beta=\alpha+\epsilon_{i_0}+\dotsb\widehat{\epsilon_{i_k}}\dotsb+\epsilon_{i_s}$, and $\gamma=\alpha+\epsilon_{i_k}+\epsilon_{j_2}+\dotsb+\epsilon_{j_s}$. 
 
 If $\beta_i\geq\gamma_i$, then there exists $k'\neq k$ with $i_{k'}=i_k$. Thus we find a second nonzero term $\chi_J(\alpha+\epsilon_{i_0}+\dotsb\widehat{\epsilon_{i_{k'}}}\dotsb+\epsilon_{i_s})\chi_J(\alpha+\epsilon_{i_{k'}}+\epsilon_{j_2}+\dotsb+\epsilon_{j_s})$ in the Pl\"ucker relation.

 If $\beta_i<\gamma_i$, then, by the exchange axiom for M-convex sets, there exists $j\in[n]$ such that $\gamma_j<\beta_j$ and such that both $\beta-\epsilon_j+\epsilon_i$ and $\gamma-\epsilon_i+\epsilon_j$ are in $J$. Since $\gamma_j<\beta_j$, we have $j=i_{k'}$ for some $k'\neq k$. Thus $\alpha+\epsilon_{i_0}+\dotsb\widehat{\epsilon_{i_{k'}}}\dotsb+\epsilon_{i_s}=\beta-\epsilon_j+\epsilon_i$ and $\alpha+\epsilon_{i_{k'}}+\epsilon_{j_2}+\dotsb+\epsilon_{j_s}=\gamma-\epsilon_i+\epsilon_j$, which yields $\chi_J(\alpha+\epsilon_{i_0}+\dotsb\widehat{\epsilon_{i_{k'}}}\dotsb+\epsilon_{i_s})\chi_J(\alpha+\epsilon_{i_{k'}}+\epsilon_{j_2}+\dotsb+\epsilon_{j_s})$ as a second nonzero term in the Pl\"ucker relation. This shows that the characteristic function of an M-convex set satisfies the Pl\"ucker relations.

 Conversely, assume that $J\subseteq\Delta^r_n$ is a subset whose characteristic function $\chi_J$ satisfies the Pl\"ucker relations. Consider $\beta,\gamma\in J$ with $\beta_i<\gamma_i$ for some $i\in[n]$.
 We need to show that there exists $j\in[n]$ such that $\gamma_j<\beta_j$ and such that both $\beta-\epsilon_j+\epsilon_i$ and $\gamma-\epsilon_i+\epsilon_j$ are in $J$.
 Let $\alpha=\inf\{\beta,\gamma\}\geq\delta_J^-$. We have $\alpha\in\Delta^{r-s}_n$ for some $1 \leq s \leq r$. 
 If $s=1$, then $\gamma=\beta+\epsilon_i-\epsilon_j$ for some $j\in[n]$ and the exchange property we're trying to show is trivially satisfied. 
 Thus, we can assume that $s\in\{2,\ldots,r\}$. There are $i_1,\dotsc,i_s,j_1,\dotsc,j_s\in[n]$ (unique up to permutation) such that 
 \[
  \beta = \alpha+\epsilon_{i_1}+\dotsb+\epsilon_{i_s}, \ \ \gamma = \alpha+\epsilon_{j_1}+\dotsb+\epsilon_{j_s} \ \ \text{and} \ \ \{i_1,\dotsc,i_s\}\cap\{j_1,\dotsc,j_s\} = \emptyset,
 \]
 and therefore $\alpha+\epsilon_{i_1}+\dotsb+\epsilon_{i_s}+\epsilon_{j_1}+\dotsb+\epsilon_{j_s}  \leq  \sup\,\{\beta,\gamma\}  \leq  \delta^+_J$. Assume without loss of generality that $i=j_1$ and define $i_0=i$. Then the Pl\"ucker relation
 \[
 \sum_{k=0}^s \ \chi_J(\alpha+\epsilon_{i_0}+\dotsb\widehat{\epsilon_{i_k}}\dotsb+\epsilon_{i_s}) \ \cdot \ \chi_J(\alpha+\epsilon_{i_k}+\epsilon_{j_2}+\dotsb+\epsilon_{j_s}) \ \in \ N_\K
 \]
 contains the nonzero term $\chi_J(\alpha+\epsilon_{i_1}+\dotsb+\epsilon_{i_s})\chi_J(\alpha+\epsilon_{i_0}+\epsilon_{j_2}+\dotsb+\epsilon_{j_s})$, and therefore contains a second nonzero term $\chi_J(\alpha+\epsilon_{i_0}+\dotsb\widehat{\epsilon_{i_k}}\dotsb+\epsilon_{i_s}) \chi_J(\alpha+\epsilon_{i_k}+\epsilon_{j_2}+\dotsb+\epsilon_{j_s})$ for some $k\neq0$. Thus $J$ contains both
 \[
  \alpha+\epsilon_{i_0}+\dotsb\widehat{\epsilon_{i_k}}\dotsb+\epsilon_{i_s} \ = \ \beta-\epsilon_j+\epsilon_i \quad \text{and} \quad \alpha+\epsilon_{i_k}+\epsilon_{j_2}+\dotsb+\epsilon_{j_s} \ = \ \gamma-\epsilon_i+\epsilon_j
 \]
 for $j=i_k$. Since $i_k\notin\{j_1,\dotsc,j_s\}$, we have $\gamma_j<\beta_j$. This shows that $J$ is M-convex.
\end{proof}

\begin{rem}
 Not all Pl\"ucker relations are needed in the characterization of M-convex sets in \autoref{thm: Pluecker relations for polymatroids}. It is visible from the proof that we only need those Pl\"ucker relations for which $\{i_1,\dotsc,i_s\}\cap\{i_0,j_2,\dotsc,j_s\}=\emptyset$ in order to imply M-convexity.
\end{rem}

\subsection{Polymatroid representations over tracts}
\label{subsection: representations over tracts}

The definition of polymatroid representations over arbitrary tracts, in which $-1$ might differ from $1$, requires a suitable sign for the Pl\"ucker relations, which in turn depends on the ordering of the coordinates. This is best formulated in terms of $r$-tuples $\balpha\in[n]^r$ instead of vectors $\alpha\in\Delta^r_n$. We can compare both viewpoints in terms of the surjection
\begin{equation} \label{eq:SigmaDefinition}
 \begin{array}{cccl}
  \Sigma\colon & [n]^r & \longrightarrow & \Delta^r_n \\
               & \balpha & \longmapsto   & \epsilon_{\balpha_1}+\dotsb+\epsilon_{\balpha_r}.
 \end{array} 
\end{equation}
Furthermore, we use the shorthand notation $\balpha i_{1}\dotsc i_s$ for $(\balpha_1,\dotsc,\balpha_r,i_{1},\dotsc,i_s)$ where $\balpha\in[n]^r$ and $i_{1},\dotsc,i_s\in[n]$.

Recall from \autoref{df:reduction} that the \emph{reduction} of an $M$-convex set $J$ is the M-convex set $\Jbar=J-\delta^-_J$, where $\delta^-_J=\inf\, J$.

\begin{df}\label{df: polymatroid representation}
 Let $J\subseteq\Delta^r_n$ be an M-convex set. The \emph{effective rank of $J$} is the rank $\rbar=r-\norm{\delta^-_J}$ of $\Jbar$. The \emph{width of $J$} is $\omega_J=\delta^+_J-\delta^-_J=\delta^+_{\Jbar}$. 
 
 Let $F$ be a tract. A \emph{strong $F$-representation of $J$} is a function $\brho\colon[n]^\rbar\to F$ that satisfies the following axioms:
 \begin{enumerate}[label={(SR\arabic*)}]
  \item\label{SR1} $\brho(\balpha)\in F^\times$ if and only if $\Sigma\balpha\in \Jbar$;
  \item\label{SR2} $\brho(i_{\sigma(1)},\dotsc,i_{\sigma(\rbar)})=\sign(\sigma)\cdot \brho(i_1,\dotsb,i_\rbar)$ for every $\sigma\in S_\rbar$;
  \item\label{SR3} $\brho$ satisfies the \emph{Pl\"ucker relation} 
  \[
   \sum_{k=0}^s (-1)^k \cdot \brho(\balpha i_0\dotsc \widehat{i_k}\dotsc i_s) \cdot \brho(\balpha i_k j_2\dotsc j_s) \quad \in \quad N_F
  \]
  for all $2\leq s\leq \rbar$, $\balpha\in [n]^{\rbar-s}$ and $i_0,\dotsc,i_s,j_2,\dotsc,j_s\in [n]$ such that \[\Sigma\balpha i_0\dotsc i_sj_2\dotsc j_s\leq\omega_J.\]
 \end{enumerate} 
 A \emph{weak $F$-representation of $J$} is a function $\brho\colon[n]^\rbar\to F$ that satisfies the following axioms:
 \begin{enumerate}[label={(WR\arabic*)}]
  \item\label{WR1} $\brho(\balpha)\in F^\times$ if and only if $\Sigma\balpha\in \Jbar$;
  \item\label{WR2} $\brho(i_{\sigma(1)},\dotsc,i_{\sigma(\rbar)})=\sign(\sigma)\cdot \brho(i_1,\dotsb,i_\rbar)$ for every $\sigma\in S_\rbar$;
  \item\label{WR3} $\brho$ satisfies the \emph{$3$-term Pl\"ucker relations} 
  \[
   \brho(\balpha jk) \cdot \brho(\balpha il) \ - \ \brho(\balpha ik) \cdot \brho(\balpha jl) \ + \ \brho(\balpha ij) \cdot \brho(\balpha kl) \quad \in \quad N_F
  \]
  for all $\balpha\in [n]^{\rbar-2}$, and $i,j,k,l\in [n]$ such that $\Sigma\balpha ijkl\leq\omega_J$.
 \end{enumerate}
 The tract $F$ is called \emph{excellent} if every weak $F$-representation of every M-convex set $J$ is strong.
\end{df}

\begin{rem}
 If $J$ is a matroid, then this definition agrees with the notion of strong (resp.\ weak) $F$-representations of matroids in \cite{Baker-Lorscheid21}, which are also called a strong (resp.\ weak) Grassmann-Pl\"ucker functions in \cite{Baker-Bowler19,Baker-Lorscheid20}.
\end{rem} 

\begin{rem}\label{remark: pluecker relations imply m-convex support}
 Consider a function $\brho\colon[n]^\rbar\to F$ that satisfies \ref{SR2} and \ref{SR3} with respect to the set
 \begin{equation*}
    J:=\{\Sigma\alpha\mid \brho(\alpha)\neq0\}.
 \end{equation*}
 Then $J$ is M-convex and $\brho$ is a strong $F$-representation of $J$.
 Indeed, the composition of $\brho$ with the unique tract morphism $F\to\K$ satisfies the assumptions of \autoref{thm: Pluecker relations for polymatroids} and has the same support as $\brho$. This shows that $J$ is M-convex, and then $\brho$ is an $F$-representation of $J$ by definition.
 
 This extends a known fact for matroid representations to polymatroids. It fails for weak $F$-representations (in fact, already in the matroid case; see\ \cite[Ex.\ 6.25]{Baker-Lorscheid21b}): for every tract $F$, there is a function $\brho\colon[6]^3\to F$ that satisfies \ref{WR2} and \ref{WR3}, but that is not a weak $F$-representation of any M-convex set $J$.
\end{rem}

\subsubsection{The unique \texorpdfstring{$\K$}{K}-representation of a polymatroid}
\label{subsubsection: the unique K-representation of a polymatroid}

 Let $J$ be an M-convex set and let $\chi_J\colon\Delta^r_n\to \K$ be its characteristic function. Then the map $\brho_J\colon[n]^\rbar\to\K$, defined by $\brho_J(i_1,\dotsc,i_r)=\chi_J(\delta^-_J+\epsilon_{i_1}+\dotsc+\epsilon_{i_r})$, is a strong $\K$-representation of $J$. More precisely, $\brho$ is the unique strong (resp.\ weak) $\K$-representation of $J$, since it is entirely determined by axiom \ref{SR1} (resp.\ by \ref{WR1}). This shows, in particular, that $\K$ is excellent.

\subsection{The idempotency principle for proper polymatroids}
\label{subsection: the idempotency principle for proper polymatroids}

An M-convex set $J$ is \emph{a translate of a matroid} if $J=J'+\tau$ for a matroid $J'$ and $\tau\in\Z^n$. Otherwise, we call $J$ a \emph{proper polymatroid} or \emph{proper M-convex set}. 

Let $\omega_J=\delta^+_J-\delta_J^-$ be the width of $J$, $\Jbar=J-\delta^-_J$ its reduction, $\1=(1,\dotsc,1)\in\N^n$, and 
\[
 U_{2,3}^+ \ = \ \big\{ (2,0,0),\ (1,1,0),\ (1,0,1),\ (0,1,1) \big\} \ \subseteq \ \Delta^2_3.
\]

\begin{lemma}\label{lemma: characterizations of translates of matroids}
 The following are equivalent:
 \begin{enumerate}
  \item \label{proper1} $J$ is a translate of a matroid;
  \item \label{proper2} $\Jbar$ is a matroid;
  \item \label{proper3} $\omega_J\leq\1$;
  \item \label{proper4} $J$ has no embedded minor of type $\Delta^2_2$ or $U_{2,3}^+$.
 \end{enumerate}
\end{lemma}

\begin{proof}
 We establish the circle of implications \eqref{proper3}$\Rightarrow$\eqref{proper2}$\Rightarrow$\eqref{proper1}$\Rightarrow$\eqref{proper4}$\Rightarrow$\eqref{proper3} in the following. Assume \eqref{proper3}, i.e., $\omega_J\leq\1$. Since $\delta^-_{\Jbar}=0$ and $\omega_J$ is invariant under translates of $J$, we have $\delta^+_\Jbar=\omega_\Jbar=\omega_J\leq\1$, which shows that $\Jbar$ is a matroid and establishes \eqref{proper3}$\Rightarrow$\eqref{proper2}.
 
 Assume \eqref{proper2}, i.e., $\Jbar$ is a matroid. Then $J=\Jbar+\delta^-_J$ is a translate of a matroid. This establishes \eqref{proper2}$\Rightarrow$\eqref{proper1}.

 Assume \eqref{proper1}, i.e., $J=M+\tau$ is the translate of a matroid $M$. The matroid $M$ does not have embedded minors of type $\Delta^2_2$ or $U_{2,3}^+$, which are proper polymatroids. Since $J$ and $M$ have the same embedded minors, this establishes \eqref{proper1}$\Rightarrow$\eqref{proper4}.
 
 Assume \eqref{proper4} and let $i\in[n]$. Choose $\alpha,\beta\in J$ with $\alpha_i=\delta^+_{J,i}$ and $\beta_i=\delta^-_{J,i}$. If $\omega_{J,i}=\alpha_i-\beta_i\geq2$, then applying exchange axiom repeatedly to $\alpha$ and $\beta$ yields $k,l\in[n]\setminus\{i\}$ (not necessarily distinct) such that 
 \[
  \alpha-\epsilon_i+\epsilon_k, \qquad \alpha-\epsilon_i+\epsilon_l,\qquad \alpha-2\epsilon_i+\epsilon_k+\epsilon_l
 \]
 are in $J$. Thus
 \[
  J\,\minor{\,\delta^+_J-\alpha+\epsilon_k+\epsilon_l\,}{\,\alpha-2\epsilon_i} \ = \ \big\{\, 2\epsilon_i,\ \ \epsilon_i+\epsilon_k, \ \ \epsilon_i+\epsilon_l,\ \ \epsilon_k+\epsilon_l \,\big\},
 \]
 which is combinatorially equivalent to $\Delta_{2,2}$ or $U_{2,3}^+$, depending on whether $k=l$ or not. This contradicts our assumption \eqref{proper4}, which shows that $\omega_{J,i}\leq1$ and establishes \eqref{proper4}$\Rightarrow$\eqref{proper3}.
\end{proof}

A tract $F$ is \emph{idempotent}\footnote{This terminology stems from the fact that an idempotent semifield (commutative, with $0$ and $1$) is naturally a tract whose nullset is generated by all relations of the form $b+\sum a_i$ for which $\sum a_i=b$ holds in $F$. This tract is idempotent in the sense of this text. More concisely, this construction defines a fully faithful functor from idempotent semifields to idempotent tracts.} if $1=-1$ (i.e., $1+1\in N_F$) and $1+1+1\in N_F$. Equivalently, a tract $F$ is idempotent if and only if there exists a (necessarily unique) morphism $\K\to F$, i.e., if and only if $F$ is an algebra over the Krasner hyperfield.

A tract $F$ is \emph{near-idempotent} if $1=-1$ and if there is an $x\in F^\times$ with $1+1+x\in N_F$. Every idempotent tract is near-idempotent. A typical example of a near-idempotent tract that is not idempotent is $\F_2\otimes\D=\pastgen{\F_2(x)}{1+1+x}$.

\begin{prop}[Idempotency principle]\label{prop: idempotency principle}
 Let $F$ be a tract, $J\subseteq\Delta^r_n$ a proper M-convex set of effective rank $\rbar$, and $\brho\colon[n]^\rbar\to F$ a weak $F$-representation of $J$. Then $F$ is near-idempotent. If $\omega_{J,i}\geq3$ for some $i\in[n]$, then $F$ is idempotent.
\end{prop}

\begin{proof}
 If $J$ is not the translate of a matroid, then it contains an element of the form $\delta^-_J+\alpha+\epsilon_i+\epsilon_i$ for some $\alpha\in\Delta^{\rbar-2}_n$ and $i\in[n]$. Choose $\balpha\in[n]^{\rbar-2}$ with $\Sigma\balpha=\alpha$. Then by axiom \ref{WR1}, $\brho(\balpha ii)\in F^\times$ and by axiom \ref{WR2}, $\brho(\balpha ii)=-\brho(\balpha ii)$. Thus $1=-1$ in $F$. 
 
 Let $\beta=\delta^-_J+\alpha+2\epsilon_i$ and $\gamma\in J$ with $\gamma_i=\delta^-_{J,i}$. Since $\gamma_i\leq\beta_i-2$, we can apply the exchange axiom twice to find $k,l\in[n]-\{i\}$ such that all of \[
  \beta-\epsilon_i+\epsilon_k, \qquad \beta-\epsilon_i+\epsilon_l, \qquad \beta-2\epsilon_i+\epsilon_k+\epsilon_l
 \]
 are in $J$. Thus, in particular, $\alpha+2\epsilon_i+\epsilon_k+\epsilon_l\leq\omega_J$. By axiom \ref{WR3}, we have the Pl\"ucker relation 
 \[
  \brho(\balpha ii)\cdot \brho(\balpha kl) \ + \ \brho(\balpha ik)\cdot \brho(\balpha il) \ + \ \brho(\balpha il)\cdot \brho(\balpha ik) \quad \in \quad N_F.
 \]
 Dividing all terms by $\brho(\balpha ik)\cdot\brho(\balpha il)$ yields $1+1+x\in N_F$ for $x=\frac{\brho(\balpha ii)\cdot \brho(\balpha kl)}{\brho(\balpha ik)\cdot\brho(\balpha il)}\in F^\times$, which shows that $F$ is near-idempotent.

 If $\omega_{J,i}\geq3$, then we can replace $\alpha$ as above by $\alpha-\epsilon_i+\epsilon_l$, which yields $\alpha+3\epsilon_i+\epsilon_k\leq\omega_J$. Thus by axiom \ref{WR3} we find the Pl\"ucker relation 
 \[
  \brho(\balpha ii)\cdot \brho(\balpha ik) \ + \ \brho(\balpha ii)\cdot \brho(\balpha ik) \ + \ \brho(\balpha ii)\cdot \brho(\balpha ik) \quad \in \quad N_F.
 \]
 Dividing all terms by $\brho(\balpha ii)\cdot\brho(\balpha ik)$ yields $1+1+1\in N_F$, which shows that $F$ is idempotent.
\end{proof}

\subsection{The up operator}\label{subsection: the up operator}

Let $N$ be the natural matroid of the reduction $\Jbar$. We denote by $E$ the ground set of $N$ and by $\theta\colon \Z^E \to \Z^n$ the associated projection. We also write $\theta$ for the map $E\to [n]$ obtained by identifying the standard basis vectors of $\Z^E$ and $\Z^n$ with the elements of $E$ and $[n]$, respectively.

The \emph{up operator} $\Up$ is given as follows. For a function $\brho\colon [n]^\rbar \to F$, we define the function $\Up \brho\colon E^\rbar \to F$ by taking 
\[
  \Up \brho (\balpha)
  =
  \begin{cases}
    \brho(\theta\balpha_1,\dots,\theta\balpha_{\rbar}) & \text{if } \balpha_1,\dots,\balpha_\rbar \text{ are pairwise distinct}, \\
    0 & \text{otherwise}. \\
  \end{cases}
\]

\begin{prop}\label{prop: up operator}
  Let $\brho\colon [n]^\rbar \to F$ be a function satisfying \ref{SR1} and \ref{SR2}. Then $\brho$ is a strong (resp.\ weak) $F$-representation of $J$ if and only if $\Up \brho$ is a strong (resp.\ weak) $F$-representation of $N$.
\end{prop}
\begin{proof}
  The support of $\Up \brho$ is the natural matroid $N$ of $\Jbar$. In addition, \ref{SR2} for $\rho$ implies \ref{SR2} for $\Up \brho$.

  Suppose that $\brho$ is a strong $F$-representation of $J$. The Pl\"ucker relations for $\Up\brho$ are \[ \sum_{k=0}^s (-1)^k \cdot \Up\brho(\balpha i_0\dotsc \widehat{i_k}\dotsc i_s) \cdot \Up\brho(\balpha i_k j_2\dotsc j_s) \quad \in \quad N_F, \] where $\balpha \in E^{\rbar-s}$ and $i_0,\dotsc,i_s,j_2,\dotsc,j_s \in E$ are chosen so that $\Sigma\balpha i_0\dotsc i_sj_2\dotsc j_s\leq\1$. Because $\theta(\Sigma \balpha i_0\dotsc i_sj_2\dotsc j_s) \leq \omega_J$, the Pl\"ucker relations for $\Up\brho$ are simply reformulations of the corresponding Pl\"ucker relations 
  \[ 
   \sum_{k=0}^s (-1)^k \cdot \brho(\theta(\balpha i_0\dotsc \widehat{i_k}\dotsc i_s)) \cdot \brho(\theta(\balpha i_k j_2\dotsc j_s)) \quad \in \quad N_F 
  \] 
  for $\brho$. Thus $\Up\brho$ is a strong $F$-representation of $N$. 
  
  Conversely, suppose that $\Up\brho$ is a strong $F$-representation of $N$. For any tuple $\balpha$ of elements in $[n]$, if $\Sigma\balpha \leq\omega_J$ then there is a tuple $\balpha'$ of distinct elements in $E$ so that $\theta\balpha'=\balpha$. Thus the Pl\"ucker relations for $\brho$ follow from those for $\Up\brho$, implying that $\brho$ is a strong $F$-representation of $J$.

  The same arugment works for weak $F$-representations.
\end{proof}

A tract $F$ is called \emph{perfect} if, for every strong $F$-representation $\brho$ of every matroid $M$, the $F$-vectors of $\brho$ are orthogonal to the $F$-covectors of $\brho$ (for details, see \cite[section 3.13]{Baker-Bowler19}). Examples of perfect tracts are fields, partial fields, $\K$, $\T$ and $\S$. The most important property of a perfect tract (from our perspective) is that every weak matroid representation over a perfect tract is strong \cite[Thm.\ 3.46]{Baker-Bowler19}. We extend this property to polymatroids in the following result.

\begin{cor}\label{cor: perfect tracts are excellent}
  Every perfect tract is excellent.
\end{cor}
\begin{proof}
  Every weak $F$-representation $\brho$ of a polymatroid $J$ lifts to a weak $F$-representation $\Up \brho$ of the natural matroid $N = N_{\Jbar}$. If $F$ is a perfect tract, then by \cite[Thm.\ 3.46]{Baker-Bowler19}, $\Up \brho$ is a strong $F$-representation of $N$. Hence, by \autoref{prop: up operator}, $\brho$ is a strong $F$-representation of $J$.
\end{proof}

\subsection{Simplified description of near-idempotent polymatroid representations}
\label{subsection: simplified description of idempotent polymatroid representations}

Due to \autoref{prop: idempotency principle}, only near-idempotent tracts $F$ possess proper polymatroid representations, in which case they can be described in an equivalent but simplified way. Namely, if $F$ is near-idempotent, we can identify a strong (or weak) $F$-representation $\brho\colon[n]^\rbar\to F$ of an M-convex set $J$ (with effective rank $\rbar$) with the function $\rho\colon\Delta^r_n\to F$ given by
\[
 \rho(\delta^-_J+\epsilon_{i_1}+\dotsc+\epsilon_{i_\rbar}) \ = \ \brho(i_1,\dotsc,i_\rbar),
\]
which does not depend on the ordering of $i_1,\dots,i_\rbar\in[n]$ due to \ref{SR2} (resp.\ \ref{WR2}) and the fact that $1=-1$ in $F$. Property \ref{SR1} (resp.\ \ref{WR1}) turns into the condition that the support of $\rho$ is $J$. More concisely, this formula identifies functions $\brho\colon[n]^\rbar\to F$ satisfying \ref{SR1} and \ref{SR2} (resp.\ \ref{WR1} and \ref{WR2}) with functions $\rho\colon\Delta^r_n\to F$ whose support is $J$.

The Pl\"ucker relations \ref{SR3} for $\brho$ turn into the relations
\[
 \sum_{k=0}^s \ \rho(\alpha -\epsilon_{i_k}+\epsilon_{i_0}+\dotsb+\epsilon_{i_s}) \cdot \rho(\alpha +\epsilon_{i_k}+\epsilon_{j_2}+\dotsb +\epsilon_{j_s}) \quad \in \quad N_F
\]
for all $2\leq s\leq r$, $\alpha\in\Delta^{r-s}_n$ with $\delta^-_J\leq\alpha$ and all $i_0,\dotsc,i_s,j_2,\dotsc,j_s\in [n]$ with $\alpha+\epsilon_{i_0}+\dotsc+\epsilon_{i_s}+\epsilon_{j_2}+\dotsc+\epsilon_{j_s}\leq\delta^+_J$. The $3$-term Pl\"ucker relations \ref{WR3} turn into the relations
\begin{multline*}
 \rho(\alpha+\epsilon_j+\epsilon_k) \cdot \rho(\alpha +\epsilon_i+\epsilon_l) \ + \ \rho(\alpha +\epsilon_i+\epsilon_k) \cdot \rho(\alpha +\epsilon_j+\epsilon_l) \\ + \ \rho(\alpha +\epsilon_i+\epsilon_j) \cdot \rho(\alpha +\epsilon_k+\epsilon_l) \quad \in \quad N_F
\end{multline*}
for all $\alpha\in\Delta^{r-2}_n$ with $\delta^-_J\leq\alpha$ and all $i,j,k,l\in [n]$ with $\alpha+\epsilon_i+\epsilon_j+\epsilon_k+\epsilon_l\leq\delta_J^+$.

We can extend this alternative perspective on polymatroid representations to arbitrary tracts in the following way.

\begin{lemma}\label{lemma: alternative description of polymatroid representations}
 Let $F$ be a tract, $J\subseteq\Delta_n^r$ an M-convex set of effective rank $\rbar=r-\norm{\delta^-_J}$, and $\brho\colon[n]^\rbar\to F$ a function that satisfies \ref{SR1} and \ref{SR2}. Let $\rho\colon\Delta^r_n\to F$ be the function with support $J$ given by $\rho(\delta^-_J+\epsilon_{i_1}+\dotsc+\epsilon_{i_\rbar})=\brho(i_1,\dotsc,i_\rbar)$ whenever $i_1\leq\dotsc\leq i_\rbar$. Then $\brho$ is a strong $F$-representation of $J$ if and only if $\rho$ satisfies the Pl\"ucker relations
\begin{equation*}
 \sum_{k=0}^s \ (-1)^{k+\sigma(k)}\cdot \rho(\alpha -\epsilon_{i_k}+\epsilon_{i_0}+\dotsb+\epsilon_{i_s}) \cdot \rho(\alpha +\epsilon_{i_k}+\epsilon_{j_2}+\dotsb +\epsilon_{j_s})  \in  N_{F}
\end{equation*}
for all $2\leq s\leq r$, $\alpha\in\Delta^{r-s}_n$, $1\leq i_0\leq \dotsc\leq i_s\leq n$, and $1\leq j_2\leq\dotsc\leq j_s\leq n$ such that $\delta^-_J\leq\alpha$ and $\alpha+\epsilon_{i_0}+\dotsc+\epsilon_{i_s}+\epsilon_{j_2}+\dotsc+\epsilon_{j_s}\leq\delta^+_J$, where $\sigma(k)$ is the number of $k\in\{2,\ldots,s\}$ with $i_k<j_s$.

The function $\brho$ is a weak $F$-representation of $J$ if and only if $\rho$ satisfies the $3$-term Pl\"ucker relations
\begin{multline*}
 \rho(\alpha+\epsilon_j+\epsilon_k) \cdot \rho(\alpha +\epsilon_i+\epsilon_l) \ - \ \rho(\alpha +\epsilon_i+\epsilon_k) \cdot \rho(\alpha +\epsilon_j+\epsilon_l) \\ + \ \rho(\alpha +\epsilon_i+\epsilon_j) \cdot \rho(\alpha +\epsilon_k+\epsilon_l)  \quad \in  \quad N_{F}
\end{multline*}
for all $\alpha\in\Delta^{r-2}_n$ and $1\leq i\leq j\leq k\leq l\leq n$ such that $\delta^-_J\leq\alpha$ and $\alpha+\epsilon_i+\epsilon_j+\epsilon_k+\epsilon_l\leq\delta^+_J$.
\end{lemma}

\begin{proof}
 If $F$ is near-idempotent, then $-1=1$ and the requirements on the ordering of the $i_k$ and $j_\ell$ become irrelevant. Thus the claim reduces to the discussion in the beginning of this section.

 If $F$ is not near-idempotent, then $J$ is a matroid by the idempotency principle (\autoref{prop: idempotency principle}) and thus $\delta^+_J-\delta^-_J\leq\1$ by \autoref{lemma: characterizations of translates of matroids}. Since the defining relations of $\brho$ and $\rho$ are translation invariant, we can assume that $\delta^-_J=0$ and thus $\delta^+_J\leq\1$. This means that $\brho(i'_1,\dotsc,i'_r)=0$ if the $i'_k$ are not pairwise distinct. It also means that the Pl\"ucker relations for $s$, $\alpha$, $1\leq i_0\leq \dotsc\leq i_s\leq n$, and $1\leq j_2\leq\dotsc\leq j_s\leq n$ are trivial (i.e., all terms are $0$ or it is of the form $a-a\in N_F$) unless $1\leq i_0< \dotsc<i_s\leq n$ and $1\leq j_2<\dotsc< j_s\leq n$ (resp.\ $i<j<k<l$ in the case of $3$-term Pl\"ucker relations). In the case of \ref{SR3}, we can further assume that $s=r$ and $\alpha=0$, since if $i_1=j_1,\dotsc,i_{r-s}=j_{r-s}$, the corresponding Pl\"ucker relation is equal to that for $\alpha=\epsilon_{i_1}+\dotsb+\epsilon_{i_{r-s}}$ (a fact which is particular to matroids and does not generalize to polymatroids).

 Thanks to these simplifications, the Pl\"ucker relations assume their usual shape (for instance, cf.\ \cite[Def.\ 3.1]{Baker-Lorscheid21b}\footnote{Note that the factor $(-1)^{\sigma(k)}$ stems from the permutation of $(i_k,j_2,\dotsc,j_s)$ that brings the coefficients into increasing order. This factor is missing in the Pl\"ucker relations in \cite{Baker-Lorscheid21b}---a mistake that requires correction.}), which reduces the lemma to the equivalence between (strong) matroid representations as alternating functions with domain $[n]^r$ and functions with domain $\binom{[n]}r$.
 The case of the $3$-term Pl\"ucker relations \ref{WR3} can be established in a similar vein.
\end{proof}

\subsection{M-convex functions as representations over the tropical hyperfield}
\label{subsection: M-convex functions as representations over the tropical hyperfield}

In this subsection, we show that a $\T_0$-representation is essentially the same thing as an \emph{M-convex function} in the sense of Murota, which is a function $f\colon\Z^n\to \R\cup\{\infty\}$ with nonempty support $J=\{\alpha\in\Z^n\mid f(\alpha)\neq\infty\}$ that satisfies the following \emph{exchange axiom}: for $\alpha,\beta\in J$ and $k\in[n]$ with $\alpha_k>\beta_k$, there is an $l\in[n]$ with $\alpha_l<\beta_l$ and
\begin{equation}\label{eq: Murota's exchange relations}
 f(\alpha) \ + \ f(\beta) \ \geq \ f(\alpha-\epsilon_k+\epsilon_l) \ + \ f(\beta+\epsilon_k-\epsilon_l).
\end{equation}
It follows from this exchange axiom that $J$ is an M-convex set (\cite[Prop.\ 6.1]{Murota03}). Note that an M-convex function whose support $J$ is a matroid is the same thing as a valuated matroid.

In the following, we identify $\T_0$-representations of $J$ with functions $\rho\colon\Delta^r_n\to\T_0$ that have support $J$ and satisfy the appropriate version of the Pl\"ucker relations; cf.\ \autoref{subsection: simplified description of idempotent polymatroid representations} for details.

\begin{prop}\label{prop: M-convex functions are strong T_0-representations}
 Let $J$ be an M-convex set and $\rho\colon\Delta^r_n\to\T_0=\R_{\geq0}$ a function with support $J$. Then $\rho$ is a strong $\T_0$-representation of $J$ if and only if $f=-\log(\rho)$ is M-convex.
\end{prop}

\begin{proof}
 Assume that $\rho$ is a strong $\T_0$-representation of $J$. Consider $\beta,\gamma\in J$ with $\beta_k<\gamma_k$ for some $k\in[n]$. We need to show that there exists an $l\in[n]$ with $\beta_l>\gamma_l$ and $f(\beta)+f(\gamma)\geq f(\beta+\epsilon_k-\epsilon_l)+f(\gamma-\epsilon_k+\epsilon_l)$.
 
 Let $\alpha=\inf\{\beta,\gamma\}\geq\delta_J^-$. We have $\alpha\in\Delta^{r-s}_n$ for some $1 \leq s \leq r$. If $s=1$, then $\gamma=\beta+\epsilon_k-\epsilon_l$ for some $l\in[n]$ and the exchange property we are trying to show is trivially satisfied. Thus we can assume that $s\in\{2,\ldots,r\}$. There are $i_1,\dotsc,i_s,j_1,\dotsc,j_s\in[n]$ (unique up to permutation) such that 
 \[
  \beta = \alpha+\epsilon_{i_1}+\dotsb+\epsilon_{i_s}, \ \ \gamma = \alpha+\epsilon_{j_1}+\dotsb+\epsilon_{j_s} \ \ \text{and} \ \ \{i_1,\dotsc,i_s\}\cap\{j_1,\dotsc,j_s\} = \emptyset,
 \]
 and therefore $\alpha+\epsilon_{i_1}+\dotsb+\epsilon_{i_s}+\epsilon_{j_1}+\dotsb+\epsilon_{j_s}  \leq  \sup\,\{\beta,\gamma\}  \leq  \delta^+_J$. Assume without loss of generality that $k=j_1$ and define $i_0=k$. Then the Pl\"ucker relation
 \[
 \sum_{t=0}^s \ \rho(\alpha+\epsilon_{i_0}+\dotsb\widehat{\epsilon_{i_t}}\dotsb+\epsilon_{i_s}) \ \cdot \ \rho(\alpha+\epsilon_{i_t}+\epsilon_{j_2}+\dotsb+\epsilon_{j_s}) \ \in \ N_\T
 \]
 assumes its maximum twice. Thus there is an $m>0$ such that
 \begin{multline*}
  \rho(\beta)\rho(\gamma) \ = \ \rho_J(\alpha+\epsilon_{i_1}+\dotsb+\epsilon_{i_s})\rho(\alpha+\epsilon_{i_0}+\epsilon_{j_2}+\dotsb+\epsilon_{j_s}) \\
  \leq \ \rho(\alpha+\epsilon_{i_0}+\dotsb\widehat{\epsilon_{i_m}}\dotsb+\epsilon_{i_s}) \rho(\alpha+\epsilon_{i_m}+\epsilon_{j_2}+\dotsb+\epsilon_{j_s}) \ = \ \rho(\beta+\epsilon_k-\epsilon_l) \rho(\gamma-\epsilon_k+\epsilon_l).
 \end{multline*}
 for $l=i_m$. Since $i_m\notin\{j_1,\dotsc,j_s\}$, we have $\gamma_l<\beta_l$. Applying $-\log$ to both sides turns the products into sums and reverses the inequality, which establishes the desired exchange axiom and shows that $f=-\log(\rho)$ is M-convex.
 
 Conversely, assume that $f$ is M-convex. We show that the exchange axiom \eqref{eq: Murota's exchange relations} for M-convex functions implies the Pl\"ucker relations \ref{SR3} for all $\alpha\in\Delta^{\rbar-s}_n$ and $i_0,\dotsc,i_s,j_2,\dotsc,j_s\in[n]$, i.e., the maximum appears at least twice in the formal sum
 \[
  \sum_{k=0}^s \ \rho(\alpha+\epsilon_{i_0}+\dotsc\widehat{\epsilon_{i_k}}\dotsc+\epsilon_{i_s}) \cdot \rho(\alpha+\epsilon_{i_k}+\epsilon_{j_2}\dotsc+\epsilon_{j_s}).
 \]
 Fix a $k$ such that $\rho(\beta)\cdot\rho(\gamma)$ assumes the maximum among these terms where
 \[
  \beta \ = \ \alpha+\epsilon_{i_0}+\dotsc\widehat{\epsilon_{i_k}}\dotsc+\epsilon_{i_s} \qquad \text{and} \qquad \gamma \ = \ \alpha+\epsilon_{i_k}+\epsilon_{j_2}\dotsc+\epsilon_{j_s}.
 \]
 If $\beta_{i_k}\geq\gamma_{i_k}\geq1$, then there exists an $l\neq k$ with $i_l=i_k$, and therefore also 
 \begin{multline*}
  \rho (\alpha+\epsilon_{i_0}+\dotsc\widehat{\epsilon_{i_l}}\dotsc+\epsilon_{i_s}) \cdot \rho(\alpha+\epsilon_{i_l}+\epsilon_{j_2}\dotsc+\epsilon_{j_s}) \\ = \ \rho (\alpha+\epsilon_{i_0}+\dotsc\widehat{\epsilon_{i_k}}\dotsc+\epsilon_{i_s}) \cdot \rho(\alpha+\epsilon_{i_k}+\epsilon_{j_2}\dotsc+\epsilon_{j_s})
 \end{multline*}
 assumes the maximum. If $\beta_{i_k}<\gamma_{i_k}$, then the exchange axiom \eqref{eq: Murota's exchange relations} implies that there is an $l\neq k$ such that 
 \[
  f(\beta) \ + \ f(\gamma) \ \geq \ f(\beta-\epsilon_{i_k}+\epsilon_{i_l}) \ + \ f(\gamma+\epsilon_{i_k}-\epsilon_{i_l})
 \]
 or, equivalently, 
 \begin{multline*}
  \rho (\alpha+\epsilon_{i_0}+\dotsc\widehat{\epsilon_{i_l}}\dotsc+\epsilon_{i_s}) \cdot \rho(\alpha+\epsilon_{i_l}+\epsilon_{j_2}\dotsc+\epsilon_{j_s}) \\ \geq \ \rho (\alpha+\epsilon_{i_0}+\dotsc\widehat{\epsilon_{i_k}}\dotsc+\epsilon_{i_s}) \cdot \rho(\alpha+\epsilon_{i_k}+\epsilon_{j_2}\dotsc+\epsilon_{j_s}).
 \end{multline*}
 By the maximality of the latter term, this inequality must be an equality, which exhibits also in this case a second maximal term in the Pl\"ucker relation under consideration. This shows that $\rho$ satisfies \ref{SR3}, which concludes the proof.
\end{proof}

Murota shows in \cite[Thm.\ 6.4]{Murota03} that M-convex functions are characterized by the following \emph{local exchange axiom}: for $\alpha\in\Delta^{r-2}_n$ and all $i,j,k,l\in[n]$ such that $\{i,k\}\cap\{j,l\}=\emptyset$,
\begin{multline}\label{eq: Murota's 3-term relations}
 f(\alpha+\epsilon_i+\epsilon_k) + f(\alpha+\epsilon_j+\epsilon_l) \\ \geq \ \min\big\{ f(\alpha+\epsilon_i+\epsilon_j)+ f(\alpha+\epsilon_k+\epsilon_l), \ f(\alpha+\epsilon_i+\epsilon_l)+ f(\alpha+\epsilon_j+\epsilon_k) \big\}.
\end{multline}
We present a new proof of this fact based on our result that $\T_0$ is excellent (\autoref{cor: perfect tracts are excellent}).

\begin{prop}\label{prop: M-convex functions are characterized by the local exchange axiom}
 Let $J$ be an M-convex set and $\rho\colon\Delta^r_n\to\T_0=\R_{\geq0}$ a function with support $J$. If $f=-\log(\rho)$ satisfies the local exchange axiom \eqref{eq: Murota's 3-term relations}, then $f$ is M-convex.
\end{prop}

\begin{proof}
 Assume that $f$ satisfies the local exchange axiom \eqref{eq: Murota's 3-term relations}. We claim that this implies that $\rho\colon\Delta^r_n\to\T_0=\R_{\geq0}$ satisfies the $3$-term Pl\"ucker relations \ref{WR3}
 \begin{multline*}
  \rho(\alpha+\epsilon_i+\epsilon_j)\cdot \rho(\alpha+\epsilon_k+\epsilon_l) \ + \ \rho(\alpha+\epsilon_i+\epsilon_k)\cdot \rho(\alpha+\epsilon_j+\epsilon_l) \\
  + \ \rho(\alpha+\epsilon_i+\epsilon_l)\cdot \rho(\alpha+\epsilon_j+\epsilon_k) \ \in \ N_{\T_0}
 \end{multline*}
 for all $\alpha\in\Delta^{r-2}_n$ and $i,j,k,l\in[n]$. 
 
 Assume that $\rho(\alpha+\epsilon_i+\epsilon_k)\cdot \rho(\alpha+\epsilon_j+\epsilon_l)$ is the largest among the $3$ terms in this Pl\"ucker relation. If $\{i,k\}\cap\{j,l\}\neq\emptyset$, then exchanging the two equal indices does not change the argument of the functions, which means that there is a second term in this Pl\"ucker relation equal to $\rho(\alpha+\epsilon_i+\epsilon_k)\cdot \rho(\alpha+\epsilon_j+\epsilon_l)$, which means that the maximum is achieved twice and therefore the $3$-term sum is in $N_{\T_0}$, as desired.
 
 If $\{i,k\}\cap\{j,l\}=\emptyset$, then the local exchange axiom implies that 
 \begin{multline*}
  \rho(\alpha+\epsilon_i+\epsilon_k)\cdot \rho(\alpha+\epsilon_j+\epsilon_l) \\ \leq \ \max\big\{ \rho(\alpha+\epsilon_i+\epsilon_j)\cdot \rho(\alpha+\epsilon_k+\epsilon_l), \ \rho(\alpha+\epsilon_i+\epsilon_l)\cdot \rho(\alpha+\epsilon_j+\epsilon_k) \big\},
 \end{multline*}
 which once again implies that the maximum is achieved twice, and thus that the $3$-term sum is in $N_{\T_0}$, as desired.
 
 This verifies our claim and shows that $\rho$ is a weak $\T_0$-representation. Since $\T_0$ is perfect (\cite[Cor.\ 3.45]{Baker-Bowler19}) and thus excellent (\autoref{cor: perfect tracts are excellent}), $\rho$ is, in fact, a strong $\T_0$-representation. It follows from \autoref{prop: M-convex functions are strong T_0-representations} that $f$ is M-convex.
\end{proof}

\subsection{Hives}
\label{subsection: hives}

As mentioned in the introduction, hives are combinatorial gadgets that were introduced by Knutson and Tao in \cite{Knutson-Tao99}, and they are naturally in bijection with $\T_0$-representations of $\Delta^r_3$. We explain the concept of a hive here and give an illustrative example, mainly following \cite{Buch00}.
 
The $r^{\rm th}$ \emph{hive triangle} is the triangular array depicted in \autoref{fig: Hive Triangle}, consisting of $r+1$ ``hive vertices'' on each side and $r^2$ ``small triangles''.
 \begin{figure}[!ht] 
  \begin{tikzpicture}[line cap=round,line join=round,x=1.0cm,y=1.0cm]
\clip(-0.43,-0.25) rectangle (4.64,3.96);
\draw [line width=1.2pt] (0,0)-- (0.5,0.87);
\draw [line width=1.2pt] (0.5,0.87)-- (1,1.73);
\draw [line width=1.2pt] (1,1.73)-- (1.5,2.6);
\draw [line width=1.2pt] (1.5,2.6)-- (2,3.46);
\draw [line width=1.2pt] (2,3.46)-- (2.5,2.6);
\draw [line width=1.2pt] (2.5,2.6)-- (1.5,2.6);
\draw [line width=1.2pt] (1.5,2.6)-- (2,1.73);
\draw [line width=1.2pt] (2,1.73)-- (2.5,2.6);
\draw [line width=1.2pt] (2.5,2.6)-- (3,1.73);
\draw [line width=1.2pt] (3,1.73)-- (2,1.73);
\draw [line width=1.2pt] (2,1.73)-- (1,1.73);
\draw [line width=1.2pt] (1,1.73)-- (1.5,0.87);
\draw [line width=1.2pt] (1.5,0.87)-- (2,1.73);
\draw [line width=1.2pt] (2,1.73)-- (2.5,0.87);
\draw [line width=1.2pt] (2.5,0.87)-- (3,1.73);
\draw [line width=1.2pt] (3,1.73)-- (3.5,0.87);
\draw [line width=1.2pt] (3.5,0.87)-- (2.5,0.87);
\draw [line width=1.2pt] (2.5,0.87)-- (1.5,0.87);
\draw [line width=1.2pt] (1.5,0.87)-- (0.5,0.87);
\draw [line width=1.2pt] (0.5,0.87)-- (1,0);
\draw [line width=1.2pt] (1,0)-- (1.5,0.87);
\draw [line width=1.2pt] (1.5,0.87)-- (2,0);
\draw [line width=1.2pt] (2,0)-- (2.5,0.87);
\draw [line width=1.2pt] (2.5,0.87)-- (3,0);
\draw [line width=1.2pt] (3,0)-- (3.5,0.87);
\draw [line width=1.2pt] (3.5,0.87)-- (4,0);
\draw [line width=1.2pt] (4,0)-- (3,0);
\draw [line width=1.2pt] (3,0)-- (2,0);
\draw [line width=1.2pt] (2,0)-- (1,0);
\draw [line width=1.2pt] (1,0)-- (0,0);
\fill [color=black] (0,0) circle (2.5pt);
\fill [color=black] (1,0) circle (2.5pt);
\fill [color=black] (0.5,0.87) circle (2.5pt);
\fill [color=black] (2,0) circle (2.5pt);
\fill [color=black] (1.5,0.87) circle (2.5pt);
\fill [color=black] (3,0) circle (2.5pt);
\fill [color=black] (2.5,0.87) circle (2.5pt);
\fill [color=black] (4,0) circle (2.5pt);
\fill [color=black] (3.5,0.87) circle (2.5pt);
\fill [color=black] (1,1.73) circle (2.5pt);
\fill [color=black] (2,1.73) circle (2.5pt);
\fill [color=black] (3,1.73) circle (2.5pt);
\fill [color=black] (1.5,2.6) circle (2.5pt);
\fill [color=black] (2.5,2.6) circle (2.5pt);
\fill [color=black] (2,3.46) circle (2.5pt);
\draw [
    thick,
    decoration={
        brace,
        mirror,
        raise=0.5cm
    },
    decorate
] (4,0)-- (2,3.46);
\draw[color=black] (4.14,2.36) node {$r+1$};
\end{tikzpicture}
  \caption{The $r^{\rm th}$ hive triangle.}
  \label{fig: Hive Triangle}
 \end{figure}
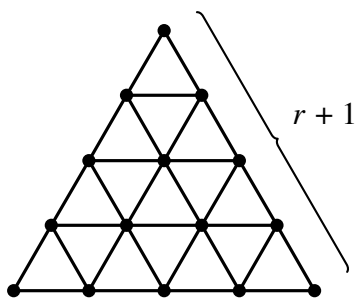

A \emph{rhombus} in the hive triangle is the union of two small triangles which share a common edge. 
There are three combinatorial types or orientations of rhombi, as depicted in \autoref{fig: Rhombi}.
 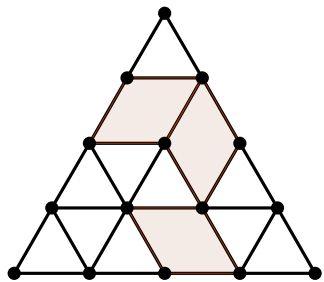
\begin{figure}[!ht] 
  \definecolor{zzttqq}{rgb}{0.6,0.2,0}
\begin{tikzpicture}[line cap=round,line join=round,x=1.0cm,y=1.0cm]
\clip(-0.43,-0.25) rectangle (4.44,3.96);
\fill[color=zzttqq,fill=zzttqq,fill opacity=0.1] (1,1.73) -- (2,1.73) -- (2.5,2.6) -- (1.5,2.6) -- cycle;
\fill[color=zzttqq,fill=zzttqq,fill opacity=0.1] (2,1.73) -- (2.5,0.87) -- (3,1.73) -- (2.5,2.6) -- cycle;
\fill[color=zzttqq,fill=zzttqq,fill opacity=0.1] (2,0) -- (3,0) -- (2.5,0.87) -- (1.5,0.87) -- cycle;
\draw [line width=1.2pt] (0,0)-- (0.5,0.87);
\draw [line width=1.2pt] (0.5,0.87)-- (1,1.73);
\draw [line width=1.2pt] (1,1.73)-- (1.5,2.6);
\draw [line width=1.2pt] (1.5,2.6)-- (2,3.46);
\draw [line width=1.2pt] (2,3.46)-- (2.5,2.6);
\draw [line width=1.2pt] (2.5,2.6)-- (1.5,2.6);
\draw [line width=1.2pt] (2,1.73)-- (2.5,2.6);
\draw [line width=1.2pt] (2.5,2.6)-- (3,1.73);
\draw [line width=1.2pt] (2,1.73)-- (1,1.73);
\draw [line width=1.2pt] (1,1.73)-- (1.5,0.87);
\draw [line width=1.2pt] (1.5,0.87)-- (2,1.73);
\draw [line width=1.2pt] (2,1.73)-- (2.5,0.87);
\draw [line width=1.2pt] (2.5,0.87)-- (3,1.73);
\draw [line width=1.2pt] (3,1.73)-- (3.5,0.87);
\draw [line width=1.2pt] (3.5,0.87)-- (2.5,0.87);
\draw [line width=1.2pt] (2.5,0.87)-- (1.5,0.87);
\draw [line width=1.2pt] (1.5,0.87)-- (0.5,0.87);
\draw [line width=1.2pt] (0.5,0.87)-- (1,0);
\draw [line width=1.2pt] (1,0)-- (1.5,0.87);
\draw [line width=1.2pt] (1.5,0.87)-- (2,0);
\draw [line width=1.2pt] (3,0)-- (2.5,0.87);
\draw [line width=1.2pt] (3,0)-- (3.5,0.87);
\draw [line width=1.2pt] (3.5,0.87)-- (4,0);
\draw [line width=1.2pt] (4,0)-- (3,0);
\draw [line width=1.2pt] (3,0)-- (2,0);
\draw [line width=1.2pt] (2,0)-- (1,0);
\draw [line width=1.2pt] (1,0)-- (0,0);
\draw [color=zzttqq] (1,1.73)-- (2,1.73);
\draw [color=zzttqq] (2,1.73)-- (2.5,2.6);
\draw [color=zzttqq] (2.5,2.6)-- (1.5,2.6);
\draw [color=zzttqq] (1.5,2.6)-- (1,1.73);
\draw [color=zzttqq] (2,1.73)-- (2.5,0.87);
\draw [color=zzttqq] (2.5,0.87)-- (3,1.73);
\draw [color=zzttqq] (3,1.73)-- (2.5,2.6);
\draw [color=zzttqq] (2.5,2.6)-- (2,1.73);
\draw [color=zzttqq] (2,0)-- (3,0);
\draw [color=zzttqq] (3,0)-- (2.5,0.87);
\draw [color=zzttqq] (1.5,0.87)-- (2.5,0.87);
\draw [color=zzttqq] (1.5,0.87)-- (2,0);
\fill [color=black] (0,0) circle (2.5pt);
\fill [color=black] (1,0) circle (2.5pt);
\fill [color=black] (0.5,0.87) circle (2.5pt);
\fill [color=black] (2,0) circle (2.5pt);
\fill [color=black] (1.5,0.87) circle (2.5pt);
\fill [color=black] (3,0) circle (2.5pt);
\fill [color=black] (2.5,0.87) circle (2.5pt);
\fill [color=black] (4,0) circle (2.5pt);
\fill [color=black] (3.5,0.87) circle (2.5pt);
\fill [color=black] (1,1.73) circle (2.5pt);
\fill [color=black] (2,1.73) circle (2.5pt);
\fill [color=black] (3,1.73) circle (2.5pt);
\fill [color=black] (1.5,2.6) circle (2.5pt);
\fill [color=black] (2.5,2.6) circle (2.5pt);
\fill [color=black] (2,3.46) circle (2.5pt);
\end{tikzpicture}
  \caption{The three combinatorial types of rhombi.}
  \label{fig: Rhombi}
 \end{figure}
 
Each rhombus has two acute angles and two obtuse angles.
Let $H$ be the set of hive vertices and $\R^H$ the set of labelings of $H$ by real numbers.
Each rhombus gives rise to an inequality on $\R^H$ saying that the sum of the labels
at the obtuse vertices must be greater than or equal to the sum of the labels at the
acute vertices. A \emph{hive} is a labeling in $\R^H$ that satisfies all rhombus inequalities. 
Of particular interest, in terms of the connection to the representation theory of ${\rm GL}_r$, are the \emph{integral hives}, which are hives for which all labels are integers. 

\begin{ex} \label{ex:hive}
The rhombus inequalities say that the labeling given in \autoref{fig: Hive Example} is a hive if and only if $4 \leq x \leq 5$.
 \begin{figure}[!ht]
  \begin{tikzpicture}[line cap=round,line join=round,x=1.0cm,y=1.0cm]
\clip(-0.8,-0.5) rectangle (3.7,3.2);
\draw [line width=1.2pt] (0,0)-- (0.5,0.87);
\draw [line width=1.2pt] (0.5,0.87)-- (1,1.73);
\draw [line width=1.2pt] (1,1.73)-- (1.5,2.6);
\draw [line width=1.2pt] (1.5,2.6)-- (2,1.73);
\draw [line width=1.2pt] (2,1.73)-- (1,1.73);
\draw [line width=1.2pt] (1,1.73)-- (1.5,0.87);
\draw [line width=1.2pt] (1.5,0.87)-- (2,1.73);
\draw [line width=1.2pt] (2,1.73)-- (2.5,0.87);
\draw [line width=1.2pt] (2.5,0.87)-- (1.5,0.87);
\draw [line width=1.2pt] (1.5,0.87)-- (0.5,0.87);
\draw [line width=1.2pt] (0.5,0.87)-- (1,0);
\draw [line width=1.2pt] (1,0)-- (1.5,0.87);
\draw [line width=1.2pt] (1.5,0.87)-- (2,0);
\draw [line width=1.2pt] (2,0)-- (2.5,0.87);
\draw [line width=1.2pt] (2.5,0.87)-- (3,0);
\draw [line width=1.2pt] (3,0)-- (2,0);
\draw [line width=1.2pt] (2,0)-- (1,0);
\draw [line width=1.2pt] (1,0)-- (0,0);
\fill [color=black] (0,0) circle (2.5pt);
\draw[color=black] (-0.32,0.27) node {6};
\fill [color=black] (1,0) circle (2.5pt);
\draw[color=black] (1,-0.35) node {6};
\fill [color=black] (0.5,0.87) circle (2.5pt);
\draw[color=black] (0.18,1.15) node {5};
\fill [color=black] (2,0) circle (2.5pt);
\draw[color=black] (2,-0.35) node {5};
\fill [color=black] (1.5,0.87) circle (2.5pt);
\draw[color=black] (1.85,1.1) node {$x$};
\fill [color=black] (3,0) circle (2.5pt);
\draw[color=black] (3.32,0.27) node {3};
\fill [color=black] (2.5,0.87) circle (2.5pt);
\draw[color=black] (2.82,1.15) node {3};
\fill [color=black] (1,1.73) circle (2.5pt);
\draw[color=black] (0.68,2.01) node {3};
\fill [color=black] (2,1.73) circle (2.5pt);
\draw[color=black] (2.32,2.01) node {2};
\fill [color=black] (1.5,2.6) circle (2.5pt);
\draw[color=black] (1.5,3) node {0};
\end{tikzpicture}
  \caption{A labeling of the $3^{\rm rd}$ hive triangle.}
  \label{fig: Hive Example}
 \end{figure}
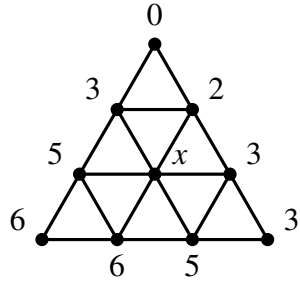
\end{ex}

We coordinatize the hive triangle by letting $(r,0,0)$ denote the lower-left corner, $(0,r,0)$ denote the lower-right corner, and $(0,0,r)$ denote the top corner; see \autoref{fig: Hive Coordinates}. This identifies $H$ with $\Delta^r_3$.

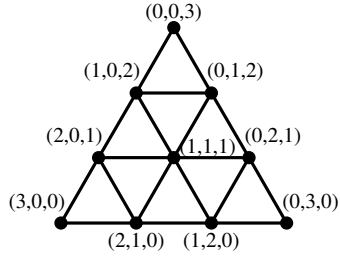
\begin{figure}[!ht]
 \begin{tikzpicture}[line cap=round,line join=round,x=1.0cm,y=1.0cm]
\clip(-0.8,-0.5) rectangle (3.7,3);
\draw [line width=1.2pt] (0,0)-- (0.5,0.87);
\draw [line width=1.2pt] (0.5,0.87)-- (1,1.73);
\draw [line width=1.2pt] (1,1.73)-- (1.5,2.6);
\draw [line width=1.2pt] (1.5,2.6)-- (2,1.73);
\draw [line width=1.2pt] (2,1.73)-- (1,1.73);
\draw [line width=1.2pt] (1,1.73)-- (1.5,0.87);
\draw [line width=1.2pt] (1.5,0.87)-- (2,1.73);
\draw [line width=1.2pt] (2,1.73)-- (2.5,0.87);
\draw [line width=1.2pt] (2.5,0.87)-- (1.5,0.87);
\draw [line width=1.2pt] (1.5,0.87)-- (0.5,0.87);
\draw [line width=1.2pt] (0.5,0.87)-- (1,0);
\draw [line width=1.2pt] (1,0)-- (1.5,0.87);
\draw [line width=1.2pt] (1.5,0.87)-- (2,0);
\draw [line width=1.2pt] (2,0)-- (2.5,0.87);
\draw [line width=1.2pt] (2.5,0.87)-- (3,0);
\draw [line width=1.2pt] (3,0)-- (2,0);
\draw [line width=1.2pt] (2,0)-- (1,0);
\draw [line width=1.2pt] (1,0)-- (0,0);
\begin{tiny}
\fill [color=black] (0,0) circle (2.5pt);
\draw[color=black] (-0.32,0.27) node {(3,0,0)};
\fill [color=black] (1,0) circle (2.5pt);
\draw[color=black] (1,-0.25) node {(2,1,0)};
\fill [color=black] (0.5,0.87) circle (2.5pt);
\draw[color=black] (0.18,1.15) node {(2,0,1)};
\fill [color=black] (2,0) circle (2.5pt);
\draw[color=black] (2,-0.25) node {(1,2,0)};
\fill [color=black] (1.5,0.87) circle (2.5pt);
\draw[color=black] (1.95,1) node {(1,1,1)};
\fill [color=black] (3,0) circle (2.5pt);
\draw[color=black] (3.32,0.27) node {(0,3,0)};
\fill [color=black] (2.5,0.87) circle (2.5pt);
\draw[color=black] (2.82,1.15) node {(0,2,1)};
\fill [color=black] (1,1.73) circle (2.5pt);
\draw[color=black] (0.68,2.01) node {(1,0,2)};
\fill [color=black] (2,1.73) circle (2.5pt);
\draw[color=black] (2.32,2.01) node {(0,1,2)};
\fill [color=black] (1.5,2.6) circle (2.5pt);
\draw[color=black] (1.5,2.8) node {(0,0,3)};
\end{tiny}
\end{tikzpicture}
 \caption{Coordinates on the $3^{\rm rd}$ hive triangle.}
 \label{fig: Hive Coordinates}
\end{figure}

With this coordinatization, the following observation of Br\"and\'en \cite[Section 4]{Branden10} becomes a direct translation of 
\autoref{prop: M-convex functions are strong T_0-representations} into the language of hives. Recall from \autoref{subsection: more examples of tracts} the definition of the discrete tropical hyperfield as the tract $\T^\Z_0=\{0\}\cup\{e^i\mid i\in\Z\}$ with null set 
\[
 N_{\T^\Z_0} \ = \ \big\{ a_1+\dotsb+a_n \, \big| \, a_1,\dotsc,a_n\text{ assumes its maximum at least twice}\big\}.
\]

\begin{prop} \label{prop:hive}
 A function $\rho : H \to \R_{>0}$ is a $\T_0$-representation of $\Delta^r_3$ if and only if $\log \rho : H \to \R$ is a hive. A function $\rho : H \to e^\Z$ is a $\T^\Z_0$-representation of $\Delta^r_3$ if and only if $\log \rho : H \to \R$ is an integral hive.
\end{prop}

This allows us to translate the key results of \cite{Buch00} into the language of polymatroid representations.

Given an integer partition $\lambda$ with at most $r$ parts, we let $\lambda_1,\ldots,\lambda_k$ denote the parts in weakly decreasing order, i.e., $\lambda_1,\ldots,\lambda_k$ are integers with $k \leq r$ and $\lambda_1 \geq \lambda_2 \geq \cdots \geq \lambda_k \geq 1$, and we set $\lambda_{k+1}=\cdots = \lambda_r = 0$.
We denote by $|\lambda| := \lambda_1 + \cdots + \lambda_k$ the integer being partitioned by $\lambda$, 
and we let $V_\lambda$ denote the unique irreducible representation of ${\rm GL}_r$ with highest weight $\lambda$. 

Given three such partitions $\lambda,\mu, \nu$ with $|\nu| = |\lambda| + |\mu|$, we denote by $c^\nu_{\lambda \mu}$ the corresponding Littlewood--Richardson coefficient, i.e., the multiplicity of the representation $V_\nu$ in $V_{\lambda} \otimes V_{\mu}$.

The following result is, in a certain precise sense, equivalent to the celebrated Littlewood--Richardson rule, cf.~\cite[Appendix A]{Buch00}.

\begin{thm}[Knutson--Tao] \label{thm:Knutson-Tao}
Let $\lambda,\mu, \nu$ be integer partitions with at most $r$ parts such that $|\nu| = |\lambda| + |\mu|$. Then the Littlewood--Richardson coefficient $c^\nu_{\lambda \mu}$ is equal to the number of representations $\rho:\Delta^r_3\to\T^\Z_0$ with logarithmic values $(\lambda,\mu,\nu)$ on the ``border'' of $\Delta^r_3$ as in \autoref{fig: Littlewood-Richardson}.
 \begin{figure}[!ht]
  \begin{tikzpicture}[line cap=round,line join=round,x=1.0cm,y=1.0cm]
\clip(-2.4,-2.5) rectangle (5,4);
\draw [line width=1.2pt] (0,0)-- (0.5,0.87);
\draw [line width=1.2pt] (0.5,0.87)-- (1,1.73);
\draw [line width=1.2pt] (1,1.73)-- (1.5,2.6);
\draw [line width=1.2pt] (1.5,2.6)-- (2,3.46);
\draw [line width=1.2pt] (2,3.46)-- (2.5,2.6);
\draw [line width=1.2pt] (2.5,2.6)-- (1.5,2.6);
\draw [line width=1.2pt] (1.5,2.6)-- (2,1.73);
\draw [line width=1.2pt] (2,1.73)-- (2.5,2.6);
\draw [line width=1.2pt] (2.5,2.6)-- (3,1.73);
\draw [line width=1.2pt] (3,1.73)-- (2,1.73);
\draw [line width=1.2pt] (2,1.73)-- (1,1.73);
\draw [line width=1.2pt] (1,1.73)-- (1.5,0.87);
\draw [line width=1.2pt] (1.5,0.87)-- (2,1.73);
\draw [line width=1.2pt] (2,1.73)-- (2.5,0.87);
\draw [line width=1.2pt] (2.5,0.87)-- (3,1.73);
\draw [line width=1.2pt] (3,1.73)-- (3.5,0.87);
\draw [line width=1.2pt] (3.5,0.87)-- (2.5,0.87);
\draw [line width=1.2pt] (2.5,0.87)-- (1.5,0.87);
\draw [line width=1.2pt] (1.5,0.87)-- (0.5,0.87);
\draw [line width=1.2pt] (0.5,0.87)-- (1,0);
\draw [line width=1.2pt] (1,0)-- (1.5,0.87);
\draw [line width=1.2pt] (1.5,0.87)-- (2,0);
\draw [line width=1.2pt] (2,0)-- (2.5,0.87);
\draw [line width=1.2pt] (2.5,0.87)-- (3,0);
\draw [line width=1.2pt] (3,0)-- (3.5,0.87);
\draw [line width=1.2pt] (3.5,0.87)-- (4,0);
\draw [line width=1.2pt] (4,0)-- (3,0);
\draw [line width=1.2pt] (3,0)-- (2,0);
\draw [line width=1.2pt] (2,0)-- (1,0);
\draw [line width=1.2pt] (1,0)-- (0,0);
\fill [color=black] (0,0) circle (2.5pt);
\draw[color=black] (-1.2,0.24) node {$|\nu|=|\lambda|+|\mu|$};
\fill [color=black] (1,0) circle (2.5pt);
\draw[color=black] (0.94,-0.47) node {. . .};
\fill [color=black] (0.5,0.87) circle (2.5pt);
\node [rotate=60] at (0,1) {. . .};
\node [rotate=300] at (4,1) {. . .};
\fill [color=black] (2,0) circle (2.5pt);
\node [rotate=300] at (2.48,-1.2) {$|\lambda|+\mu_1+\mu_2$};
\node [rotate=300] at (3.4,-0.8) {$|\lambda|+\mu_1$};
\fill [color=black] (1.5,0.87) circle (2.5pt);
\fill [color=black] (3,0) circle (2.5pt);
\fill [color=black] (2.5,0.87) circle (2.5pt);
\fill [color=black] (4,0) circle (2.5pt);
\draw[color=black] (4.45,0.24) node {$|\lambda|$};
\fill [color=black] (3.5,0.87) circle (2.5pt);
\fill [color=black] (1,1.73) circle (2.5pt);
\draw[color=black] (0.4,2) node {$\nu_1+\nu_2$};
\fill [color=black] (2,1.73) circle (2.5pt);
\fill [color=black] (3,1.73) circle (2.5pt);
\draw[color=black] (3.6,2) node {$\lambda_1+\lambda_2$};
\fill [color=black] (1.5,2.6) circle (2.5pt);
\draw[color=black] (1,2.75) node {$\nu_1$};
\fill [color=black] (2.5,2.6) circle (2.5pt);
\draw[color=black] (3,2.75) node {$\lambda_1$};
\fill [color=black] (2,3.46) circle (2.5pt);
\draw[color=black] (2,3.85) node {$0$};
\end{tikzpicture}
  \caption{Border labels corresponding to a triple of integer partitions.}
  \label{fig: Littlewood-Richardson}
 \end{figure}
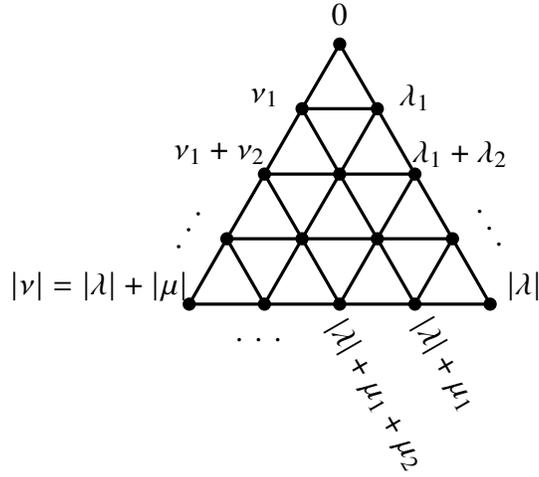
\end{thm}

For example, if $\nu = (3,2,1)$, $\lambda = \mu = (2,1)$, then by \autoref{ex:hive} there are two integral hives with the corresponding border labels (corresponding to $x=4$ and $x=5$). Thus $c^\nu_{\lambda \mu} = 2$.

The saturation theorem proved by Knutson and Tao is equivalent to the following statement:

\begin{thm}[Knutson--Tao]
 If there exists a representation $\rho:\Delta^r_3\to\T_0$ with given border labels in $\T^\Z_0$, then there exists a representation $\rho:\Delta^r_3\to\T_0^\Z$ with these border labels.
\end{thm}

As discussed in Fulton's survey \cite{Fulton98}, the work of Klyachko \cite{Klyachko98}, combined with the Knutson--Tao saturation theorem, implies the following result about eigenvalues of sums of Hermitian matrices which was previously known as ``Horn's Conjecture'':

\begin{thm}
 There are $r \times r$ Hermitian matrices $A,B,C$ with $A+B=C$ having respective eigenvalues $\lambda_1 \geq \lambda_2 \geq \cdots \geq \lambda_r$, $\mu_1 \geq \mu_2 \geq \cdots \geq \mu_r$, and $\nu_1 \geq \nu_2 \geq \cdots \geq \nu_r$ if and only if there is a representation $\rho:\Delta^r_3\to\T_0^\Z$ with logarithmic border labels $(\lambda,\mu,\nu)$.
\end{thm}

\part{Foundations of polymatroids}
\section{The universal tract and the universal pasture}
\label{subsection: universal tract and universal pasture}

In this section, we extend the notions of ``universal tract'' and ``universal pasture'', as introduced in \cite{Baker-Lorscheid21b}, from matroids to polymatroids.

\subsection{Representation spaces and thin Schubert cells}
\label{subsection: representation spaces and thin Schubert cells}

Let $J\subseteq\Delta^r_n$ be an M-convex set and $F$ a tract. The \emph{(strong) representation space of $J$ over $F$} is the set $\upR_J(F)$ of all strong $F$-representations of $J$, and the \emph{weak representation space of $J$ over $F$} is the set $\upR_J^w(F)$ of all weak $F$-representations of $J$. 

The multiplicative group $F^\times$ of $F$ acts diagonally on both the weak and strong representation spaces of $J$ over $F$ (cf.\ \autoref{lemma: torus action on representations} for a generalization). The \emph{(strong) thin Schubert cell of $J$ over $F$} is $\Gr_J(F)=\upR_J(F)/F^\times$. The \emph{weak thin Schubert cell of $J$ over $F$} is $\Gr^w_J(F)=\upR^w_J(F)/F^\times$. 

\begin{rem}
\begin{enumerate}
\item[]
\item If $F$ is a field and $M$ is a matroid, then $\Gr^w_M(F)=\Gr_M(F)$ corresponds to the usual notion of the thin Schubert cell of $M$, which consists of all points $x$ of the Grassmannian $\Gr(r,n)(F)$ over $F$ for which the Pl\"ucker coordinate $x_\alpha$ is nonzero  precisely when $\alpha$ is a basis of $M$.
\item If $F=\T_0$ is the tropical hyperfield and $J$ is an $M$-convex set, we have $\Gr^w_J(\T_0)=\Gr_J(\T_0)$ since $\T_0$ is excellent (\autoref{cor: perfect tracts are excellent}). 
\item If $M$ is a matroid, then the association $\brho\mapsto-\log\brho$ (cf.\ \autoref{subsection: M-convex functions as representations over the tropical hyperfield}) identifies $\Gr_M(\T_0)$ with the \emph{local Dressian} $\Dr_M$ of all tropical linear spaces with underlying matroid $M$.
\end{enumerate}
\end{rem}

\subsection{Functoriality}
\label{subsection: functoriality}

Let $f\colon F_1\to F_2$ be a tract morphism and $\brho\colon[n]^r\to F_1$ a strong (resp.\ weak) $F_1$-representation of an M-convex set $J\subseteq\Delta^r_n$. The \emph{push-forward of $\brho$ along $f$} is the function
\[
 \begin{array}{cccc}
  f_\ast(\brho)\colon & [n]^r  & \longrightarrow & F_2 \\
                & \balpha & \longmapsto     & f(\brho(\balpha)).
 \end{array} 
\]
Since tract morphisms preserve null sets as well as non-zero elements, $f_\ast(\brho)$ is a strong (resp.\ weak) $F_2$-representation of $J$. Thus $f\colon F_1\to F_2$ defines maps
\[
 f_\ast\colon \ \upR_J(F_1) \ \longrightarrow \ \upR_J(F_2) \quad \text{and} \quad
  f_\ast\colon \ \upR^w_J(F_1) \ \longrightarrow \ \upR^w_J(F_2).
\]
More precisely, taking the strong (resp.\ weak) representation space of $J$ defines a functor $\upR_J\colon\Tracts\to\Sets$ (resp.\ $\upR^w_J\colon\Tracts\to\Sets$). 

Similarly, the strong (resp.\ weak) thin Schubert cell of $J$ is functorial in $F$, i.e., a tract morphism $f\colon F_1\to F_2$ induces maps 
\[
 f_\ast\colon\Gr_J(F_1)\longrightarrow \Gr_J(F_2) \quad \text{and} \quad f_\ast\colon\Gr_J^w(F_1)\longrightarrow \Gr^w_J(F_2),
\]
yielding functors $\Gr_J$ and $\Gr^w_J$ from $\Tracts$ to $\Sets$.

\subsection{The universal tract}
\label{subsection: the universal tract}

Let $J\subseteq\Delta^r_n$ be an M-convex set of effective rank $\rbar=r-\norm{\delta^-_J}$ and define $\bJ:=\{\bbeta\in[n]^\rbar\mid \Sigma\bbeta\in \Jbar\}$, where $\Jbar=J-\delta^-_J$ is the reduction of $J$. (Here $\Sigma : [n]^\rbar \to \Delta^{\rbar}_n$ is the map defined in \autoref{eq:SigmaDefinition}.)

Let $\omega_J=\delta^+_J-\delta^-_J$ be the width of $J$. The \emph{extended universal tract of $J$} is the tract $\widehat T_J = \pastgen{\Funpm(x_\bbeta\mid\bbeta\in \bJ)}{S}$, where $S$ consists of the relations
\[
 x_{i_{\sigma(1)},\dotsc,i_{\sigma(\rbar)}} \ = \ \sign(\sigma) \cdot x_{i_1,\dotsc,i_\rbar}
\]
for all $(i_1,\dotsc,i_\rbar)\in\bJ$ and $\sigma\in S_\rbar$ together with the Pl\"ucker relations
\[
 \sum_{k=0}^s \ (-1)^k \cdot x_{\balpha{i_0}\dotsc\widehat{i_k}\dotsc{i_s}} \cdot x_{\balpha{i_k}{j_2}\dotsc{j_s}}
\]
for all $2\leq s\leq \rbar$, $\balpha\in[n]^{\rbar-s}$ and $i_0,\dotsc,i_s,j_2,\dotsc,j_s\in [n]$ with $\sum\balpha i_0\dotsc i_s j_2\dotsc j_s\leq\omega_J$, using the convention that $x_\bbeta=0$ if $\bbeta\notin\bJ$. 

The \emph{universal representation of $J$} is the representation
\[
 \hat\brho\colon \ [n]^\rbar \ \longrightarrow \ \widehat T_J
\]
defined by $\hat\brho(\balpha)=x_\balpha$. It is a strong $\widehat T_J$-representation by the very definition of $\widehat T_J$. 

The extended universal tract $\widehat T_J$ is graded by the multiplicative map
\[
 \deg\colon \ \widehat T_J \ \longrightarrow \ \Z
\]
with $\deg(x_\bbeta)=1$ for $\bbeta\in J$ and $\deg(0)=0$.

The \emph{universal tract of $J$} is the subtract
\[
 T_J \ = \ \big\{ a\in\widehat  T_J \, \big| \, \deg(a)=0 \big\}
\]
of $\widehat T_J$.

By the idempotency principle for proper polymatroids (\autoref{prop: idempotency principle}), the existence of the universal representation implies that both $\widehat T_J$ and $T_J$ are near-idempotent if $J$ is not the translate of a matroid. If $\omega_{J,i}\geq3$ for some $i\in[n]$, then both $\widehat T_J$ and $T_J$ are idempotent.

\begin{prop}\label{prop: universal property of the universal tract}
 Let $J\subseteq\Delta^r_n$ be an M-convex set with extended universal tract $\widehat T_J$ and universal tract $T_J$. Let $F$ be a tract. Then composing the universal representation $\hat\brho\colon[n]^\rbar\to\widehat T_J$ with a tract morphism $f\colon\widehat T_J\to F$ yields a bijection
 \[
  \Phi_{J,F}\colon  \ \Hom(\widehat T_J,\ F) \ \longrightarrow \ \upR_J(F),
 \]
 which descends to a bijection
 \[
  \overline\Phi_{J,F}\colon \ \Hom(T_J, \ F) \ \longrightarrow \ \Gr_J(F).
 \]
 Both bijections are functorial in $F$.
\end{prop}

\begin{proof}
 This is proven exactly as for usual matroids, cf.\ \cite[Thm.\ 6.15 and Prop.\ 6.23]{Baker-Lorscheid20}. For completeness, we sketch the argument.

 The inverse bijection $\Psi_{J,P}\colon\upR_J(P)\to\Hom(\widehat T_J, P)$ to $\Phi_{J,P}$ is given as follows: an $F$-representation $\bsigma\colon[n]^\rbar\to F$ of $J$ is mapped to the tract morphism $f:\widehat T_J\to F$ determined by $f(x_\balpha)=\bsigma(\balpha)$. Since $\bsigma$ satisfies the Pl\"ucker relations as a representation of the M-convex set $J$, it follows from the universal properties of free algebras and quotients (cf.\ \autoref{subsubsection: free algebras} and \autoref{subsubsection: quotients of tracts}) that the assignment $x_\balpha\mapsto\bsigma(\balpha)$ defines a tract morphism $f\colon\widehat T_J\to F$. By construction, $\Phi_{P,J}(f)=\bsigma$. Since $\widehat T_J$ is generated by the $x_\balpha$ over $\Funpm$, $f$ is uniquely determined by the images of the $x_\balpha$, which completes the proof that $\Psi_{J,P}$ is the inverse bijection of $\Phi_{J,P}$.

 The bijection $\Phi_{J,P}$ descends to a bijection $\overline\Phi_{J,P}\colon \Hom(T_J, \ P) \to\Gr_J(P)$ for the following reason: two morphisms $f_i\colon\widehat T_J\to F$ (for $i=1,2$) have the same restriction to $T_J=\{c\in\widehat T_J\mid\deg c=0\}$ if and only if there exists $a\in P^\times$ such that $f_2(x_\balpha)=a f_1(x_\balpha)$ for all $\balpha\in[n]^\rbar$. This is the case if and only if $\bsigma_2=a\bsigma_1$ for the $F$-representations $\bsigma_i=f_i\circ\hat\brho$ (for $i=1,2$), which means, by definition, that $[\bsigma_1]=[\bsigma_2]$ in $\Gr_J(P)$.

 The functoriality of $\Phi_{J,P}$ in $F$ follows from the fact that both $\Hom(\widehat T_J,-)$ and $\upR_J(-)$ act on morphisms in terms of compositions of maps. Since a tract morphism $P\to Q$ restricts to a group homomorphism $P^\times\to Q^\times$, and since $\Hom(\widehat T_J,P)$ and $\upR_J(P)$ are sets of $P^\times$-orbits, $\overline\Phi_{J,P}$ is also functorial.
\end{proof}

\subsection{The universal pasture}
\label{subsection: the universal pasture}

Roughly speaking, the universal pasture of a polymatroid is the $3$-term truncation of the universal tract, which only captures the $3$-term Pl\"ucker relations. For simplicity, we refrain from introducing pastures in this text, and instead define the universal pasture as a tract. For more details on pastures, including their precise relationship to tracts, see \cite[Section 6.4]{Baker-Lorscheid21b} and \cite{Baker-Lorscheid20}.

Let $J\subseteq\Delta^r_n$ be an M-convex set of effective rank $\rbar$, and let $\bJ$ be defined as above. The \emph{extended universal pasture of $J$} is the tract $\widehat P_J = \pastgen{\Funpm(x_\bbeta\mid\bbeta\in \bJ)}{S}$, where $S$ consists of the relations
\[
 x_{i_{\sigma(1)},\dotsc,i_{\sigma(\rbar)}} \ = \ \sign(\sigma) \cdot x_{i_1,\dotsc,i_\rbar}
\]
for all $(i_1,\dotsc,i_\rbar)\in\bJ$ and $\sigma\in S_\rbar$ together with the $3$-term Pl\"ucker relations
\[
 x_{\balpha ij} \cdot x_{\balpha kl} \ - \ x_{\balpha ik} \cdot x_{\balpha jl} \ + \ x_{\balpha il} \cdot x_{\balpha jk}
\]
for all $\balpha\in[n]^{\rbar-2}$ and $i,j,k,l\in [n]$. The \emph{universal representation of $J$} is the weak $\widehat P_J$-representation $\hat\brho\colon [n]^\rbar \to \widehat P_J$ of $J$ defined by $\hat\brho(\bbeta)=x_\bbeta$. 

Analogous to the extended universal tract, the extended universal pasture is graded by the multiplicative map $\deg\colon \widehat P_J\to\Z$ with $\deg(x_\bbeta)=1$ for $\bbeta\in \bJ$ and $\deg(0)=0$. The \emph{universal pasture of $J$} is the subtract $P_J=\{ a\in\widehat  P_J \mid \deg(a)=0\}$ of $\widehat P_J$.

By \autoref{prop: idempotency principle}, $\widehat P_J$ and $P_J$ are near-idempotent if $J$ is not the translate of a matroid. If $\omega_{J,i}\geq3$ for some $i\in[n]$, then both $\widehat P_J$ and $P_J$ are idempotent.

\begin{prop}\label{prop: universal property of the universal pasture}
 Let $J\subseteq\Delta^r_n$ be an M-convex set with extended universal pasture $\widehat P_J$ and universal pasture $P_J$. Let $F$ be a tract. Then composing the universal representation $\hat\brho\colon[n]^\rbar\to\widehat P_J$ with a tract morphism $f\colon\widehat P_J\to F$ yields a bijection
 \[
  \Phi_{J,P}\colon \ \Hom(\widehat P_J,\ P) \ \longrightarrow \ \upR^w_J(P),
 \]
 which descends to a bijection
 \[
  \overline\Phi_{J,P}\colon \ \Hom(P_J, \ P) \ \longrightarrow \ \Gr^w_J(P).
 \]
 Both bijections are functorial in $F$.
\end{prop}

\begin{proof}
The proof of \autoref{prop: universal property of the universal tract} applies \emph{mutatis mutandis}; we omit the details.
\end{proof}

\subsection{The comparison map}
\label{subsection: comparison map}

Let $J\subseteq\Delta^r_n$ be an M-convex set, and let $\bJ$ be defined as above. Since the extended universal pasture $\widehat P_J=\pastgen{\Funpm(x_\bbeta\mid\bbeta\in \bJ)}{S'}$ of $J$ is defined by the set $S'$ of $3$-term Pl\"ucker relations, which are a subset of the set $S$ of all Pl\"ucker relations, which define the extended universal tract $\widehat T_J=\pastgen{\Funpm(x_\bbeta\mid\bbeta\in \bJ)}{S}$ of $J$, these two tracts come with a canonical morphism $\hat\pi_J\colon\widehat P_J\to\widehat T_J$, which is degree preserving and thus restricts to a morphism $\pi_J\colon P_J\to T_J$.

For M-convex sets $J\subseteq\Delta^r_n$ of rank $r\leq2$, we have $S'=S$, and thus the canonical morphisms $\hat\pi_J$ and $\pi_J$ are isomorphisms. This fails, in general, for M-convex sets of larger rank due to the presence of Pl\"ucker relations with $4$ or more terms.

However, as we prove below (\autoref{thm: bijection between the universal pasture and the universal tract}), the canonical maps $\hat\pi_J$ and $\pi_J$ are bijective. As a preparation for the proof, we call a Pl\"ucker relation
\[
 \sum_{k=0}^s \ (-1)^k \cdot x_{\balpha{i_0}\dotsc\widehat{i_k}\dotsc{i_s}} \cdot x_{\balpha{i_k}{j_2}\dotsc{j_s}}
\]
(with $2\leq s\leq \rbar$, $\balpha\in[n]^{\rbar-s}_n$ and $i_0,\dotsc,i_s,j_2,\dotsc,j_s\in [n]$) \emph{degenerate} if it has exactly two nonzero terms, i.e., there are exactly two indices $k$ in $\{0,\dotsc,s\}$ for which both $\balpha{i_0}\dotsc\widehat{i_k}\dotsc{i_s}$ and $\balpha{i_k}{j_2}\dotsc{j_s}$ are in $\bJ$.

Since $a-b\in N_{\widehat T_J}$ implies that $a=b$ in $\widehat T_J$ by the uniqueness of additive inverses, the degenerate Pl\"ucker relations enforce relations between the generators $x_\beta$ of $\widehat T_J$ and, in the case of degenerate $3$-term Pl\"ucker relations, of $\widehat P_J$. The following result implies that, in fact, the degenerate $3$-term Pl\"ucker relations generate all relations between the generators of $\widehat T_J$. (It does not, however, imply that the non-degenerate Pl\"ucker relations are generated by the $3$-term Pl\"ucker relations; in particular, the bijective morphism $P_J\to T_J$ is in general not an isomorphism.)

\begin{thm}\label{thm: bijection between the universal pasture and the universal tract}
 Let $J$ be a polymatroid. Then the canonical morphisms  $\hat\pi_J\colon\widehat P_J\to\widehat T_J$ and $\pi_J\colon P_J\to T_J$ are bijections.
\end{thm}

\begin{proof}
 Since $\hat\pi_J(0)=0$, it suffices to show that the restriction of $\hat\pi_J$ to $\widehat P_J^\times\to\widehat T_J^\times$ is bijective for the first claim. The second claim (about $\pi_J$) follows from the first claim by taking the respective degree $0$ parts. 
 
 The groups $\widehat T_J^\times$ and $\widehat P_J$ are quotients of the free abelian group generated by $-1$ and the symbols $x_\bbeta$ with $\bbeta\in\bJ$ by certain respective subgroups $H_J$ and $H_J^w$ (defined below). We show by an elementary induction over the number of terms of a Pl\"ucker relation that $H_J^w=H_J$, which shows that $\widehat P_J^\times\to\widehat T_J^\times$ is a bijection. Note that since $\hat\pi_J$ is degree preserving, this result implies at once that $\pi_J$ is also a bijection.

 Before we begin with the induction, we note that if $J$ is a matroid, the claim follows from general results on perfect tracts. Namely, enriching the nullset of $\widehat P_J$ by the set $S$ of all relations $\sum a_i$ with at least $3$ nonzero terms yields a tract $F=\pastgen{\widehat P_J}{S}$, which is perfect since it satisfies the modified strong fusion rule; see \cite[Thm.\ 1.11]{Baker-Zhang23}. By \cite[Thm.\ 3.46]{Baker-Bowler19}, every weak $F$-representation of $J$ is strong. Thus, by \autoref{prop: universal property of the universal tract} and \autoref{prop: universal property of the universal pasture}, we have canonical identifications
 \[
  \Hom(\widehat T_J,\ F) \ = \ \upR_J(F) \ = \ \upR^w_J(F) \ = \ \Hom(\widehat P_J,\ F).
 \]
 This means that the canonical projection $\pi_S\colon\widehat P_J\to \pastgen{\widehat P_J}{S}=F$ factors through the surjection $\hat\pi_J\colon\widehat P_J\to \widehat T_J$. Since every relation in $S$ has at least $3$ nonzero terms, $\pi_S$ is injective, and so is $\hat\pi_J$. This shows that $\hat\pi_J$ is a bijection. Since the translation $J\to \Jbar=J-\delta^-_J$ induces isomorphisms between the respective universal tracts and universal pastures (cf.\ \autoref{thm: foundations of embedded minors}), this proof extends to all translates of matroids.
 
 Even though the following proof does not rely on the previous discussion, we can use it to simplify matters: if $J$ is not the translate of a matroid, then the idempotency principle for proper polymatroids (\autoref{prop: idempotency principle}) implies that $1=-1$ in $\widehat T_J$ and $\widehat P_J$. 
 
 We therefore assume that $-1=1$, which leads to a number of simplifications: 
 \begin{enumerate}
  \item 
  The group $\widehat T_J$ is generated by the symbols $x_\bbeta$ with $\bbeta\in\bJ$, i.e., we can remove the generator $-1$.
  Moreover, the generators $x_\bbeta$ of $\widehat T_J$ are invariant under the permutation of the coefficients of $\bbeta$, which allows us to define $x_{\beta}:=x_\bbeta$ for $\beta=\Sigma\bbeta\in \Jbar$, independently of the order of the coefficients of $\bbeta$. 
  \item 
  The Pl\"ucker relations for the $x_\beta$ turn into
  \[ \Pl(\alpha|i_0\dotsc i_s|j_2\dotsc j_s) :=
   \sum_{k=0}^s \quad x_{\alpha+\epsilon_{i_0}+\dotsb+\widehat{\epsilon_{i_k}}+\dotsb+\epsilon_{i_s}} \cdot x_{\alpha+\epsilon_{i_k}+\epsilon_{j_2}+\dotsb+\epsilon_{j_s}} \quad \in \quad N_{\widehat T_J}
  \]
  for $2\leq s\leq \rbar$, $\alpha\in\Delta^{\rbar-s}_n$ and $i_0,\dotsc,i_s,j_2,\dotsc,j_s\in[n]$ with $\delta^-_J\leq\alpha$ and $\alpha +\epsilon_{i_0}+\dotsc +\epsilon_{i_s}+\epsilon_{j_2}+\dotsc +\epsilon_{j_s}\leq\omega_J$. In particular, we can drop the sign $(-1)^{k+\sigma(k)}$. 
  \item 
  The group $\widehat T_J^\times$ is the quotient of the free abelian group generated by the symbols $x_\beta$ with $\beta\in J$ modulo the subgroup $H_J$ generated by the \emph{degenerate generalized cross ratios}
  \[
   \frac{x_{\alpha+\epsilon_{i_0}+\dotsb+\widehat{\epsilon_{i_k}}+\dotsb+\epsilon_{i_s}} \cdot x_{\alpha+\epsilon_{i_k}+\epsilon_{j_2}+\dotsb+\epsilon_{j_s}}}{x_{\alpha+\epsilon_{i_0}+\dotsb+\widehat{\epsilon_{i_l}}+\dotsb+\epsilon_{i_s}} \cdot x_{\alpha+\epsilon_{i_l}+\epsilon_{j_2}+\dotsb+\epsilon_{j_s}}},
  \]
  whose numerator and denominator are the two nonzero terms of a degenerate 
  Pl\"ucker relation $\Pl(\alpha|i_0\dotsc i_s|j_2\dotsc j_s)$.
 \end{enumerate}
 
 We show by induction over $s\geq2$ that the degenerate generalized cross ratios lie in the subgroup $H_J^w$ generated by the \emph{degenerate cross ratios} 
 \[
  \frac{x_{\alpha+\epsilon_i+\epsilon_l}\cdot x_{\alpha+\epsilon_j+\epsilon_k}}{x_{\alpha+\epsilon_i+\epsilon_k}\cdot x_{\alpha+\epsilon_j+\epsilon_l}}
 \]
 with $\alpha\in\Delta^{\rbar-2}_r$ (i.e., $s=2$), which stem from degenerate $3$-term Pl\"ucker relations $\Pl(\alpha|ikl|j)$. Since $\widehat P_J$ is the quotient of the free abelian group generated by the $x_\beta$ modulo $H_J^w$, this proves the claim of the theorem.
 
 The base case $s=2$ is tautologically true. Thus we assume that $s\geq3$, and we consider a degenerate Pl\"ucker relation $\Pl(\alpha|i_0\dotsc i_s|j_2\dotsc j_s)$. After permuting the indices, we can assume that the two nontrivial terms are indexed by $k=0$ and $l=1$, which means that we need to show that the generalized degenerate cross ratio
 \[
  \frac{x_{\alpha+\epsilon_{i_1}+\epsilon_{i_2}+\dotsb+\epsilon_{i_s}} \cdot x_{\alpha+\epsilon_{i_0}+\epsilon_{j_2}+\dotsb+\epsilon_{j_s}}}{x_{\alpha+\epsilon_{i_0}+\epsilon_{i_2}+\dotsb+\epsilon_{i_s}} \cdot x_{\alpha+\epsilon_{i_1}+\epsilon_{j_2}+\dotsb+\epsilon_{j_s}}}
 \]
 lies in $H_J^w$. If $\{i_2,\dotsc,i_s\}\cap\{j_2,\dotsc,j_s\}$ contains a common element, say $i_{s}=j_{s}$ (after rearranging indices), then the degenerate Pl\"ucker relation $\Pl(\alpha|i_0\dotsc i_s|j_2\dotsc j_s)$ is equal to the degenerate Pl\"ucker relation $\Pl(\alpha+\epsilon_{i_s}|i_0\dotsc i_{s-1}|j_2\dotsc j_{s-1})$, up to a zero term. Thus the generalized degenerate cross ratio in question appears already for a shorter Pl\"ucker relation and lies in $H_J^w$ by the inductive hypothesis. Therefore, we can assume that $\{i_2,\dotsc,i_s\}\cap\{j_2,\dotsc,j_s\}=\emptyset$ in the following.
 
 Define $\beta=\epsilon_{i_1}+\dotsb+\epsilon_{i_s}$ and $\gamma=\epsilon_{i_1}+\epsilon_{j_2}+\dotsb+\epsilon_{j_s}$. Since $\Pl(\alpha|i_0\dotsc i_s|j_2\dotsc j_s)$ has only two nonzero terms for $k=0$ and $l=1$, the indices $i_0$ and $i_1$ do not appear in $\{i_2,\dotsc,i_s\}$ and thus $\gamma_{i_1}\geq\beta_{i_1}=1$. 
 Since $\{i_2,\dotsc,i_s\}\cap\{j_2,\dotsc,j_s\}=\emptyset$, we have $\gamma_{j_s}>\beta_{j_s}$, and the exchange axiom for M-convex sets yields an index in $\{i_2,\dotsc,i_s\}$, say $i_s$, such that both $\beta-\epsilon_{i_s}+\epsilon_{j_s}$ and $\gamma+\epsilon_{i_s}-\epsilon_{j_s}$ are in $J$.
 
 For the ease of notation, we introduce the abbreviations $\zeta^{\widehat S}_R=x_\eta$ and $\xi^{\widehat R}_S=x_\vartheta$ for subsets $S\subseteq\{0,\dotsc,s\}$ and $R\subseteq\{2,\dotsc,s\}$, where
 \[
  \eta = \alpha+\epsilon_{i_0}+\dotsb+\epsilon_{i_s}-\sum_{k\in S} \epsilon_{i_k}+\sum_{k\in R} \epsilon_{j_k} \ \text{and} \ \vartheta = \alpha+\epsilon_{j_2}+\dotsb+\epsilon_{j_s}+\sum_{k\in S} \epsilon_{i_k}-\sum_{k\in R} \epsilon_{j_k}.
 \]
 So the Pl\"ucker relation {$\Pl(\alpha|i_0\dotsc i_s|j_2\dotsc j_s)$} reads
 \[
  \smallunderbrace{\zeta^{\widehat 0} \cdot \xi_0}_{\neq0} \ + \ \smallunderbrace{\zeta^{\widehat 1} \cdot \xi_1}_{\neq0} \ + \ \smallunderbrace{\zeta^{\widehat 2} \cdot \xi_2}_{=0} \ + \ \dotsb \ + \  \smallunderbrace{\zeta^{\widehat s} \cdot \xi_s}_{=0} \quad \in \quad N_{\widehat T_J},
 \]
 and the conclusion from the previous paragraph means that $\zeta^{\widehat{0s}}_s\cdot\xi^{\widehat s}_{1s}\neq0$. Our goal is to show that the generalized degenerate cross ratio
 \[
  \frac{\zeta^{\widehat 1} \cdot \xi_1}{\zeta^{\widehat 0} \cdot \xi_0} \ = \ \frac{x_{\alpha+\epsilon_{i_1}+\epsilon_{i_2}+\dotsb+\epsilon_{i_s}} \cdot x_{\alpha+\epsilon_{i_0}+\epsilon_{j_2}+\dotsb+\epsilon_{j_s}}}{x_{\alpha+\epsilon_{i_0}+\epsilon_{i_2}+\dotsb+\epsilon_{i_s}} \cdot x_{\alpha+\epsilon_{i_1}+\epsilon_{j_2}+\dotsb+\epsilon_{j_s}}}
 \]
 is in $H_J^w$. We divide the proof into two cases: either $\zeta^{\widehat s}=0$ or $\xi_s=0$.
 
 \medskip\noindent\textbf{Case 1:} $\zeta^{\widehat s}=0$. Consider the $3$-term Pl\"ucker relation 
 \[
 \Pl(\zeta^{\widehat{01s}}|i_0i_1i_s|j_s) := 
  \smallunderbrace{\zeta^{\widehat 0}}_{\neq0} \cdot \zeta^{\widehat{1s}}_s \ + \ 
  \smallunderbrace{\zeta^{\widehat 1}}_{\neq0} \cdot \smallunderbrace{\zeta^{\widehat{0s}}_s}_{\neq0} \ + \ 
  \smallunderbrace{\zeta^{\widehat s}}_{=0} \cdot \zeta^{\widehat{01}}_s \quad \in \quad N_{\widehat T_J}.
 \]
 Since a relation in $N_{\widehat T_J}$ cannot have exactly one nonzero term, we conclude that $\zeta^{\widehat{1s}}_s\neq0$, so that $\Pl(\zeta^{\widehat{01s}}|i_0i_1i_s|j_s)$ is degenerate and thus
 \[
  \frac{\zeta^{\widehat 1} \cdot \zeta^{\widehat{0s}}_s}{\zeta^{\widehat 0} \cdot \zeta^{\widehat{1s}}_s} \quad \in \quad H_J^w.
 \]
 
 Next we aim to show that the Pl\"ucker relation 
\[ \Pl(\alpha+\epsilon_{j_s}|i_0\dotsc i_{s-1}|j_2\dotsc j_{s-1}) :=
  \smallunderbrace{\zeta^{\widehat{0s}}_s\cdot\xi_0}_{\neq0} \ + \ \smallunderbrace{\zeta^{\widehat{1s}}_s\cdot\xi_1}_{\neq0} \ + \ \zeta^{\widehat{2s}}_s\cdot\xi_2 \ + \ \dotsb \ + \ \zeta^{\widehat{(s-1)s}}_s\cdot\xi_{s-1} \quad \in \quad N_{\widehat T_J}
 \]
 is degenerate, i.e., $\zeta^{\widehat{ks}}_s\cdot\xi_k=0$ for $k=2,\dotsc,s-1$. We know that $\zeta^{\widehat k}\cdot\xi_k=0$. If $\xi_k=0$, then $\zeta^{\widehat{ks}}_s\cdot\xi_k=0$, as desired. If $\zeta^{\widehat k}=0$, then we consider the $3$-term Pl\"ucker relation 
\[ \Pl(\zeta^{\widehat{0ks}}|i_0i_1i_s|j_s) :=
  \smallunderbrace{\zeta^{\widehat 0}}_{\neq0} \cdot \zeta^{\widehat{ks}}_s \ + \ 
  \smallunderbrace{\zeta^{\widehat k}}_{=0} \cdot \zeta^{\widehat{0k}}_s \ + \ 
  \smallunderbrace{\zeta^{\widehat s}}_{=0} \cdot \zeta^{\widehat{0s}}_s \quad \in \quad N_{\widehat T_J}.
 \]
 Since a relation in $N_{\widehat T_J}$ cannot have exactly one nonzero term, we conclude that $\zeta^{\widehat{ks}}_s=0$, which implies $\zeta^{\widehat{ks}}_s\cdot\xi_k=0$ that in this case as well.
 
 This shows that the Pl\"ucker relation {$\Pl(\alpha+\epsilon_{j_s}|i_0\dotsc i_{s-1}|j_2\dotsc j_{s-1})$} is degenerate. So the inductive hypothesis applies and shows that
 \[
  \frac{\zeta^{\widehat{1s}}_s\cdot\xi_1}{\zeta^{\widehat{0s}}_s\cdot\xi_0} \quad \in \quad H_J^w.
 \]
 Therefore
 \[
  \frac{\zeta^{\widehat 1} \cdot \xi_1}{\zeta^{\widehat 0} \cdot \xi_0} \ = \ \frac{\zeta^{\widehat 1} \cdot \zeta^{\widehat{0s}}_s}{\zeta^{\widehat 0} \cdot \zeta^{\widehat{1s}}_s} \ \cdot \ \frac{\zeta^{\widehat{1s}}_s\cdot\xi_1}{\zeta^{\widehat{0s}}_s\cdot\xi_0}
  \quad \in \quad H_J^w,
 \]
 which completes the inductive step in case 1.

\medskip\noindent\textbf{Case 2:} $\xi_s=0$. The proof is analogous to case 1. Consider the $3$-term Pl\"ucker relation 
 \[ \Pl(\xi^{\widehat{s}}|i_0i_1i_s|j_s) :=
  \smallunderbrace{\xi_0}_{\neq0} \cdot \smallunderbrace{\xi^{\widehat{s}}_{1s}}_{\neq0} \ + \ 
  \smallunderbrace{\xi_1}_{\neq0} \cdot \xi^{\widehat{s}}_{0s} \ + \ 
  \smallunderbrace{\xi_s}_{=0} \cdot \xi^{\widehat{s}}_{01} \quad \in \quad N_{\widehat T_J}.
 \]
 Then $\xi^{\widehat{1s}}_s\neq0$ and $\Pl(\xi^{\widehat{s}}|i_0i_1i_s|j_s)$ is degenerate, which shows that
 \[
  \frac{\xi_1 \cdot \xi^{\widehat{s}}_{0s}}{\xi_0\cdot\xi^{\widehat{s}}_{1s}} \quad \in \quad H_J^w.
 \]
 
 Next we aim to show that the Pl\"ucker relation 
\[ \Pl(\alpha+\epsilon_{i_s}|i_0\dotsc i_{s-1}|j_2\dotsc j_{s-1}) :=
\smallunderbrace{\zeta^{\widehat{0}}\cdot\xi^{\widehat{s}}_{0s}}_{\neq0} \ + \ \smallunderbrace{\zeta^{\widehat{1}}\cdot\xi^{\widehat{s}}_{1s}}_{\neq0} \ + \ \zeta^{\widehat{2}}\cdot\xi^{\widehat{s}}_{2s} \ + \ \dotsb \ + \ \zeta^{\widehat{s-1}}\cdot\xi^{\widehat{s}}_{(s-1)s} \quad \in \quad N_{\widehat T_J}
 \]
 is degenerate, i.e., $\zeta^{\widehat{k}}\cdot\xi^{\widehat{s}}_{ks}=0$ for $k=2,\dotsc,s-1$. We know that $\zeta^{\widehat k}\cdot\xi_k=0$. If $\zeta^{\widehat k}=0$, then $\zeta^{\widehat{k}}\cdot\xi^{\widehat s}_{ks}=0$, as desired. If $\xi_k=0$, then we consider the $3$-term Pl\"ucker relation 
  \[ \Pl(\xi^{\widehat{s}}|i_0i_ki_s|j_s) :=
  \smallunderbrace{\xi_0}_{\neq0} \cdot \xi^{\widehat{s}}_{ks} \ + \ 
  \smallunderbrace{\xi_k}_{=0} \cdot \xi^{\widehat{s}}_{0s} \ + \ 
  \smallunderbrace{\xi_s}_{=0} \cdot \xi^{\widehat{s}}_{0k} \quad \in \quad N_{\widehat T_J}.
 \]
 Thus $\xi^{\widehat{s}}_{ks}=0$, which implies $\zeta^{\widehat{ks}}_s\cdot\xi_k=0$ also in this case.
 
 This shows that the Pl\"ucker relation {$\Pl(\alpha+\epsilon_{i_s}|i_0\dotsc i_{s-1}|j_2\dotsc j_{s-1})$} is degenerate. So the inductive hypothesis applies and shows that
 \[
  \frac{\zeta^{\widehat{1}}\cdot\xi^{\widehat{s}}_{1s}}{\zeta^{\widehat{0}}\cdot\xi^{\widehat{s}}_{0s}} \quad \in \quad H_J^w.
 \]
 Therefore in case 2 we also have 
 \[
  \frac{\zeta^{\widehat 1} \cdot \xi_1}{\zeta^{\widehat 0} \cdot \xi_0} \ = \ \frac{\zeta^{\widehat{1}}\cdot\xi^{\widehat{s}}_{1s}}{\zeta^{\widehat{0}}\cdot\xi^{\widehat{s}}_{0s}} \ \cdot \ \frac{\xi_1 \cdot \xi^{\widehat{s}}_{0s}}{\xi_0\cdot\xi^{\widehat{s}}_{1s}} \quad \in \quad H_J^w,
 \]
 which completes the proof.
\end{proof}

\autoref{thm: bijection between the universal pasture and the universal tract} has several consequences, some of which will appear later in the paper. Right away, however, we gain a series of new examples of excellent tracts. We call a tract $F$ \emph{degenerate} if every formal sum $\sum a_i\in F^+$ with at least $3$ nonzero terms is contained in $N_F$. An example of a degenerate tract is the degenerate triangular hyperfield $\T_\infty$.

\begin{cor}\label{cor: degenerate tracts are excellent}
 Every degenerate tract is excellent.
\end{cor}

\begin{proof}
 Let $J$ be a polymatroid and $\hat\pi_J\colon\widehat P_J\to \widehat T_J$ the canonical morphism from the extended universal pasture to the extended universal tract of $J$. Let $\brho\colon[n]^\rbar\to F$ be a weak $F$-representation of $J$ with associated morphism $f_\brho\colon\widehat P_J\to F$ (using \autoref{prop: universal property of the universal pasture}). By \autoref{thm: bijection between the universal pasture and the universal tract}, $\hat\pi_J$ is a bijection, which means that $f_\brho$ defines a multiplicative map $\tilde f\colon\widehat T_J\to F$. Since $N_F$ contains all relations with more than two nonzero terms, $\tilde f$ is automatically a tract morphism, which in turn corresponds to a strong $F$-representation of $J$ by \autoref{prop: M-convex functions are strong T_0-representations}. This shows that $\brho$ is a strong $F$-representation of $J$. Thus $F$ is excellent.
\end{proof}

\section{The foundation}
\label{section: the foundation}

The foundation $F_M$ of a matroid $M$ was introduced in \cite{Baker-Lorscheid21b}. It is characterized by the fact that
$\Hom(F_M,-)$ represents the functor $\ulineGr_M\colon\Tracts\to\Sets$ that sends a tract $F$ to the set of $F$-rescaling classes of representations of $M$. In this section, we extend this concept from matroids to M-convex sets.

\subsection{The realization space}
\label{subsection: the realization space}

The set $\upR^w_J(F)$ of all weak $F$-representations of an M-convex set $J\subseteq\Delta^r_n$ is invariant under scalar multiplication by $F^\times$ and under rescaling by elements of the torus $T(F):=(F^\times)^{n}$. In more detail, given a weak $F$-representation $\brho\colon[n]^r\to F$ of $J$ and $a\in F^\times$, we define $a.\brho\colon[n]^r\to F$ by the formula
\[
 (a.\brho)(\balpha) \ = \ a\cdot \brho(\balpha)
\]
for $\balpha\in\Delta^r_n$. Given $t=(t_1,\dotsc,t_n)\in T(F)$, we define
\[
 (t.\brho)(\balpha) \ = \ \Big(\prod_{i=1}^n \ t_{i}^{\balpha_i} \Big) \cdot \brho(\balpha).
\]

\begin{lemma}\label{lemma: torus action on representations}
 Both $a.\brho$ and $t.\brho$ are weak $F$-representations. If $\brho\in\upR_J(F)$, then both $a.\brho$ and $t.\brho$ are strong $F$-representations.
\end{lemma}

\begin{proof}
 This follows from the fact that for all $\balpha\in[n]^{r-2}$ and $i,j,k,l\in[n]$, the corresponding $3$-term Pl\"ucker relation for $a.\brho$ (resp.\ for $t.\brho$) is a multiple of the Pl\"ucker relation for $\brho$ by a factor $a^{2}$ (resp.\ by a factor $(t_it_jt_kt_l\cdot\prod_{m=1}^n t_m^{2\balpha_m})$). The same holds for Pl\"ucker relations with more terms, which establishes the latter claims.
\end{proof}

We thus have actions of $F^\times$ and of $T(F)$ on the set $\upR^w_J(F)$ of all weak $F$-representations. If $F^\times$ is $n$-divisible, then the $F^\times$-orbit of an $F$-representation $\brho$ is contained in the $T(F)$-orbit of $\brho$. Since this is not always the case, it is useful to consider the action of $\widehat T(F):=F^\times\times T(F)$ on $\upR^w_J(F)$ defined by
\[
 (a,t).\brho(\balpha) \ = \ a\cdot \Big(\prod_{i=1}^n \ t_i^{\balpha_i}\Big) \cdot \brho(\balpha).
\]

\begin{df}
 An \emph{$F$-rescaling class of $J$} is the $\widehat T(F)$-orbit of a weak $F$-representation of $J$. The \emph{realization space of $J$} is the set $\ulineGr^w_J(F)=\upR^w_J(F)/\widehat T(F)$ of all $F$-rescaling classes of $J$.
\end{df}

The torus action $\widehat T(F)$ on $\upR^w_J(F)$ is functorial in $F$, and in particular the push-forward $f_\ast\colon\upR^w_J(F_1)\to\upR^w_J(F_2)$ for a tract morphism $f\colon F_1\to F_2$ induces a map $f_\ast\colon\ulineGr^w_J(F_1)\to\ulineGr^w_J(F_2)$ between the respective realization spaces. Thus we may consider the realization space of $J$ as a functor $\ulineGr^w_J\colon\Tracts\to\Sets$.

\subsection{The foundation of a polymatroid}
\label{subsection: the foundation of a polymatroid}

The foundation represents the realization space $\ulineGr^w_J$ of an M-convex set $J$ just as the universal pasture represents the weak thin Schubert cell $\Gr^w_J$ (considered as a functor from $\Tracts$ to $\Sets$). The advantage of the foundation over the universal tract and the universal pasture is that it is easier to compute, and it allows for several structural results that transfer the combinatorics of (poly)matroids into algebraic properties. At the same time, the foundation and the realization space still capture essential information about the thin Schubert cells (see \autoref{thm: decomposition of the representation space}). 

We note that there is an analogous tract that represents the strong realization space $\ulineGr_J$ of $J$, but for simplicity we omit a treatment of this theory.

The extended universal pasture $\widehat P_J$ of an M-convex set $J\subseteq\Delta^r_n$ is multi-graded by the group homomorphism
\[
 \deg_{[n]}\colon \ \widehat P_J \ \longrightarrow \ \Z^n
\]
defined by $\deg_{[n]}(x_\balpha)=\Sigma\balpha$.

\begin{df}
 The \emph{foundation of $J$} is the subtract
 \[
  F_J \ = \ \big\{ a\in \widehat P_J \, \big| \, \deg_{[n]}(a)=0 \big\}
 \]
 of $\widehat P_J$.
\end{df}

Note that $F_J\subseteq P_J$ since $\deg(a)=0$ if $\deg_{[n]}(a)=0$. Note further that the idempotency principle for proper polymatroids (\autoref{prop: idempotency principle}) implies that $F_J$ is near-idempotent if $J$ is not the translate of a matroid. If $\omega_{J,i}\geq3$ for some $i\in[n]$, then $F_J$ is idempotent.

\begin{prop}\label{prop: universal property of the foundation}
 Let $J\subseteq\Delta^r_n$ be an M-convex set of effective rank $\rbar$ with extended universal pasture $\widehat P_J$, weak universal representation $\hat\brho\colon[n]^\rbar\to\widehat P_J$, and foundation $F_J$. Let $F$ be a tract. Then there exists a unique bijection
 \[
  \underline{\Phi}_{J,F}\colon \ \Hom(F_J, \ F) \ \stackrel\sim\longrightarrow \ \ulineGr^w_J(F)
 \]
 that satisfies $\underline{\Phi}_{J,F}(f\vert_{F_J})=[f\circ\hat\brho]$ for every tract morphism $f\colon\widehat P_J\to F$. Moreover, this bijection is functorial in $F$.
\end{prop}

\begin{proof}
 The proof is similar to that of the bijection $\overline\Phi_{J,F}$ in \autoref{prop: universal property of the universal tract}. In the present case, the group $\widehat T(F)=(F^\times)^{n+1}$ acts on both $\Hom(\widehat P_J,F)$ and $\upR^w_J(F)$, and $\Phi_{J,F}$ is $\widehat T(F)$-equivariant. So $\Phi_{J,F}$ descends to a functorial bijection $\underline{\Phi}_{J,F}: \Hom(F_J,F) \to \ulineGr^w_J(F)$ between the respective sets of $\widehat T(F)$-orbits.
\end{proof}

\begin{lemma}\label{lemma: the universal tract is freely generated over the foundation}
 Let $J\subseteq\Delta^r_n$ be an M-convex set with foundation $F_J$, universal pasture $P_J$, and extended universal pasture $\widehat P_J$. Then 
 \[
  P_J \ \simeq \ F_J(x_1,\dotsc,x_s) \qquad \text{and} \qquad \widehat P_J \ \simeq \ F_J(x_0,x_1\dotsc,x_s)
 \]
 for some $0\leq s\leq n$.
\end{lemma}

\begin{proof}
 The proof is analogous to \cite[Cor.\ 7.14]{Baker-Lorscheid20}. We sketch the argument for completeness. If we multiply a $3$-term Pl\"ucker relation
 \[
  x_{\balpha ij} \cdot x_{\balpha kl} \ - \ x_{\balpha ik} \cdot x_{\balpha jl} \ + \ x_{\balpha il} \cdot x_{\balpha jk} \quad \in \quad N_{\widehat P_J}
 \]
 (where $\balpha\in[n]^{\rbar-2}$ and $i,j,k,l\in[n]$ with $\sum\balpha ijkl\leq\omega_J$) by $(x_{\balpha ik} \cdot x_{\balpha jl})^{-1}$, we obtain
 \[
  \frac{x_{\balpha ij} \cdot x_{\balpha kl}}{x_{\balpha ik} \cdot x_{\balpha jl}} \quad - \quad 1 \quad + \quad \frac{x_{\balpha il} \cdot x_{\balpha jk}}{x_{\balpha ik} \cdot x_{\balpha jl}},
 \]
 which is contained in $N_{F_J}$. Therefore, the null sets of $P_J$ and of $\widehat P_J$ are generated by the null set of $N_{F_J}$. The groups $P_J^\times/F_J^\times$ and $\widehat P_J^\times/F_J^\times$ are isomorphic to subgroups of the free abelian groups $\Z^n$ and $\Z\times\Z^n$, respectively, via the multi-degree map. Thus both $P_J^\times/F_J^\times$ and $\widehat P_J^\times/F_J^\times$ are themselves free abelian groups, and the lemma follows easily from this.
\end{proof}

We show in \autoref{thm: decomposition of the representation space} that $s=n-c(J)$, where $c(J)$ is the number of indecomposable components of $J$.

\subsection{First examples}
\label{subsection: first examples}

We present here some examples of foundations of matroids. We postpone examples of foundations of proper polymatroids to \autoref{subsection: more examples of foundations}, since they are based on further theory. For a more comprehensive list of foundations of matroids, see \cite[Appendix A]{Baker-Lorscheid-Zhang24}.

\subsubsection{Regular matroids}
\label{subsubsection: representation space of regular matroids}
 The foundation of a regular matroid $M$ is $F_M=\Funpm$ (\cite[Thm.\ 7.35]{Baker-Lorscheid21b}). In this case, $\ulineGr^w_M(F)=\Hom(\Funpm,F)$ is a point for every tract $F$. Thus $\upR^w_M(F)$ consists of a single $\widehat T(F)$-orbit, which is in bijection with $(F^\times)^s$ for some $s\leq n+1$ by \autoref{lemma: the universal tract is freely generated over the foundation}.
 In particular, we have
 \[
  \upR^w_{U_{2,2}}(F) \ \simeq \ F^\times \qquad \text{and} \qquad \upR^w_{U_{2,3}}(F) \ \simeq \ (F^\times)^2
 \]
 for the uniform rank $2$ matroids $U_{2,2}$ and $U_{2,3}$ with $2$ and $3$ elements, respectively.
 
\subsubsection{The uniform rank $2$ matroid on $4$ elements}
\label{subsubsection: representation space of U24}
 The smallest matroid with a nontrivial realization space is $U_{2,4}$, whose foundation is the near-regular partial field $\U=\pastgen{\Funpm(x,y)}{x+y-1}$ (\cite[Prop.\ 4.11]{Baker-Lorscheid20}). Thus
 \[
  \ulineGr^w_{U_{2,4}}(F) \ = \ \Hom(\pastgen{\Funpm(x,y)}{x+y-1},\ F) \ = \ \big\{ (a,b)\in(F^\times)^2 \, \big| \, a+b-1\in N_{F} \big\}.
 \]
 The torus orbit $\widehat T(F).\brho$ of an $F$-representation $\brho$ of $U_{2,4}$ is in bijection with $(F^\times)^3$. 
 
\subsubsection{Binary matroids}
\label{subsubsection: representation space of binary matroids}
 The foundation of a binary matroid $M$ is $\Funpm$ if $M$ is regular and $\F_2$ otherwise (\cite[Thm.\ 7.32]{Baker-Lorscheid21b}). 
 In the latter case, $\ulineGr^w_M(F)$ is a singleton if $-1=1$ in $F$ and empty otherwise.
 
\subsubsection{Ternary matroids}
\label{subsubsection: representation space of ternary matroids}
 By \cite[Thm.\ 6.28]{Baker-Lorscheid20}, the foundation $F_M$ of a ternary matroid $M$ is isomorphic to the coproduct, or \emph{tensor product}, $F_1\otimes\dotsb\otimes F_r$ of tracts $F_1,\dotsc,F_r\in\{\F_3,\ \H,\ \D,\ \U\}$ (see \autoref{subsection: tensor products} and \autoref{subsection: more examples of tracts} for definitions), which is the tract characterized by the functorial bijection
 \[
  \ulineGr^w_M(F) \ = \ \Hom(F_M,\ F) \ \simeq \ \Hom(F_1,\ F) \ \times \ \dotsb \ \times \ \Hom(F_r,\ F).
 \]
 The terms $\Hom(F_i,F)$ are of the following forms: $\Hom(\F_3,F)$ is a singleton if $1+1+1\in N_F$ and empty otherwise,
 \begin{align*}
  \Hom(\H,F) \ &= \ \{a\in F^\times\mid a^2-a+1\in N_F\}, \\
  \Hom(\D,F) \ &= \ \{a\in F^\times\mid a+a-1\in N_F\}, 
 \end{align*}
 and $\Hom(\U,F)$ is as described in \autoref{subsubsection: representation space of U24}.

\section{Representations of embedded minors and duals}
\label{section: representations of embedded minors and duals}

In this section, we compare representations of embedded minors $J\minor\nu\mu+\tau$ of $J$ in terms of induced morphisms between the corresponding universal tracts, universal pastures, and foundations.

We fix an M-convex set $J\subseteq\Delta^r_n$ for the rest of this section and let $\bJ=\{\balpha\in[n]^r\mid \Sigma\balpha\in J\}$. As usual, we denote its duality vector by $\delta_J=\delta^-_J+\delta^+_J$ with $\delta^-_J=\inf\, J$ and $\delta^+_J=\sup\, J$. We denote its effective rank by $\rbar=r-\norm{\delta^-_J}$. We write $T_J$ for its universal tract, $P_J$ for its universal pasture, and $F_J$ for its foundation. 

Since the canonical maps $F_J\to T_J$ and $P_J \to T_J$ are both injective (cf.\ \autoref{thm: bijection between the universal pasture and the universal tract}), a morphism $T_J\to T_{J'}$ between the universal tracts of two M-convex sets $J$ and $J'$ restricts to at most one morphism $P_J\to P_{J'}$ between the respective universal pastures and to at most one morphism $F_J\to F_{J'}$ between the respective foundations. For this reason, we begin with the description of the morphism $T_J\to T_{J'}$ in the following results.

\subsection{Minor embeddings}
\label{subsection: representations of minor embeddings}

When we say that $J\minor\nu\mu+\tau$ is an embedded minor of $J$ in the following, we assume that $\nu,\mu\in\N^n$ and $\tau\in\Z^n$ with $\tau\geq-\delta^-_J$, and that there is an $\alpha\in J$ such that $\mu+\delta^-_J\leq\alpha\leq\delta_{J/\mu}-\nu$.

Recall from \autoref{subsection: embedded minors} that the embedded minor $J\minor\nu\mu+\tau$ comes with the minor embedding
\[
 \begin{array}{cccc}
  \iota_{J\minor\nu\mu+\tau}\colon & J\minor{\nu}{\mu}+\tau & \longrightarrow & J \\
                              & \alpha                 & \longmapsto     & \alpha+\mu-\tau.
 \end{array}
\]

\begin{thm}\label{thm: foundations of embedded minors}
 Let $J\minor\nu\mu+\tau$ be an embedded minor of $J$. Let $\rbar'=r-\norm\mu-\norm{\delta^-_{J\minor\nu\mu}}$ be the effective rank of $J\minor\nu\mu$. Fix $\bgamma\in[n]^{\rbar'-\rbar}$ with $\sum\bgamma=\delta^-_{J\minor\nu\mu}+\mu-\delta^-_J$. Then the association $x_\bbeta\mapsto x_{\bgamma\bbeta}$ defines a morphism
 \[
  \psi_{J\minor\nu\mu+\tau}\colon \ T_{J\minor\nu\mu+\tau} \ \longrightarrow \ T_J, 
 \]
 which does not depend on the choice of $\bgamma$ and which restricts (uniquely) to morphisms 
 \[
  \psi^w_{J\minor\nu\mu+\tau}\colon \ P_{J\minor\nu\mu+\tau} \ \longrightarrow \ P_J \qquad \text{and} \qquad \varphi_{J\minor\nu\mu+\tau}\colon \ F_{J\minor\nu\mu+\tau} \ \longrightarrow \ F_J.
 \]
 In the case of a translation (i.e., if $\nu=\mu=0$), all three morphisms are isomorphisms.
\end{thm}

\begin{proof}
 We can separate the formation of an embedded minor into deletion, contraction, and translation, which allows us to prove the claim for these cases separately. This is easiest for translations: we have $\delta_{J+\tau}^-=\delta^-_J+\tau$ and $\rbar'=\rbar$. Thus the association of the theorem is the identity map $x_\bbeta\mapsto x_\bbeta$. From the shape of the Pl\"ucker relations, it is evident that this map identifies $\widehat T_{J+\tau}$ with $\widehat T_J$ as tautologically isomorphic tracts. This isomorphism preserves (multi)degrees and the $3$-term Pl\"ucker relations, and thus restricts to isomorphisms $T_{J+\tau}\to T_J$, \ $P_{J+\tau}\to P_J$, and $F_{J+\tau}\to F_J$.
 
 As the next case we consider contractions $J/\mu$. By \autoref{prop: bounds on the duality vectors of minors}, $\delta_{J/\mu}=\delta^-_J$, and thus $\Sigma\bgamma=\mu$. We claim that the association $x_\bbeta\mapsto x_{\bgamma\bbeta}$ defines a multiplicative map 
 \[
  \hat\psi \ = \ \hat\psi_{J/\mu}\colon \ \widehat T_{J/\mu} \ \longrightarrow \ \widehat T_J.
 \]

 Firstly note that $x_{\bgamma\bbeta}\neq0$ for $\Sigma\bbeta\in J/\mu$ since $\beta\mapsto\beta+\mu$ defines an injection $J/\mu\to J$. We need to show that $\hat\psi^+\colon\widehat T_{J/\mu}^+\to \widehat T_J^+$ restricts to the respective nullsets, which can be tested on generators, i.e.,\ elements of $T_{J/\mu}^+$ of the form
 \begin{equation} \label{Pl-x}
\Pl(\balpha|i_0\dotsc i_s|j_2\dotsc j_s) :=
 \sum_{k=0}^s \ (-1)^k \cdot x_{\balpha i_0\dotsc \widehat{i_k}\dotsc i_s} \cdot x_{\balpha i_k j_2\dotsc j_s}
 \end{equation}
 for $2\leq s\leq \rbar'$, $\balpha\in\Delta^{\rbar'-s}_n$ and $i_0,\dotsc,i_s,j_2,\dotsc,j_s\in[n]$ such that $\Sigma\balpha i_0\dotsc i_s j_2\dotsc j_s\leq\omega_{J/\mu}$. 
 
 Note that every $\alpha\in J$ with $\alpha\geq\delta^-_J+\mu$ is in the image of $\iota_{J/\mu}\to J$. Thus a variable $x_\bbeta$ that appears in \eqref{Pl-x} is nonzero if and only if $x_{\bgamma\bbeta}$ is nonzero in $\widehat T_J$. Further, we have
 \[
  \Sigma \bgamma\balpha i_0\dotsc i_s j_2\dotsc j_s \ \leq \ \omega_{J/\mu} +\mu \ \leq \ \omega_J
 \]
 by \autoref{prop: bounds on the duality vectors of minors}. 
 So the image of \eqref{Pl-x} under $\hat\psi$ is the Pl\"ucker relation
 \begin{equation*}
 \sum_{k=0}^s \ (-1)^k \cdot x_{\bgamma\balpha i_0\dotsc \widehat{i_k}\dotsc i_s} \cdot x_{\bgamma\balpha i_k j_2\dotsc j_s},
 \end{equation*}
 which is in $N_{\widehat T_J}$.  This shows that the map $\hat\psi\colon\widehat T_{J/\mu}\to \widehat T_J$ is a morphism of tracts.
 
 Note that $\psi$ depends on the choice of $\bgamma$ up to a sign, due to axiom \ref{SR2}. Since $\psi$ is degree preserving, it restricts to the tract morphism $\psi_{J/\mu}\colon T_{J/\mu}\to T_J$, which does not depend on the choice of $\bgamma$. From the above argument, it is clear that $\hat\psi^+$ sends $3$-term Pl\"ucker relations to $3$-term Pl\"ucker relations, giving a tract morphism $\psi^w\colon P_{J/\mu}\to P_J$ between the respective universal pastures.
 
 Since $P_{J/\mu}$ is generated by elements of the form $x_{\bbeta}/x_{\bbeta'}$ and the association $x_{\bbeta}/x_{\bbeta'}\mapsto x_{\bgamma\bbeta}/x_{\bgamma\bbeta'}$ is invariant under the rescaling action by $T(F)=(F^\times)^n$, the morphism $\psi^w\colon P_{J/\mu}\to P_J$ restricts to a morphism $\phi\colon F_{J/\mu}\to F_J$ between the respective foundations.

Finally, for deletions $J\setminus\nu$, the argument is analogous to that of contractions. In this case, $\Sigma\bgamma=\delta^-_{J\setminus\nu}-\delta^-_J$. We claim that the association $x_\bbeta\mapsto x_{\bgamma\bbeta}$ defines a tract morphism $\hat\psi\colon\widehat T_{J\setminus\nu}\to\widehat T_J$. 
  
 Since $J\setminus\nu\subseteq J$, $\hat\psi$ maps nonzero elements $x_\bbeta$ of $\widehat{T}_{J\setminus\nu}$ to nonzero elements $x_{\bgamma\bbeta}$ of $\widehat T_J$. We are left with verifying that $\hat\psi^+$ maps the generators
 \begin{equation} \label{Pl-y}
\Pl(\balpha|i_0\dotsc i_s|j_2\dotsc j_s) :=
  \sum_{k=0}^s \ (-1)^k \cdot x_{\balpha i_0\dotsc \widehat{i_k}\dotsc i_s} \cdot x_{\balpha i_k j_2\dotsc j_s}
 \end{equation}
 of the nullset of $\widehat T_{J\setminus\nu}$ to the nullset of $\widehat T_J$. Since 
 \[
  \Sigma\bgamma\balpha i_0\dotsc i_s j_2\dotsc j_s \ \leq \ \delta^-_{J\setminus\nu}-\delta^-_J+\delta^+_{J\setminus\nu}-\delta^-_{J\setminus\nu} \ = \ \delta^+_J-\delta^-_J-\nu
 \]
 by \autoref{prop: bounds on the duality vectors of minors}, we conclude that an element $x_\bbeta$ in \eqref{Pl-y} is nonzero in $\widehat T_{J\setminus\nu}$ if and only if  $x_{\bgamma\bbeta}$ is nonzero in $\widehat T_J$, and that 
 \[\Sigma\bgamma\balpha i_0\dotsc i_s j_2\dotsc j_s\leq\omega_J.\] This shows that $\hat\psi^+$ maps \eqref{Pl-y} to the Pl\"ucker relation
 \begin{equation*}
 \sum_{k=0}^s \ (-1)^k \cdot x_{\bgamma\balpha i_0\dotsc \widehat{i_k}\dotsc i_s} \cdot x_{\bgamma\balpha i_k j_2\dotsc j_s}
 \end{equation*}
 in $N_{\widehat T_J}$. This shows that the map $\hat\psi\colon\widehat T_{J\setminus\nu}\to \widehat T_J$ is a tract morphism. 
 
 The same arguments as in the case of contractions show that $\hat\psi$ induces tract morphisms $\psi\colon T_{J\setminus\nu}\to T_J$, $\psi^w\colon P_{J\setminus\nu}\to P_J$, and $\varphi\colon F_{J\setminus\nu}\to F_J$, which are independent of the choice of $\bgamma$. 
\end{proof}

Let $J\minor\nu\mu+\tau$ be an embedded minor of $J$ and let $F$ be a tract. Applying $\Hom(-,F)$ to the morphisms from \autoref{thm: foundations of embedded minors} yields maps between the respective (strong and weak) thin Schubert cells and realization spaces by \autoref{prop: universal property of the universal tract}, \autoref{prop: universal property of the universal pasture}, and \autoref{prop: universal property of the foundation}, respectively:
\begin{align*}
 \psi_{J\minor\nu\mu+\tau}\colon \ & T_{J\minor\nu\mu+\tau} \ \longrightarrow \ T_J & \text{yields} \quad \psi^\ast_{J\minor\nu\mu+\tau}\colon & \Gr_J(F) \ \longrightarrow \ \Gr_{J\minor\nu\mu+\tau}(F);\\
 \psi^w_{J\minor\nu\mu+\tau}\colon \ & P_{J\minor\nu\mu+\tau} \ \longrightarrow \ P_J & \text{yields} \quad \psi^{w,\ast}_{J\minor\nu\mu+\tau}\colon & \Gr^w_J(F) \ \longrightarrow \ \Gr^w_{J\minor\nu\mu+\tau}(F);\\
 \varphi_{J\minor\nu\mu+\tau}\colon \ & F_{J\minor\nu\mu+\tau} \ \longrightarrow \ F_J & \text{yields} \quad \varphi^\ast_{J\minor\nu\mu+\tau}\colon & \ulineGr^w_J(F) \ \longrightarrow \ \ulineGr^w_{J\minor\nu\mu+\tau}(F).\\
\end{align*}
We denote the image of a class $[\brho]$ in a (strong or weak) thin Schubert cell or in the realization space under the corresponding map by $[\brho]\minor\nu\mu+\tau$, and call it an \emph{embedded minor of $[\brho]$}.

\subsection{Change of coordinates}

In this section, we show that combinatorially equivalent M-convex sets have isomorphic universal tracts and foundations. We have established this already for translations in \autoref{thm: foundations of embedded minors}. 

\begin{prop}\label{prop: foundations of coordinate inclusions}
 Let $\iota_n\colon\N^n\to\N^{n+1}$ be the embedding into the first $n$ coordinates, $J'=\iota_n(J)$ and $\biota_n\colon[n]^r\to[n+1]^r$ the tautological embedding. Then the association $x_\balpha\mapsto x_{\biota_n(\balpha)}$ defines an isomorphism $T_J\to T_{J'}$, which restricts (uniquely) to isomorphisms $P_J\to P_{J'}$ and $F_J\to F_{J'}$.
\end{prop}

\begin{proof}
 We claim that the association $x_\balpha\mapsto x_{\biota_n(\balpha)}$ defines an isomorphism of tracts $\hat\eta\colon\widehat T_J\to\widehat T_{J'}$. Since $\iota_n\colon J\to J'$ is a bijection, $x_\balpha$ is nonzero in $\widehat T_J$ if and only if $x_{\biota_n(\balpha)}$ is nonzero in $\widehat T_{J'}$. The image of the Pl\"ucker relation
 \[
  \sum_{k=0}^s \ (-1)^k \cdot x_{\balpha i_0\dotsc \widehat{i_k}\dotsc i_s} \cdot x_{\balpha i_k j_2\dotsc j_s}
 \]
 in $N_{\widehat T_J}$ under $\hat\eta^+$ is the Pl\"ucker relation
 \[
  \sum_{k=0}^s \ (-1)^k \cdot x_{\biota_n(\balpha i_0\dotsc \widehat{i_k}\dotsc i_s)} \cdot x_{\biota_n(\balpha i_k j_2\dotsc j_s)}
 \]
 in $N_{\widehat T_{J'}}$, where we note that $\omega_{J'}=\iota_n(\omega_J)$ and thus $\Sigma\balpha i_0\dotsc i_sj_2\dotsc j_s\leq\omega_J$ if and only if $\Sigma\biota_n(\balpha i_0\dotsc i_sj_2\dotsc j_s)\leq\omega_{J'}$. This verifies that $\hat\eta:\widehat T_J\to \widehat T_{J'}$ is an isomorphism of tracts. It is clearly degree preserving and thus restricts to an isomorphism $T_J\to T_{J'}$, as well as isomorphisms $P_J\to P_{J'}$ (since it preserves 3-term Pl\"ucker relations) and $F_J\to F_{J'}$ (since it preserves the multi degree).
\end{proof}

\begin{prop}\label{prop: foundations of permutations}
 Let $\sigma\colon\N^n\to\N^n$ be a coordinate permutation, $\bsigma\colon[n]^r\to[n]^r$ the induced bijection, and $J'=\sigma(J)$. Then the association $x_\balpha\mapsto x_{\bsigma(\balpha)}$ defines an isomorphism $T_J\to T_{J'}$, which restricts (uniquely) to isomorphisms $P_J\to P_{J'}$ and $F_J\to F_{J'}$.
\end{prop}

\begin{proof}
 We claim that the association $x_\balpha\mapsto x_{\bsigma(\balpha)}$ defines an isomorphism of tracts $\hat\tau\colon\widehat T_J\to\widehat T_{J'}$. Since $\sigma\colon J\to J'$ is a bijection, $x_\balpha$ is nonzero in $\widehat T_J$ if and only if $x_{\bsigma(\balpha)}$ is nonzero in $\widehat T_{J'}$. The image of the Pl\"ucker relation
 \[
  \sum_{k=0}^s \ (-1)^k \cdot x_{\balpha i_0\dotsc \widehat{i_k}\dotsc i_s} \cdot x_{\balpha i_k j_2\dotsc j_s}
 \]
 in $N_{\widehat T_J}$ under $\hat\tau^+$ is the Pl\"ucker relation
 \[
  \sum_{k=0}^s \ (-1)^k \cdot x_{\bsigma(\balpha i_0\dotsc \widehat{i_k}\dotsc i_s)} \cdot x_{\bsigma(\balpha i_k j_2\dotsc j_s)}
 \]
 in $N_{\widehat T_{J'}}$, where we note that $\omega_{J'}=\sigma(\omega_J)$ and thus $\Sigma\balpha i_0\dotsc i_sj_2\dotsc j_s\leq\omega_J$ if and only if $\Sigma\bsigma(\balpha i_0\dotsc i_sj_2\dotsc j_s)\leq\omega_{J'}$. This verifies that $\hat\tau\colon\widehat T_J\to \widehat T_{J'}$ is an isomorphism of tracts. It is clearly degree preserving and thus restricts to an isomorphism $T_J\to T_{J'}$, as well as isomorphisms $P_J\to P_{J'}$ (since it preserves 3-term Pl\"ucker relations) and $F_J\to F_{J'}$ (since it preserves the multi degree). 
\end{proof}

Applying $\Hom(-,F)$ to these isomorphisms yields due to \autoref{prop: universal property of the universal tract}, \autoref{prop: universal property of the universal pasture} and \autoref{prop: universal property of the foundation} canonical bijections
\begin{align*}
 \Gr_{\iota_n(J)}(F) \ &\simeq \ \Gr_{J}(F), & \Gr_{\sigma(J)}(F) \ &\simeq \ \Gr_{J}(F), \\
 \Gr^w_{\iota_n(J)}(F) \ &\simeq \ \Gr^w_{J}(F), & \Gr^w_{\sigma(J)}(F) \ &\simeq \ \Gr^w_{J}(F), \\
 \ulineGr^w_{\iota_n(J)}(F) \ &\simeq \ \ulineGr^w_{J}(F), & \ulineGr^w_{\sigma(J)}(F) \ &\simeq \ \ulineGr^w_{J}(F)
\end{align*}
between the respective (strong and weak) thin Schubert cells and realization spaces.

\begin{cor}\label{cor: functoriality of representations space in polymatroid embeddings}
 The spaces and maps
 \[
  \begin{tikzcd}[column sep=40pt]
   \Gr_J(F) \ar[>->,r] & \Gr^w_J(F) \ar[->>,r] & \ulineGr^w_J(F)
  \end{tikzcd}
 \]
 are functorial in tract morphisms $F\to F'$ and polymatroid embeddings $J\to J'$.
\end{cor}

\begin{proof}
 The functoriality in $F$ has been established in \autoref{subsection: functoriality} and \autoref{subsection: the realization space}. The functoriality in polymatroid embeddings follows from \autoref{thm: foundations of embedded minors}, \autoref{prop: foundations of coordinate inclusions}, and \autoref{prop: foundations of permutations}.
\end{proof}

\begin{cor}\label{cor: combinatorially equivalent polymatroids have isomorphic foundations}
 Two combinatorially equivalent M-convex sets have isomorphic universal tract, universal pastures, and foundations.
\end{cor}

\begin{proof}
 By \autoref{thm: foundations of embedded minors}, \autoref{prop: foundations of coordinate inclusions}, and \autoref{prop: foundations of permutations}, two elementary equivalent M-convex sets have isomorphic universal tract, universal pastures, and foundations. The result follows by composing such isomorphisms.
\end{proof}

\subsection{Duality}

Let $\omega_J=\delta^-_J-\delta^+_J$ be the width of $J$ and $d=\norm{\omega_J}$. For $\bbeta\in[n]^d$, we define the \emph{signature of $\bbeta$} as
\[
 \sign\bbeta \ = \ \sign \sigma \ \in \ T_J,
\]
where $\sigma\in S_{d}$ is a permutation such that $\bbeta_{\sigma(1)}\leq\dotsc\leq\bbeta_{\sigma(d)}$. The signature of $\bbeta$ is well-defined, since:

\begin{itemize}
\item $\sigma$ is uniquely determined by strict inequalities between the $\bbeta_i$ if $J$ is the translate of a matroid;
\item otherwise, $1=-1$ by the idempotency principle (\autoref{prop: idempotency principle}), and thus $\sign\bbeta=1$, independently of the choice of $\sigma$. 
\end{itemize}

\begin{thm}\label{thm: foundations of duals}
 Let $J$ be M-convex, $\rbar$ its effective rank, and $d=\norm{\omega_J}$. The association $x_{\bbeta}\mapsto \sign(\bbeta\bbeta^\ast)\cdot x_{\bbeta^\ast}$, where $\bbeta^\ast\in[n]^{d-\rbar}$ satisfies $\Sigma\bbeta^\ast=\delta_J-\Sigma\bbeta$, defines an isomorphism $T_J\to T_{J^\ast}$, which is independent of the choices of the $\bbeta$'s and which restricts (uniquely) to isomorphisms $P_J\to P_{J^\ast}$ and $F_J\to F_{J^\ast}$.
\end{thm}

\begin{proof}
 Since $\alpha\mapsto\delta_J-\alpha$ defines a bijection $J\to J^\ast$, the association $x_{\bbeta}\mapsto \sign(\bbeta\bbeta^\ast)\cdot x_{\bbeta^\ast}$ with $\Sigma\bbeta^\ast=\delta_J-\Sigma\balpha$ defines a bijection between the variables of $\widehat T_J$ and $\widehat T_{J^\ast}$, where we note that axiom \ref{SR2} identifies $\sign(\bbeta\bbeta^\ast)\cdot x_{\bbeta^\ast}$ with $\sign(\bbeta\tilde\bbeta^\ast)\cdot x_{\tilde\bbeta^\ast}$ if $\tilde\bbeta^\ast$ is another choice of element in $[n]^{d-\rbar}$ with $\Sigma\tilde\bbeta^\ast=\delta_J-\Sigma\balpha$.
 
 Consider the Pl\"ucker relation
 \[
  \sum_{k=0}^s \ (-1)^k \cdot x_{\balpha i_0\dotsc \widehat{i_k}\dotsc i_s} \cdot x_{\balpha i_k j_2\dotsc j_s} \quad \in \quad \widehat T_J
 \]
 with $2\leq s\leq \rbar$, $\balpha\in[n]^{\rbar-s}$, $i_0,\dotsc,i_s,j_,\dotsc,j_s\in[n]$ with $\Sigma\balpha i_0\dotsc i_sj_2\dotsc j_s\leq\omega_J$. Let $\balpha'\in[n]^{d-\rbar-s}$ with $\Sigma\balpha'=\omega_J-\Sigma\balpha i_0\dotsc i_sj_2\dotsc j_s$. Then, evidently, $\Sigma\balpha' i_0\dotsc i_sj_2\dotsc j_s\leq\omega_J$ and the association $x_\bbeta\mapsto \sign(\bbeta\bbeta^\ast)\cdot x_{\bbeta^\ast}$ sends $x_{\balpha i_0\dotsc\widehat{i_k}\dotsc i_s}$ to $\eta\cdot x_{i_kj_2\dotsc j_s\balpha'}$ and $x_{\balpha i_kj_2\dotsc j_s}$ to $\eta\cdot x_{i_0\dotsc\widehat{i_k}\dotsc i_s\balpha'}$, where $\eta=\sign(\balpha i_0\dotsc i_sj_2\dotsc j_s\balpha')$. Thus the above Pl\"ucker relation corresponds to the Pl\"ucker relation
 \[
  \sum_{k=0}^s \ (-1)^k \cdot x_{\balpha' i_k j_2\dotsc j_s} \cdot x_{\balpha' i_0\dotsc \widehat{i_k}\dotsc i_s} \quad \in \quad \widehat T_{J^\ast}.
 \]
 This shows that the null sets of $\widehat T_J$ and $\widehat T_{J^\ast}$ agree, which establishes the desired isomorphism $\widehat T_J\simeq\widehat T_{J^\ast}$.
 
 Since this isomorphism is degree preserving, it restricts to an isomorphism $T_J\simeq T_{J^\ast}$. Since the $3$-term Pl\"ucker relations of $T_J$ correspond to the $3$-term Pl\"ucker relations of $T_{J^\ast}$, this isomorphism restricts further to an isomorphism $P_J\simeq P_{J^\ast}$. Since an element $\prod x_{\bbeta_j}^{e_j}$ of $P_J$ has multidegree zero if and only if its image $\prod x_{\bbeta_j^\ast}^{e_j}$ in $P_{J^\ast}$ has multidegree zero, $P_J\simeq P_{J^\ast}$ restricts to an isomorphism $F_J\simeq F_{J^\ast}$, which concludes the proof.
\end{proof}

Applying $\Hom(-,F)$ to these isomorphisms yields, by \autoref{prop: universal property of the universal tract}, \autoref{prop: universal property of the universal pasture} and \autoref{prop: universal property of the foundation}, canonical bijections
\[
 \Gr_{J^\ast}(F) \ \simeq \ \Gr_J(F), \qquad \Gr^w_{J^\ast}(F) \ \simeq \ \Gr^w_J(F), \qquad \ulineGr^w_{J^\ast}(F) \ \simeq \ \ulineGr^w_J(F).
\]

\subsection{Direct sums}
\label{subsection: direct sums of representations}

Thin Schubert cells and realization spaces of direct sums of M-convex sets decompose into products, as detailed in the following result.

\begin{thm}\label{thm: representations of direct sums}
 Let $J_1\subseteq\Delta^{r_1}_{n_1}$ and $J_2\subseteq\Delta^{r_2}_{n_2}$ be M-convex and $J=J_1\oplus J_2\subseteq\Delta^r_n$. Then there are canonical bijections
 \begin{align*}
  T_J \ &\simeq \ T_{J_1}\otimes T_{J_2}, & \Gr_J(F)\ &\simeq \ \Gr_{J_1}(F) \times \Gr_{J_2}(F), \\
  P_J \ &\simeq \ P_{J_1}\otimes P_{J_2}, & \Gr^w_J(F)\ &\simeq \ \Gr^w_{J_1}(F) \times \Gr^w_{J_2}(F), \\
  F_J \ &\simeq \ F_{J_1}\otimes F_{J_2}, & \ulineGr^w_J(F)\ &\simeq \ \ulineGr^w_{J_1}(F) \times \ulineGr^w_{J_2}(F),
 \end{align*}
 which are tract isomorphisms (left column) and functorial in the tract $F$ (right column), respectively.
\end{thm}

\begin{proof}
 The canonical isomorphism $T_J\simeq T_{J_1}\otimes T_{J_2}$ is equivalent to the functorial bijection $\Gr_J(F)\simeq \Gr_{J_1}(F) \times \Gr_{J_2}(F)$, since $\Gr_J(F)=\Hom(T_J,F)$ and
 \[
  \Gr_{J_1}(F) \times \Gr_{J_2}(F) \ = \ \Hom(T_{J_1},F) \times \Hom(T_{J_2},F) \ = \ \Hom(T_{J_1}\otimes T_{J_2},F)
 \]
 by \autoref{prop: universal property of the universal tract}. The analogous equivalence holds for the other claims of the proposition.
 
 We now establish the claims in the right-hand column.
We begin with the bijection $\Gr_J(F)\simeq \Gr_{J_1}(F) \times \Gr_{J_2}(F)$. Consider representations $\brho_1\colon[n_1]^{r_1}\to F$ and $\brho_2\colon[n_2]^{r_2}\to F$ of $J_1$ and $J_2$, respectively, and define $\brho\colon[n]^r\to F$ as the function that satisfies axiom \ref{SR2},
 \[
  \brho(\balpha\bbeta') \ = \ \brho_1(\balpha)\cdot\brho_2(\bbeta)
 \]
 whenever $\balpha\in[n_1]^{r_1}$, $\bbeta\in[n_2]^{r_2}$ and $\bbeta'_\ell=\bbeta_\ell+n_1$ for all $\ell\in[r_2]$, and $\brho(\bgamma)=0$ whenever $\#\{\ell\in[r]\mid \bgamma_\ell\in[n_1]\}\neq r_1$. It is immediate from the definitions that the function $\brho$ satisfies axioms \ref{SR1} and \ref{SR2} of a strong $F$-representation of $J$. Consider the Pl\"ucker relation
 \[
  \sum_{k=0}^s \ (-1)^k \cdot \brho(\balpha i_0\dotsc\widehat{i_k}\dotsc i_s) \cdot \brho(\balpha i_kj_2\dotsc j_s) \in N_F,
 \]
 which contains a nontrivial term only if there is a $k$ such that
 \[
  \#\big\{\ell\in[r] \, \big| \, \balpha i_0\dotsc\widehat{i_k}\dotsc i_s)_\ell\in[n_1] \big\} \ = \ \# \big\{ \ell\in[r] \, \big| \, (\balpha i_kj_2\dotsc j_s)_\ell\in[n_1] \big\} \ = \ r_1.
 \]
 Depending on whether $i_k\leq n_1$ or $i_k>n_1$, this Pl\"ucker relation is equivalent to a corresponding Pl\"ucker relation for $\brho_1$ or $\brho_2$, respectively. Carrying out this comparison carefully leads to the conclusion that the Pl\"ucker relations for $\brho$ are equivalent to the Pl\"ucker relations for $\brho_1$ and $\brho_2$.

 Therefore, the $F^\times$-class $[\brho]$ is in $\Gr_J(F)$, and every class in $\Gr_J(F)$ stems from a unique pair of classes $[\brho_1]\in\Gr_{J_1}(F)$ and $[\brho_2]\in \Gr_{J_2}(F)$, which establishes the bijection $\Gr_J(F)\simeq \Gr_{J_1}(F) \times \Gr_{J_2}(F)$. It is evident that this bijection is functorial in $F$.

 The canonical bijection $\Gr_J(F)^w\simeq \Gr^w_{J_1}(F) \times \Gr^w_{J_2}(F)$ can be established analogously (one only considers the $3$-term Pl\"ucker relations for $s=2$). The canonical bijection  $\ulineGr^w_J(F)\simeq \ulineGr^w_{J_1}(F) \times \ulineGr^w_{J_2}(F)$ follows from this, since it is invariant under the action of the torus $(F^\times)^n\simeq(F^\times)^{n_1}\times(F^\times)^{n_2}$. 
\end{proof}

\section{Generators and relations for the foundation}
\label{section: generators and relations for the foundation}

A fundamental result about the foundation $F_M$ of a matroid $M$ is that it is generated as a tract over $\Funpm$ by the cross ratios of $M$. In this section we generalize this result to polymatroids. Moreover,  in the matroid case, we know a complete system of relations between the cross ratios, which determines the foundation. We show that these relations extend to the polymatroid case.

\subsection{Cross ratios}
\label{subsection: cross ratios}

Let $J\subseteq\Delta^r_n$ be an M-convex set and $\bJ=\{\bbeta\in[n]^\rbar\mid\sum\bbeta\in \Jbar\}$. Consider $\balpha\in[n]^{\rbar-2}$ and $i,j,k,l\in[n]$ such that $\balpha ik,\ \balpha jk,\ \balpha il,\ \balpha jl\in\bJ$. Then the element
\[
 \cross ijkl\balpha \ = \ \frac{x_{\balpha ik}\cdot x_{\balpha jl}}{x_{\balpha il}\cdot x_{\balpha jk}}
\]
of $\widehat P_J$ is invertible and has multidegree $\deg_{[n]}\big(\cross ijkl\balpha\big)=0$; thus it is contained in $F_J^\times$. Note that a permutation of the coefficients of $\alpha$ leads to a simultaneous sign change of all 4 terms in the definition of $\cross ijkl\balpha$, which shows that this element only depends on $\alpha=\sum\balpha\in\Delta^{\rbar-2}_n$.

\begin{df}
 Let $\Omega_J$ be the collection of all tuples $(\alpha,i,j,k,l)$ with $\alpha\in\Delta_n^{\rbar-2}$ and $i,j,k,l\in[n]$ such that all of 
 \[
  \alpha+\epsilon_i+\epsilon_k, \qquad \alpha+\epsilon_i+\epsilon_l, \qquad \alpha+\epsilon_j+\epsilon_k, \qquad \alpha+\epsilon_j+\epsilon_l
 \]
  are in $J$. We call $(\alpha,i,j,k,l)\in\Omega_J$ \emph{non-degenerate} if also $\alpha+\epsilon_i+\epsilon_j$ and $\alpha+\epsilon_k+\epsilon_l$ are in $J$; otherwise we call $(\alpha,i,j,k,l)$ \emph{degenerate}. We define $\Omega_J^\diamond\subseteq\Omega_J$ as the subset of all non-degenerate elements. 
  
  Let $(\alpha,i,j,k,l)\in\Omega_J$ and $\balpha\in[n]^{\rbar-2}$ with $\sum\balpha=\alpha$. The \emph{cross ratio for $(\alpha,i,j,k,l)$} is the element
  \[
   \cross ijkl\balpha \ = \ \frac{x_{\balpha ik}\cdot x_{\balpha jl}}{x_{\balpha il}\cdot x_{\balpha jk}} \quad \in \quad F_J^\times.
  \]
\end{df}

\subsection{Generators and relations}\label{subsection: generators and relations for the foundation}

A deep structural result, based on Tutte's homotopy theorem, shows that the foundation of a matroid is generated by cross ratios and exhibits a complete system of relations between them (cf. \cite[Thm.\ 4.21]{Baker-Lorscheid20}). We generalize this result to polymatroids.

\begin{thm}\label{thm: generators and relations for the foundation}
 Let $J\subseteq\Delta^r_n$ be an M-convex set. Then the foundation of $J$ is generated by the cross ratios $\cross ijkl\alpha$ with $(\alpha,i,j,k,l)\in\Omega_J^\diamond$ over $\Funpm$, and the following relations between the cross ratios generate all additive and multiplicative relations between these generators in the foundation of $J$:
 \begin{enumerate}
  \item[\mylabel{CRs}{(CR$\sigma$)}] 
  If $(\alpha,i,j,k,l)\in\Omega_J^\diamond$, then $\big(\alpha,\sigma(i),\sigma(j),\sigma(k),\sigma(l)\big)\in\Omega_J^\diamond$ for every permutation $\sigma\in S_4$ and 
  \[
   \cross ijkl\alpha \ = \ \cross klij\alpha \ = \ \cross jilk\alpha \ = \ \cross lkji\alpha.
  \]
  \item[\mylabel{CR-}{(CR-)}] 
  If $J$ has an embedded minor that is isomorphic or dual to the Fano matroid, then
  \[
   1 \ = \ -1.
  \]
  If $J$ is a proper polymatroid, then
  \[
   1 \ = \ -1.
  \]
  \item[\mylabel{CR+}{(CR+)}] 
  If $(\alpha,i,j,k,l)\in\Omega_J^\diamond$, then 
  \[
   \cross ijkl\alpha \ + \ \cross ikjl\alpha \ - \ 1 \quad \in \quad N_{F_J}.
  \]
  \item[\mylabel{CR0}{(CR0)}] 
  If $(\alpha,i,j,k,l)\in\Omega_J$ is degenerate, then 
  \[
   \cross ijkl\alpha \ = \ 1. 
  \]
  \item[\mylabel{CR1}{(CR1)}] 
  If $(\alpha,i,j,k,l)\in\Omega_J^\diamond$, then
  \[
   \cross ijkl\alpha \ \cdot \ \cross ijlk\alpha \ = \ 1.
  \]
  \item[\mylabel{CR2}{(CR2)}] 
  If $(\alpha,i,j,k,l)\in\Omega_J^\diamond$, then
  \[
   \cross ijkl\alpha \ \cdot \ \cross iklj\alpha \ \cdot \ \cross iljk\alpha \ = \ -1.
  \]
  \item[\mylabel{CR3}{(CR3)}] 
  If $(\alpha,i,j,k,l),(\alpha,i,j,l,m),(\alpha,i,j,m,k)\in\Omega_J$, then
  \[
   \cross ijkl\alpha \ \cdot \ \cross ijlm\alpha \ \cdot \ \cross ijmk\alpha \ = \ 1.
  \]
  \item[\mylabel{CR4}{(CR4)}] 
  If $(\alpha +\epsilon_m,i,j,k,l),(\alpha +\epsilon_k,i,j,l,m),(\alpha +\epsilon_l,i,j,m,k)\in\Omega_J$ for $\alpha\in\Delta_n^{\rbar-3}$, then
  \[
   \cross ijkl{\alpha +\epsilon_m}\ \cdot \ \cross ijlm{\alpha +\epsilon_k} \ \cdot \ \cross ijmk{\alpha +\epsilon_l} \ = \ 1.
  \]
  \item[\mylabel{CR5}{(CR5)}] 
  If $(\alpha +\epsilon_p,i,j,k,l),(\alpha +\epsilon_q,i,j,k,l)\in\Omega_J^\diamond$ for $\alpha\in\Delta_n^{\rbar-3}$ and if both $(\alpha +\epsilon_i,k,l,p,q)$, $(\alpha +\epsilon_j,k,l,p,q)\in\Omega_J$ are degenerate, then
  \[
   \cross ijkl{\alpha +\epsilon_p} \ = \ \cross ijkl{\alpha +\epsilon_q}.
  \]
 \end{enumerate}
\end{thm}

Throughout this subsection, we write $N$ for the natural matroid of the reduction $\Jbar$. We prove \autoref{thm: generators and relations for the foundation} by making use of a canonical tract morphism $F_N \to F_J$. 

We write $E$ for the ground set of $N$ and denote the associated projection by $\theta\colon\Z^E\to\Z^n$. We also use $\theta$ for the map $E\to[n]$ obtained by identifying the standard basis vectors with the corresponding coordinates. Let $\bN:=\{\bbeta\in E^\rbar\mid \Sigma\bbeta\in N\}$.

Recall that the extended universal pasture of $J$ is
\[ 
 \widehat P_J = \pastgen{\Funpm(x_\balpha\mid\balpha\in \bJ)}{S_J}, 
\] 
where $\bJ = \{\balpha\in[n]^\rbar\mid \Sigma\balpha\in \Jbar\}$ and the set $S_J$ consists of the $3$-term Pl\"ucker relations along with the relations $\sign(\sigma) x_{i_{\sigma(1)} \dots i_{\sigma(\rbar)}} = x_{i_1 \dots i_\rbar}$ for $(i_1,\dots,i_\rbar)\in\bJ$ and $\sigma\in S_\rbar$.

The canonical morphism $F_N \to F_J$ is constructed as follows. We define a tract morphism $\Funpm(y_\bbeta\mid\bbeta\in\bN) \to \Funpm(x_\balpha\mid\balpha\in \bJ)$ given by $-1 \mapsto -1$ and $y_\balpha \mapsto x_{\theta(\balpha)}$ for each $\balpha\in\bN$. Then we obtain a morphism $\Funpm(y_\bbeta\mid\bbeta\in\bN) \to \widehat P_J = \pastgen{\Funpm(x_\balpha\mid\balpha\in \bJ)}{S_J}$. Each element in the preimage $\theta^{-1}(S_J)$ is generated by $S_N \cup R$, where $R$ is the set of all relations $y_{\bbeta} = y_{\bbeta'}$ with $\theta(\bbeta) = \theta(\bbeta')$. Thus we deduce an isomorphism $\pastgen{\Funpm(y_\bbeta\mid\bbeta\in\bN)}{S_N \cup R} \simeq \widehat P_J$ and the following commutative diagram:
\[
  \begin{tikzcd}[column sep=50, row sep=25]
  F_N \ar[r,hook] \ar[d] & \widehat P_N \ar[r,"\deg_{E}"] \ar[d] & \Z^E \ar[d] \\
  F_J \ar[r,hook] & \widehat P_J \ar[r,"\deg_{[n]}"] & \Z^n
  \end{tikzcd}
\]

\begin{lemma}\label{lemma: surjectivity of tract morphisms from N to J}
  The tract morphisms $\widehat P_N \to \widehat P_J$ and $F_N \to F_J$ are surjective.
\end{lemma}
\begin{proof}
  The morphism $\widehat P_N \to \widehat P_J$ is surjective since the pre-composition with the quotient map $\Funpm(y_\bbeta\mid\bbeta\in\bN) \to \widehat P_N$ is surjective.  

  For $\prod x_{\balpha_i} \in F_J$, let $\bbeta_i\in \bN$ be such that $\theta(\bbeta_i) = \balpha_i$. Then $\sum_{j\in E_k} \deg_E (\prod y_{\bbeta_i})_j = 0$ for each $k\in[n]$. By multiplying elements of the form $y_{\bgamma a} / y_{\bgamma b}$ with $\theta(a) = \theta(b)$ together with $\prod y_{\bbeta_i}$, we obtain an element in $F_N$ that maps to $\prod x_{\balpha_i}$. Thus $F_N \to F_J$ is surjective.
\end{proof}

Note that $F_N\to F_J$ is an isomorphism if $J$ is the translate of a matroid. Thus we may assume in the following that $J$ is a proper polymatroid, and thus $1=-1$ in $F_J$ and $\widehat P_J$. This means that the kernel of $\widehat P_N^\times \to \widehat P_J^\times$ contains $-1$. More specifically, the kernel, denoted by $\fR$, is generated by $-1$ and the elements of the form $y_{\bbeta i} / y_{\bbeta j}$ with $\bbeta \in E^{\rbar-1}$ and $\theta(i) = \theta(j)$. By \cite[Thm.\ 4.21]{Baker-Lorscheid20}, $F_N^\times$ is generated by $-1$ and the cross ratios.

\begin{lemma}\label{lemma: specific cross ratios for N}
  If $J$ is proper, the intersection $F_N^\times \cap \fR$ is generated by $-1$ and the cross ratios of the form $\cross{i}{j}{k}{l}{\alpha}$ with $\theta(k)=\theta(l)$.
\end{lemma}
\begin{proof}
  First note that the cross ratio \[ \cross{i}{j}{k}{l}{\alpha} = \frac{y_{\alpha i k} \cdot y_{\alpha j l}}{y_{\alpha i l} \cdot y_{\alpha j k}} \] is indeed in $\fR$ if $\theta(k) = \theta(l)$. 
  
  Let $Y$ be an element of $\fR$ whose multidegree is $0$. Up to multiplication by $-1$, the element $Y$ can be written as a product of $m$ elements of the form $y_{\bbeta i} / y_{\bbeta j}$ with $\theta(i) = \theta(j)$, along with cross ratios of the form $\cross{i}{j}{k}{l}{\alpha}$ with $\theta(k) = \theta(l)$. We proceed by induction on $m$ to show that $Y$ is generated by cross ratios $\cross{i}{j}{k}{l}{\alpha}$ with $\theta(k) = \theta(l)$. 
  
  We may assume that $m$ is positive. 
  Pick a multiplicative factor $y_{\bbeta i} / y_{\bbeta j}$ with $\theta(i)=\theta(j)$ from such an expression for $Y$. 
  If $i=j$, then we can remove this factor and apply the induction hypothesis. Thus we may assume that $i\ne j$. Then there is another factor $y_{\bgamma j} / y_{\bgamma k}$ because $\deg_E (Y) = 0$. We may assume that $\Sigma \beta \ne \Sigma \gamma$, since otherwise $(y_{\bbeta i} / y_{\bbeta j}) \cdot (y_{\bgamma j} / y_{\bgamma k}) = \pm y_{\bbeta i} / y_{\bbeta k}$. By the exchange axiom, there are $p, q\in [n]$ such that $(\Sigma\bbeta)_p > (\Sigma\bgamma)_p$, $(\Sigma\bgamma)_q > (\Sigma\bbeta)_q$, and $(\Sigma\bbeta - \epsilon_p + \epsilon_q) + \epsilon_j \in N$. Thus \[ \frac{y_{\bbeta i}}{y_{\bbeta j}} = \cross pqij{\Sigma\bbeta - \epsilon_p} \cdot \frac{y_{\bbeta' i}}{y_{\bbeta' j}}, \] where $\bbeta' \in E^{\rbar-1}$ with $\Sigma \bbeta' = \Sigma\bbeta - \epsilon_p + \epsilon_q$. By applying such relations inductively, we deduce that \[ \frac{y_{\bbeta i}}{y_{\bbeta j}} \cdot \frac{y_{\bgamma j}}{y_{\bgamma k}} = \delta \cdot \prod \cross {p_s}{q_s}{i}{j}{\alpha_s} \cdot \frac{y_{\bgamma i}}{y_{\bgamma k}} \]
  for some $\delta\in\{1,-1\}$, $p_s$ and $q_s \in E$, and $\alpha_s \in \Delta_n^{\rbar-2}$. We obtain another expression for $Y$ having only $m-1$ factors of the form $y_{\bbeta i} / y_{\bbeta j}$ with $\theta(i) = \theta(j)$. Thus the induction hypothesis applies, which completes the inductive step.
\end{proof}

\begin{proof}[Proof of \autoref{thm: generators and relations for the foundation}]
  We may assume that $J$ is a proper polymatroid, as the matroid case is covered by \cite[Thm.\ 4.21]{Baker-Lorscheid20}.

  By \cite[Thm.\ 4.21]{Baker-Lorscheid20}, the foundation $F_N$ of the natural matroid $N$ is generated over $\Funpm$ by the non-degenerate cross ratios $\cross ijkl\beta \in \widehat P_N$ with $(\alpha,i,j,k,l) \in\Omega_N^\diamond$. Because the canonical morphism $F_N \to F_J$ is surjective by \autoref{lemma: surjectivity of tract morphisms from N to J} and each non-degenerate cross ratio for $N$ maps to a non-degenerate cross ratio for $J$, we conclude that $F_J$ is generated by the non-degenerate cross ratios over $\Funpm$.

 
  The relations \ref{CRs} and \ref{CR1}--\ref{CR4} follow from a direct verification. Relation \ref{CR-} is proven for matroids in \cite[Thm.\ 4.21]{Baker-Lorscheid20}; for proper polymatroids it follows from the idempotency principle \autoref{prop: idempotency principle}. In order to show \ref{CR+} and \ref{CR0}, we divide the $3$-term Pl\"ucker relation
  \[
    x_{\balpha ik} \cdot x_{\balpha jl} \ - \ x_{\balpha il} \cdot x_{\balpha jk} \ + \ x_{\balpha ij} \cdot x_{\balpha kl} \quad \in \quad N_{\widehat P_J}
  \]
  by $x_{\balpha il} \cdot x_{\balpha jk}$ (where we assume that $i\leq j\leq l\leq k$; the other cases are similar). This yields the relation \ref{CR+} if $(\alpha,i,j,k,l)$ is non-degenerate (where $\alpha=\sum\balpha$) and the relation
  \[
    \cross ijkl\alpha \ - \ 1 \quad \in \quad N_{F_J}
  \]
  if $(\alpha,i,j,k,l)$ is degenerate, which establishes\ref{CR0}. 
  
  Conversely, the $3$-term relations \ref{CR+} generate the null set of the foundation, which follows at once from the definition of the extended universal pasture in terms of $3$-term Pl\"ucker relations and generalizes to polymatroids; cf.\ the proof of \autoref{lemma: the universal tract is freely generated over the foundation}. 
  
  Relation \ref{CR5} follows from the direct computation
  \[
    \cross ijkl{\alpha +\epsilon_p} \cdot \crossinv ijkl{\alpha +\epsilon_q} \ = \ \frac{x_{\balpha ikp} \cdot x_{\balpha jlp} \cdot x_{\balpha ilq} \cdot x_{\balpha jkq}}{x_{\balpha ilp} \cdot x_{\balpha jkp} \cdot x_{\balpha ikq} \cdot x_{\balpha jlq}} \ = \ \cross klpq{\alpha +\epsilon_i} \cdot \cross klqp{\alpha +\epsilon_j} \ = \ 1,
  \]
  where we apply \ref{CR1} to express the inverse of $\cross ijkl{\alpha +\epsilon_q}$ and \ref{CR0} to identify the latter two degenerate cross ratios with $1$.

  It remains to show that the relations \ref{CRs}, \ref{CR-}, and \ref{CR0}--\ref{CR5} generate all multiplicative relations between cross-ratios. More precisely, we claim that the kernel of the natural embedding $f_J\colon \Funpm\big(\cross ijkl\alpha \mid (\alpha,i,j,k,l) \in \Omega_J^\diamond\big)^\times \to F_J^\times$ is generated by the elements corresponding to the relations \ref{CRs}, \ref{CR-}, and \ref{CR0}--\ref{CR5}, where we regard the cross ratios in the domain as symbols. For instance, $\cross ijkl\alpha \cdot \crossinv klij\alpha$ is one of the six elements corresponding to relation \ref{CRs}.
  
  By mapping a non-degenerate cross ratio $\cross ijkl\beta$ for $N$ to a non-degenerate cross ratio $\Theta\big(\cross ijkl\beta\big) := \cross {\theta(i)}{\theta(j)}{\theta(k)}{\theta(l)}{\theta(\beta)}$ for $J$, we deduce the following commutative diagram:
  \[
    \begin{tikzcd}[column sep=50, row sep=25]
    \ker f_N \ar[r,hook] \ar[d] & \Funpm\big(\cross ijkl\beta \mid (\beta,i,j,k,l) \in \Omega_N^\diamond\big)^\times \ar[r,"f_N"] \ar[d,"\Theta"] & F_N^\times \ar[d] \\
    \ker f_J \ar[r,hook] & \Funpm\big(\cross ijkl\alpha \mid (\alpha,i,j,k,l) \in \Omega_J^\diamond\big)^\times \ar[r,"f_J"] & F_J^\times
    \end{tikzcd}
  \]
  Notice that the middle and right vertical maps are surjective, but that the left vertical map $\ker f_N \to \ker f_J$ does not need to be surjective.

  Let $\prod X_i$ be an element in $\ker f_J$, where each $X_i$ is a non-degenerate cross ratio for $J$. There are non-degenerate cross ratios $Y_i$ for $N$ such that $\Theta(Y_i) = X_i$ for each $i$. Since $f_J(\prod X_i) = 1$, we have $f_N(\prod Y_i) \in \fR$. Therefore, by \autoref{lemma: specific cross ratios for N}, $f_N(\prod Y_i)$ can be written as a product of cross ratios of the form $\cross{a}{b}{c}{c'}{\beta}$ with $\theta(c) = \theta(c')$, possibly up to a factor $-1$. Thus there are elements $Z_1,\dots,Z_l$ such that each $Z_i$ is either $-1$ or a cross ratio of the form $\cross{a}{b}{c}{c'}{\beta}$ with $\theta(c) = \theta(c')$ and $\prod Y_i \cdot \prod Z_j \in \ker f_N$. By \cite[Thm.\ 4.21]{Baker-Lorscheid20}, $\prod Y_i \cdot \prod Z_j \in \ker f_N$ is generated by the elements for the corresponding relations \ref{CRs}, \ref{CR-}, and \ref{CR0}-\ref{CR5} for $N$. It is straightforward to see that such elements for $N$ map to the elements standing for multiplicative relations \ref{CRs}, \ref{CR-}, and \ref{CR0}-\ref{CR5} for $J$. Furthermore, $\Theta(Z_i)$ is either $-1$ or a cross ratio of the form $\cross {a}{b}{c}{c}{\alpha}$ for $J$. The element $-1$ corresponds to relation \ref{CR-} because $J$ is proper. If $\cross {a}{b}{c}{c}{\alpha}$ is degenerate, then it is an element corresponding to relation \ref{CR0}. If $\cross {a}{b}{c}{c}{\alpha}$ is non-degenerate, then it is a product of elements corresponding to relations \ref{CR-}, \ref{CR2}, and \ref{CR1} as follows:
  \[
    \cross abcc\alpha = (-1) \cdot \left( - \cross abcc\alpha \cdot \cross accb\alpha \cdot \cross acbc\alpha \right) \cdot \left( \cross accb \alpha \cdot \cross acbc\alpha \right)^{-1}.
  \]
  Therefore $\prod X_i$ is generated by the elements corresponding to the relations \ref{CRs}, \ref{CR-}, and \ref{CR0}-\ref{CR5} for $J$, which completes the proof.
\end{proof}

\subsection{Non-degenerate cross ratios}
\label{subsection: non-degenerate cross ratios}

Since degenerate cross ratios are equal to $1$, the foundation is generated by the non-degenerate cross ratios (\autoref{thm: generators and relations for the foundation}). Non-degenerate cross ratios stem from certain types of embedded minors of $J$, which we classify in this section. This has been done in the matroid case (cf.\ \cite{Baker-Lorscheid20}): every non-degenerate cross is in the image of the canonical map $F_{J\minor\nu\mu}\to F_J$ for a minor $J\minor\nu\mu$ of type $U_{2,4}$.
In the polymatroid case, we find two additional types. Recall that
\[
 U_{2,3}^+ \ = \ \big\{ (2,0,0),\ (1,1,0),\ (1,0,1),\ (0,1,1) \big\}.
\]
In \autoref{subsubsection: foundation of U23+ and D22}, we determine the foundation of both $U_{2,3}^+$ and $\Delta^2_2$ as $\pastgen{\K(x)}{x+1+1}$.

\begin{prop}\label{prop: nondegenerate cross ratios}
 A cross ratio $\cross ijkl\alpha$ in $F_J$ that is not equal to $1$ is in the image of the canonical map $F_{J'}\to F_J$ for an embedded minor $J'=J\minor\nu\mu+\tau$ of $J$ that is of type $U_{2,4}$, $U_{2,3}^+$, or $\Delta^2_2$.
\end{prop}

\begin{proof}
 Let $\cross ijkl\alpha$ be a cross ratio of $F_J$ that is not equal to $1$. Whenever we find an embedded minor $J'$ of $J$ such that $\cross ijkl\alpha$ is in the image of the canonical map $F_{J'}\to F_J$, we can replace $J$ by $J'$ until we have arrived at one of the three M-convex sets in the claim of  the proposition.
 
 To begin with, the cross ratio $\cross ijkl\alpha$ in $F_J$ is the image of the cross ratio $\cross ijkl{}\in F_{J/\alpha}$ under the canonical map $\varphi_{J/\alpha}\colon F_{J/\alpha}\to F_J$, which allows us to assume that $\alpha=0$ and that $J$ is of rank $2$. 

 By \ref{CR0}, $\cross ijkl{}$ has to be non-degenerate, i.e., $J$ contains 
 \[
  S \ = \ \big\{ \ \epsilon_i+\epsilon_j, \ \ \epsilon_i+\epsilon_k, \ \ \epsilon_i+\epsilon_l, \ \ \epsilon_j+\epsilon_k, \ \ \epsilon_j+\epsilon_l, \ \ \epsilon_k+\epsilon_l \ \big\}.
 \]
 Any other element of $J$ does not matter, which means that $\cross ijkl{}$ lies in the image of $F_{J\setminus\nu}\to F_J$ for $\nu=\delta^+_J-\sup S$. This allows us to assume that $\delta^+_J=\sup S$ and, after permuting $[n]$ and restricting the support suitably, that $[n]=\{i,j,k,l\}$.

 Since 
 \[
  \cross ijkl \ = \ \frac{x_{ik}\cdot x_{jl}}{x_{il}\cdot x_{jk}}
 \]
 is not equal to $1$, we can assume that $i\neq j$ and that $k\neq l$. If $i$, $j$, $k$ and $l$ are pairwise distinct, then $J=U_{2,4}$.
 
 If two of $i$, $j$, $k$ and $l$ are equal, then we can apply \ref{CRs} and \ref{CR1} (i.e., passing to the multiplicative inverse) and assume that $i\neq j=k\neq l$. In this case, $J$ is one of the last two M-convex sets of the proposition, depending on whether $i\neq l$ or $i=l$.
\end{proof}

\subsection{More examples of foundations}
\label{subsection: more examples of foundations}

In this section we compute the foundations of some proper polymatroids. 

\subsubsection{The foundation of \texorpdfstring{$U_{2,3}^+$}{U23+} and of \texorpdfstring{$\Delta^2_2$}{Delta22}}
\label{subsubsection: foundation of U23+ and D22}

Consider $J=U_{2,3}^+$ or $J=\Delta^2_2$, which is a proper polymatroid in either case. Thus its foundation $F_J$ is near-idempotent by \autoref{prop: idempotency principle}. Consider a cross ratio 
\[
 \cross ijkl{} \ = \ \frac{x_{ik}\cdot x_{jl}}{x_{il}\cdot x_{jk}} \quad \in \quad F_J
\]
of $J$ that is not equal to $1$. Then we have $i\neq j$ and $k\neq l$ and it is non-degenerate by \ref{CR0}, and so are all cross-ratios obtained by permuting $i,j,k,l$ by \ref{CRs}. After reordering the rows (using \ref{CRs}) and possibly exchanging $k$ and $l$ (which exchanges $\cross ijkl{}$ by its inverse by \ref{CR1}), we can assume that $i\neq j=k\neq l$. This implies that $i\neq l$ if $J=U_{2,3}$ and $i=l$ if $J=\Delta^2_2$. 
Finally, since
 \[
  \cross ijjl{}= \cross ljji{}
 \]
by \ref{CRs}, we find that the foundation is generated by $\cross ijjl{}$, or, equivalently, by its multiplicative inverse $x=\cross ijlj{}$. Note that the degree of $x_{jj}$ in $\cross ijjl{}=\frac{x_{ij}x_{jl}}{x_{il}x_{jj}}$ maps the powers of $x$ bijectively to $\Z$, which shows that there are no further multiplicative relations.
 
 The unique Pl\"ucker relation of $J$ is
 \[
  x_{il}\cdot x_{jj} \ + \ x_{ij}\cdot x_{jl} \ + \ x_{ij}\cdot x_{jl} \quad \in \quad N_{\widehat P_J},
 \]
 which is equivalent to $1+1+x\in N_{F_J}$, after dividing by $x_{ij}\cdot x_{jl}$. We therefore find that
 \[
  F_J \ = \ \pastgen{\F_2(x)}{1+1+x} \ \simeq \ \F_2\otimes\D
 \]
 for both $J=U_{2,3}^+$ and $J=\Delta^2_2$.
 
\subsubsection{The foundation of \texorpdfstring{$\Delta^2_3\setminus \epsilon_2$}{the one element deletion of Delta23}}

Consider the proper polymatroid
\[
 J \ = \ \Delta^2_3\setminus \epsilon_2 \ = \ \big\{ (2,0,0),\ (1,1,0),\ (1,0,1),\ (0,1,1),\ (0,0,2) \big\},
\]
whose foundation is near-idempotent. By \autoref{prop: nondegenerate cross ratios}, the nontrivial cross ratios stem from the embedded minors of $J$ of types $U_{2,4}$, $U_{2,3}^+$, and $\Delta^2_2$, which are
\begin{align*}
 J\setminus \epsilon_1 \ &= \ \big\{ (1,1,0),\ (1,0,1),\ (0,1,1),\ (0,0,2) \big\}  && \text{(type $U_{2,3}^+$);}  \\
 J\setminus \epsilon_2 \ &= \ \big\{ (2,0,0),\ (1,0,1),\ (0,0,2) \big\}            && \text{(type $\Delta^2_2$);} \\
 J\setminus \epsilon_3 \ &= \ \big\{ (2,0,0),\ (1,1,0),\ (1,0,1),\ (0,1,1) \big\}  && \text{(type $U_{2,3}^+$).}
\end{align*}
The corresponding cross ratios are $x=\cross 2313{}$, $y=\cross 1313{}$, $z=\cross 2131{}$, which satisfy $x+1+1,\ y+1+1,\ z+1+1 \in N_{F_J}$ by \autoref{subsubsection: foundation of U23+ and D22}. By \ref{CR3}, we have
\[
 \cross 1231{} \ \cdot \ \cross 2331{} \ \cdot \ \cross 3131{} \ = \ 1,
\]
and thus (using \ref{CRs} and \ref{CR1}), $y=xz$. There is no further relation between $x$ (which involves $x_{11}$) and $z$ (which involves $x_{33}$). Thus we find that the foundation of $J=\Delta^2_3\setminus \epsilon_2$ is
\[
 F_J \ = \ \pastgen{\F_2(x,z)}{1+1+x,\ 1+1+xz,\ 1+1+z}.
\]

\subsubsection{The foundation of \texorpdfstring{$\Delta^2_3$}{Delta23}}

The polymatroid $J=\Delta^2_3$ has $3$ embedded minors $J\setminus(\epsilon_i+\epsilon_j)$ of type $U_{2,3}^+$ and $3$ embedded minors $J\setminus2\epsilon_k$ of type $\Delta^2_2$, where $\{i,j,k\}=\{1,2,3\}$. We denote the corresponding cross ratios by $x_k=\cross ikjk{}$ and $y_k=\cross ijij{}$. For the same reasons as in the previous example, they satisfy the relations $y_k=x_ix_j$, $x_k+1+1\in N_{F_J}$, and $y_k+1+1\in N_{F_J}$, which determines the foundation $F_J$ of $\Delta^2_3$ as
\[
 \pastgen{\F_2(x_1,x_2,x_3)}{x_1+1+1,\ x_2+1+1,\ x_3+1+1,\ x_1x_2+1+1,\ x_1x_3+1+1,\ x_2x_3+1+1}.
\]

\subsubsection{The foundation of \texorpdfstring{$\Delta^3_2$}{Delta32}}

The polymatroid $J=\Delta^3_2$ has $2$ embedded minors of type $\Delta^2_2$, which are $J/\epsilon_1$ and $J/\epsilon_2$, and none of types $U_{2,3}^+$ or $U_{2,4}$. The corresponding cross ratios are $x=\cross1122{\epsilon_1}$ and $y=\cross1122{\epsilon_2}$, and their respective inverses. 

By the idempotency principle, the foundation $F_J$ of $J$ is idempotent and, in particular, $-1=1$. Since the variables $x_{111}$ and $x_{222}$ only occur in one of $x$ and $y$, taking the multidegree in $x_{111}$ and $x_{222}$ defines an group isomorphism $F_J^\times\to\Z^2$. 

The Pl\"ucker relations are parametrized by $\alpha\in\{\epsilon_1,\epsilon_2\}$ and $i,j,k,l\in\{1,2\}$ such that $\alpha+\epsilon_i+\epsilon_j+\epsilon_k+\epsilon_l\leq\omega_J=(3,3)$. If $\alpha=\epsilon_1$, then up to permutation of $i,j,k,l$, we have either $i=1$ and $j=k=l=2$ or $i=j=1$ and $k=l=2$. In the former case, the Pl\"ucker relation is equivalent to $1+1+1\in N_{F_J}$ (which we know already from the idempotency principle), and the latter relation is equivalent to $1+1+x\in N_{F_J}$. Similarly the Pl\"ucker relations for $\alpha=\epsilon_2$ yield $1+1+1\in N_{F_J}$ and $1+1+y\in N_{F_J}$. Thus we find
\[
 F_{\Delta^3_2} \ = \ \pastgen{\K(x,y)}{1+1+x,\ 1+1+y}.
\]

\part{Canonical embeddings}

In this last part of the paper, we describe canonical embeddings of the representation space $\upR^w_J(F)$ and of the realization space $\ulineGr^w_J(F)$ into tori, by which we mean sets of the form $(F^\times)^N$ for some $N$. This is of particular interest if $F$ (and thus $F^\times$) carries a topology, since the torus embeddings then endow the representation / realization spaces with a topology (cf.\ \cite{BHKL1} and \cite{BHKL2}, where such a study is carried out in the case of triangular hyperfields).

Further, we discuss the Pl\"ucker embedding $\pl_J\colon\Gr^w_J(F)\to\P^N(F)$ (for $N=\#\Delta^r_n-1$) and a decomposition of $\upR^w_J(F)$ into the product of $\ulineGr^w_J(F)$ with a torus.

\section{The torus embedding of the representation space}
\label{section: the torus embedding of the representation space}

We call a group isomorphic to $(F^\times)^s$ (for some $s\geq0$) a \emph{torus over $F$}. 
Let $J\subseteq\Delta^r_n$ be an M-convex set, $\Jbar$ its reduction, and $\rbar$ its effective rank. By the definition of a strong (resp.\ weak) $F$\hyph representation of $J$ as a function $\brho\colon[n]^\rbar\to F$, the representation space $\upR_J(F)$ (resp.\ $\upR^w_J(F)$) can be considered as a subspace of $F^{n^\rbar}$. Axiom \ref{SR1} (or, equivalently, \ref{WR1}) determines the non-vanishing coordinates of $\brho$ as those whose indices $\alpha$ lie in $\bJ=\{\alpha\in[n]^\rbar\mid \alpha_1+\dotsb+\alpha_r\in \Jbar\}$. This results in the canonical embedding 
\[
 \upR_J(F) \ \subseteq \ \upR^w_J(F) \ \longrightarrow \ (F^\times)^\bJ
\]
of the (strong resp.\ weak) representation space into the torus $(F^\times)^\bJ$. Axiom \ref{SR2} (or, equivalently, \ref{WR2}) implies that $\upR_J(F)$ (resp.\ $\upR^w_J(F)$) is contained in the subgroup $C_\bJ(F)$ of $(F^\times)^\bJ$ that consists of all $\brho\in(F^\times)^\bJ$ that satisfy\[\brho({i_{\sigma(1)},\dotsc,i_{\sigma(r)}})=\sign(\sigma)\cdot\brho({i_1,\dotsc,i_r}),\] which is a subgroup of $(F^\times)^\bJ$.
Choosing an $\balpha\in\bJ$ with $\sum\balpha=\alpha$ for each $\alpha\in \Jbar$, e.g.\ the unique such $\balpha$ with $\balpha_1\leq\dotsc\leq\balpha_r$, yields the torus embedding $\upR^w_J(F)\hookrightarrow (F^\times)^\Jbar$. 
Note that if $1\neq-1$ in $F$, then the signs of the coordinates depend on the choices of the $\balpha$.

\subsection{The degeneracy locus}
\label{subsection: the degeneracy locus}

There is yet a smaller subgroup of $\upD_J(F)\subseteq(F^\times)^\bJ$ that contains the representation space. It is cut out by the degenerate $3$-term Pl\"ucker relations, which force the two nonzero terms to be additive inverses of each other. These relations are of the form
\[
 \brho({\alpha ik}) \cdot \brho({\alpha jl}) \ = \ \brho({\alpha il}) \cdot \brho({\alpha jk})
\]
for $\alpha\in\Delta^{\rbar-2}_n$, assuming that $\brho({\alpha ij}) \cdot \brho({\alpha kl})=0$.
The \emph{degeneracy locus of $J$ over $F$} is defined as the subgroup 
\[
 \upD_J(F) \ = \ \Big\{ \brho\in(F^\times)^\bJ \ \Big| \ \begin{array}{c}\brho\text{ satisfies \ref{WR2} and all}\\\text{degenerate $3$-term Pl\"ucker relations}\end{array}\Big\}
\]
of $(F^\times)^\bJ$. Summing up, this yields a chain of inclusions
\[
 \upR_J(F) \ \subseteq \ \upR^w_J(F) \ \subseteq \ \upD_J(F) \ \subseteq \ (F^\times)^\bJ \ \subseteq \ F^{n^r}.
\]
The following fact verifies that the degeneracy locus does not get smaller if we require \emph{all} degenerate Pl\"ucker relations to hold.

\begin{cor}\label{cor: the degeneracy locus is defined by the degenerate 3-term Plucker relations}
 Let $J$ be an M-convex set of $F$ a tract. Then
 \[
  \upD_J(F) \ = \ \Big\{ \brho\in(F^\times)^\bJ \ \Big| \ \brho\text{ satisfies \ref{WR2} and all degenerate Pl\"ucker relations}\Big\}.
 \]
\end{cor}

\begin{proof}
 This follows at once from \autoref{thm: bijection between the universal pasture and the universal tract}: every degenerate Pl\"ucker relation is contained in the ideal generated by the degenerate $3$-term Pl\"ucker relations.
\end{proof}

\begin{prop}\label{prop: characterization of the degeneracy locus as group homomorphisms from the extended universal pasture}
 Let $J$ be an M-convex set with extended universal pasture $\widehat P_J$ and let $F$ be a tract. Then there is a canonical bijection
 \[
  \upD_J(F) \ \longrightarrow \ \big\{ \text{group homomorphisms }f\colon\widehat P_J^\times\to F^\times\text{ with }f(-1)=-1\big\}.
 \]
 If $F$ is a degenerate tract, then $\upD_J(F)=\upR_J^w(F)=\upR_J(F)$.
\end{prop}

\begin{proof}
 Let $G$ be the (multiplicatively written) free abelian group generated by symbols $-1$ and $x_\beta$ for $\beta\in\bJ$. An element $\brho\in \upD_J(F)$ defines a group homomorphism $f_\brho\colon G\to F^\times$ with $f_\brho(-1)=-1$ and $f_\brho(x_\beta)=\brho(\beta)$. Since $(f_\brho(-1))^2=(-1)^2=1$ and
 \[
  f_\brho\bigg(\frac{x_{\alpha ik} \cdot x_{\alpha jl}}{x_{\alpha il} \cdot x_{\alpha jk}}\bigg) \ = \ \frac{\brho({\alpha ik}) \cdot \brho({\alpha jl})}{\brho({\alpha il}) \cdot \brho({\alpha jk})} \ = \ 1
 \]
 whenever $\brho({\alpha ik}) \cdot \brho({\alpha jl})$ and $\brho({\alpha il}) \cdot \brho({\alpha jk})$ are the two nonzero terms of a degenerate $3$-term Pl\"ucker relation, the group homomorphism $f_\brho$ factors through a uniquely determined group homomorphism
 \[
  \bar f_\brho\colon \widehat P_J^\times \ = \ G/H_J^w \ \longrightarrow \ F
 \]
 with $\bar f_\brho(-1)=-1$, where $H_J^w$ is the subgroup of $G$ generated by the elements $(-1)^2$ and the degenerate cross ratios $\frac{x_{\alpha ik} \cdot x_{\alpha jl}}{x_{\alpha il} \cdot x_{\alpha jk}}$. 

 This defines the canonical map given in the statement of the proposition. It is injective since $\brho$ can be recovered from $\bar f_\brho$ via $\brho_\beta=\bar f_\brho(x_\beta)$. It is surjective since all defining relations between the generators $x_\beta$ of $\widehat P_J^\times$ hold for the coefficients $\brho_\beta$ of $\brho\in \upD_J(F)$.
 
 The claim $\upD_J(F)=\upR^w_J(F)$ follows from the fact that the non-degenerate $3$-term Pl\"ucker relations are vacuous if $F$ is degenerate. By \autoref{cor: degenerate tracts are excellent}, $F$ is excellent, i.e., $\upR^w_J(F)=\upR_J(F)$.
\end{proof}

\subsection{The lineality space}
\label{subsection: the lineality space}

If $F$ is idempotent, then $\upR_J(F)$ contains a certain subgroup of the ambient torus $(F^\times)^\bJ$, which we call the lineality space.

Recall from \autoref{subsection: the realization space} that the torus $\widehat{T}(F)=F^\times\times (F^\times)^n$ acts on $\upR_J(F)$ by the formula $(a,t).\brho(\balpha)=a\cdot\big(\prod_{i=1}^\rbar t_{\balpha_i}\big)\cdot\rho(\balpha)$.

If $F$ is idempotent, then there is a (necessarily unique) morphism $i_F\colon\K\to F$. The composition of the unique $\K$-representation $\bchi_J\colon[n]^\rbar\to\K$ of $J$ with $i_F$ yields the \emph{trivial $F$-representation} $\bchi_{J,F}\colon[n]^\rbar\to F$ given by $\bchi_{J,F}(\balpha)=1$ if $\sum\balpha\in\Jbar$ and $\bchi_{J,F}(\balpha)=0$ otherwise.

\begin{df}
 The \emph{lineality space of $\upR_J(F)$} is the orbit $\upLin_J(F)=\widehat T(F).\bchi_{J,F}$, which is a subgroup of $(F^\times)^\bJ$ that is contained in $\upR_J(F)$. The thin Schubert cell $\Gr_J(F)=\upR_J(F)/F^\times$ contains the quotient torus $\widehat{T}(F).\bchi_{J,F}/F^\times$, which we call the \emph{lineality space of $\Gr_J(F)$}.
\end{df}

\begin{ex}\label{ex: lineality space of the Dressian}
 If $J$ is a matroid, then the bijection $-\log\colon\Gr_J(F)\to\Dr_J$ (cf.\ \autoref{subsection: M-convex functions as representations over the tropical hyperfield}) identifies the lineality space of $\Gr_J(\T_0)$ with the lineality space of the local Dressian $\Dr_J$, which consists of all valuated matroids with underlying matroid $J$ (cf.\ \cite{Brandt-Speyer22} for details on local Dressians and their lineality spaces).
\end{ex}

Whether or not a weak $F$-representation $\brho\colon[n]^\rbar\to F$ belongs to the lineality space $\upLin_J(F)$ can be evaluated in terms of the triviality of the \emph{cross ratios of $\brho$}, which are the elements
\[
 \cross ijkl{\alpha,\brho} \ = \ \frac{\brho(\balpha ik)\cdot \brho(\balpha jl)}{\brho(\balpha il)\cdot \brho(\balpha jk)} \quad \in \quad F^\times
\]
for $(\alpha,i,j,k,l)\in\Omega_J$ and $\balpha\in\Delta^{\rbar-2}_n$ with $\sum\balpha=\alpha$.
If $f_\brho\colon\widehat P_J\to F$ is the tract morphism associated with $\brho$ (see \autoref{prop: universal property of the universal pasture}), which is given by $f_\brho(x_\bbeta)=\brho(\beta)$, then
\[
 f_\brho\Big(\cross ijkl\alpha \Big) \ = \ f_\brho \Big(\frac{x_{\balpha ik}\cdot x_{\balpha jl}}{x_{\balpha il}\cdot x_{\balpha jk}}\Big) \ = \ \frac{\brho(\balpha ik)\cdot \brho(\balpha jl)}{\brho(\balpha il)\cdot \brho(\balpha jk)} \ = \ \cross ijkl{\alpha,\brho}.
\]

\begin{prop}\label{prop: characterization of the lineality space in terms of cross ratios}
 Let $F$ be idempotent, let $\brho\in\upR^w_J(F)$, and let $f_\brho\colon\widehat P_J\to F$ the associated tract morphism. Then $\brho\in \upLin_J(F)$ if and only if $\cross ijkl{\alpha,\brho}=1$ for all $(\alpha,i,j,k,l)\in\Omega_J$.
\end{prop}

\begin{proof}
    It is clear from the definition that the cross ratios of $\brho$ are invariant under the action of $\widehat T(F)$. The ``only if'' direction now follows from the obvious fact that the cross ratios of the trivial $F$-representation are all equal to one.

  Conversely, we have $\brho\in \upLin_J(F)=\widehat T(F).\bchi_{J,F}$ if and only if $[\brho]=[\bchi_{J,F}]$ as classes of the realization space $\ulineGr^w_J(F)$, i.e., if and only if the restriction of $f_\brho$ to $F_J\to F$ agrees with the tract morphism $f_{\bchi_{J,F}}\colon F_J\to F$, using \autoref{prop: universal property of the foundation}. 
 Now if $\cross ijkl{\alpha,\brho}=1$ for all $(\alpha,i,j,k,l)\in\Omega_J$, then $f_\brho|_{F^\times_J}\colon F_J^\times\to F$ is indeed trivial because $F_J^\times$ is generated by $-1$, whose image in $F$ is $-1=1$, and all cross ratios by \autoref{thm: generators and relations for the foundation}.
\end{proof}

\section{The Pl\"ucker embedding for thin Schubert cells}
\label{section: the Plucker embedding for thin Schubert cells}

We define the \emph{projective $N$-space over $F$} as the quotient 
\[
 \P^N(F) \ = \ \big\{ (a_0,\dotsc,a_N)\in F^{N+1} \, \big| \, a_i\in F^\times\text{ for some }0\leq i\leq N \big\} \ \big/ \ F^\times.
\]
We denote elements of $\P^N(F)$ by 
\[
 [a_1:\dotsc:a_N] \ := \ F^\times\cdot(a_0,\dotsc,a_N).
\]

Quotienting out the domain and the codomain of the canonical embedding $\upR^w_J(F)\to F^{n^r}$ (cf.\ \autoref{section: the torus embedding of the representation space}) by the diagonal action of $F^\times$ yields the \emph{effective Pl\"ucker embedding}
\[
 \Gr_J(F) \ \subseteq \ \Gr^w_J(F) \ \longrightarrow \ \P^{n^\rbar-1}(F) \ = \ \big\{ [x_\balpha] \, \big| \, \balpha\in[n]^\rbar \big\},
\]
which sends $\big[\brho\colon[n]^\rbar\to F\big]\in\Gr^w_J(F)$ to $[\brho(\balpha)]_{\balpha\in[n]^\rbar}$.

\subsection{The Polygrassmannian for idempotent fusion tracts}
\label{subsection: the Polygrassmannian for idempotent fusion tracts}

The thin Schubert cells $\Gr_J(F)$ for various M-convex subsets $J\subseteq\Delta^r_n$ (with fixed $r$ and $n$) glue to a subvariety of $\P^{n^r-1}$, which we call the Polygrassmannian. If $F$ is not near-idempotent, then the idempotency principle (\autoref{prop: idempotency principle}) implies that $\Gr_J(F)$ is nonempty only if $J$ is the translate of a matroid. So in this case, the Polygrassmannian is a union of ``translates'' of usual Grassmannians. If $F$ is near-idempotent, then the Polygrassmannian is larger than (a union of translates of) usual Grassmannians.

In order to realize the embedding of $\Gr_J(F)$ into $\P^{n^r-1}$ (as opposed to $\P^{n^\rbar-1}$, as in the effective Pl\"ucker embedding), we need to ``translate'' $[\brho]\in\Gr_J(F)$ to a function $\brho\colon[n]^r\to F$ by adding coefficients that correspond to $\delta^-_J$ in a certain sense, as explained below. Unless $J$ is a matroid, the signs of the coordinates of the ``translate'' of $\brho$ depends on the ordering of $[n]$ if $1\neq -1$ in $F$. To avoid these complications, we concentrate on the near-idempotent case in the following, which is the only genuinely interesting case as explained above.

Thus, let us assume for the remainder of this discussion that $F$ is near-idempotent. By \autoref{lemma: alternative description of polymatroid representations}, a function $\brho\colon[n]^\rbar\to F$ that satisfies axioms \ref{SR1} and \ref{SR2} is a strong $F$-representation if and only if the function $\rho\colon\Delta^r_n\to F$ with support $J$, given by $\rho(\delta^-_J+\sum\balpha)=\brho(\balpha)$ for $\balpha\in\bJ$, satisfies the Pl\"ucker relations
\begin{equation*}
 \sum_{k=0}^s \ \rho(\alpha -\epsilon_{i_k}+\epsilon_{i_0}+\dotsb+\epsilon_{i_s}) \cdot \rho(\alpha +\epsilon_{i_k}+\epsilon_{j_2}+\dotsb +\epsilon_{j_s})  \in  N_{F}
\end{equation*}
for all $2\leq s\leq r$, $\alpha\in\Delta^{r-s}_n$, and $i_0,\dotsb,i_s,j_2,\dotsc,j_s\in[n]$ such that $\delta^-_J\leq\alpha$ and $\alpha+\epsilon_{i_0}+\dotsc+\epsilon_{i_s}+\epsilon_{j_2}+\dotsc+\epsilon_{j_s}\leq\delta^+_J$.

This identifies $F$-representations of different M-convex subsets $J\subseteq\Delta^r_n$ (for fixed $r$ and $n$) with functions $\rho\colon\Delta^r_n\to F$ on the same domain. However, the defining relations depend on $J$ because of the bounds that involve $\delta^-_J$ and $\delta^+_J$. The following result resolves this issue by removing these bounds. It comes with the price that we need to assume the \emph{fusion rule} for $F$, which is
\begin{enumerate}[label=(FR)]
 \item \label{FR} If $\sum a_i-c,\ c+\sum b_j\in N_F$, then $\sum a_i+\sum b_j\in N_F$.
\end{enumerate}
We call a tract that satisfies the fusion rule \ref{FR} a \emph{fusion tract}. Note that the most interesting examples of tracts are fusion tracts, which include all hyperfields and partial fields and, in particular, all examples mentioned in this text. Note that we also need to assume that $F$ is idempotent.

\begin{lemma}\label{lemma: extended Plucker relations over idempotent F}
 Let $F$ be an idempotent fusion tract. Let $J\subseteq\Delta^r_n$ be an M-convex set, $\Jbar$ its reduction, and $\rbar=r-\norm{\delta^-_J}$ its effective rank. Let $\brho\colon[n]^\rbar\to F$ a function that satisfies axioms \ref{SR1} and \ref{SR2} and $\rho\colon\Delta^r_n\to F$ the function with support $J$ given by $\rho(\delta^-_J+\sum\balpha)=\brho(\balpha)$ for $\balpha\in[n]^\rbar$. Then $\brho$ is a strong $F$-representation of $J$ if and only if $\rho$ satisfies the Pl\"ucker relations
 \begin{equation*}
  \sum_{k=0}^s \ \rho(\alpha -\epsilon_{i_k}+\epsilon_{i_0}+\dotsb+\epsilon_{i_s}) \cdot \rho(\alpha +\epsilon_{i_k}+\epsilon_{j_2}+\dotsb +\epsilon_{j_s})  \in  N_{F}
 \end{equation*}
 for all $2\leq s\leq r$, $\alpha\in\Delta^{r-s}_n$, and $i_0,\dotsb,i_s,j_2,\dotsc,j_s\in[n]$.
\end{lemma}

\begin{proof}
 By \autoref{lemma: alternative description of polymatroid representations}, we only need to verify that $\rho$ satisfies the Pl\"ucker relations in question if $\delta^-_J\not\leq\alpha$ or if $\alpha+\epsilon_{i_0}+\dotsc+\epsilon_{i_s}+\epsilon_{j_2}+\dotsc+\epsilon_{j_s}\not\leq\delta^+_J$. These two cases can be proven analogously (and are, more concisely, related by polymatroid duality). We only explain the proof for $\delta^-_J\not\leq\alpha$, proceeding by induction on $d^+=d^+(\alpha,\delta^-_J)=\sum_{i\in[n]}\max\{0,\delta^-_{J,i}-\alpha_i\}$. 
 
 If $d^+=0$, then $\delta^-_J\leq\alpha$, so the corresponding Pl\"ucker relation holds by assumption. If $d^+>0$, then $\alpha_\ell<\delta^-_{J,\ell}$ for some $\ell\in[n]$ and the expression
 \begin{equation}  \label{Pl-rho}
 \Pl(\alpha|i_0\dotsc i_s|j_2\dotsc j_s) :=
 \sum_{k=0}^s \ \rho(\alpha -\epsilon_{i_k}+\epsilon_{i_0}+\dotsb+\epsilon_{i_s}) \cdot \rho(\alpha +\epsilon_{i_k}+\epsilon_{j_2}+\dotsb +\epsilon_{j_s})
 \end{equation}
 is either identically $0$, and thus in $N_F$, or 
 \[
  \beta \ = \ \alpha+\epsilon_{i_0}+\dotsb+\epsilon_{i_s} \qquad \text{and} \qquad \gamma \ = \ \alpha+\epsilon_{j_2}+\dotsb+\epsilon_{j_s}
 \]
 satisfy $\beta_\ell\geq\delta^-_{J,\ell}$, \ $\gamma_\ell\geq\delta^-_{J,\ell}-1$, and $\beta_\ell+\gamma_\ell\geq 2\delta^-_{J,\ell}$.
 
 We begin with the case where $\gamma_\ell=\delta^-_{J,\ell}-1$, and thus $\beta_\ell\geq\delta^-_{J,\ell}+1\geq\alpha_\ell+2$. Then \eqref{Pl-rho} is of the form
 \begin{multline*}
  \sum_{k\text{ s.t.\ }i_k=\ell} \rho(\alpha -\epsilon_{i_k}+\epsilon_{i_0}+\dotsb+\epsilon_{i_s}) \cdot \rho(\alpha +\epsilon_{i_k}+\epsilon_{j_2}+\dotsb +\epsilon_{j_s}) \\
  = \ \big(\underbrace{1+\dotsb+1}_{m\text{-times}}\big) \cdot \rho(\alpha +\epsilon_{i_0}+\dotsb+\epsilon_{i_s}-\epsilon_\ell) \cdot \rho(\alpha +\epsilon_{\ell}+\epsilon_{j_2}+\dotsb +\epsilon_{j_s})
 \end{multline*}
 for $m=\beta_\ell-\alpha_{\ell}\geq2$, which is in $N_F$ since $F$ is idempotent and thus $m.1=1+\dotsc+1\in N_F$.
 
 In the other case, $\gamma_\ell\geq\delta^-_{J,\ell}\geq\alpha_\ell+1$, we have $\beta_\ell\geq\delta^-_{J,\ell}\geq\alpha_\ell+1$ and thus $\ell$ appears in both $\{i_0,\dotsc,i_s\}$ and $\{j_2,\dotsc,j_s\}$, say $i_0=j_2=\ell$. If $\beta_\ell=\alpha_\ell+1$, then
 \[
  \rho(\alpha+\epsilon_{i_1}+\dotsb +\epsilon_{i_s}) \cdot \rho(\alpha +\epsilon_{i_0}+\epsilon_{j_2}+\dotsb +\epsilon_{j_s}) \ = \ 0
 \]
 and thus \eqref{Pl-rho} equals $\Pl(\alpha+\epsilon_\ell|i_1\dotsc i_s|j_3\dotsc j_s)$, which is in $N_F$ by the inductive hypothesis since $d^+(\alpha+\epsilon_\ell,\delta^-_J)=d^+(\alpha,\delta_J)-1<d^+$. If $\beta_\ell\geq\alpha_\ell+2$, then $\ell$ appears at least twice in $\{i_0,\dotsc,i_s\}$, say $i_0=i_1=\ell$, and 
 \eqref{Pl-rho} is of the form
 \begin{multline*}
  \rho(\beta-\epsilon_\ell) \cdot \rho(\gamma+\epsilon_{\ell}) \ + \ \rho(\beta-\epsilon_\ell) \cdot \rho(\gamma+\epsilon_{\ell}) \\
  \ + \ \sum_{k=2}^s \ \rho(\alpha -\epsilon_{i_k}+\epsilon_{i_0}+\dotsb+\epsilon_{i_s}) \cdot \rho(\alpha +\epsilon_{i_k}+\epsilon_{j_2}+\dotsb +\epsilon_{j_s}),
 \end{multline*}
 which is in $N_F$ by the fusion rule \ref{FR} applied to the two terms 
 \[ 
  \rho(\beta-\epsilon_\ell) \cdot \rho(\gamma+\epsilon_{\ell}) \ + \ \rho(\beta-\epsilon_\ell) \cdot \rho(\gamma+\epsilon_{\ell}) \ + \ \rho(\beta-\epsilon_\ell) \cdot \rho(\gamma+\epsilon_{\ell}), 
 \] 
 which is in $N_F$ since $F$ is idempotent, and
 \[
  \rho(\beta-\epsilon_\ell) \cdot \rho(\gamma+\epsilon_{\ell}) \ + \ \sum_{k=2}^s \ \rho(\alpha -\epsilon_{i_k}+\epsilon_{i_0}+\dotsb+\epsilon_{i_s}) \cdot \rho(\alpha +\epsilon_{i_k}+\epsilon_{j_2}+\dotsb +\epsilon_{j_s}),
 \]
 which is $\Pl(\alpha+\epsilon_\ell|i_1\dotsc i_s|j_3\dotsc j_s)$ and thus in $N_F$ by the inductive hypothesis.
\end{proof}

\begin{df}\label{definition:polygrassmannian}
 Let $F$ be an idempotent fusion tract. The \emph{Polygrassmannian of rank $r$ on $[n]$ over $F$} is the subset $\PolyGr(r,n)(F)$ of $\P^{N}(F)$ (for $N=\#\Delta^r_n-1=\binom{n+r}r-1$) that consists of all classes $[\rho\colon\Delta^r_n\to F]$ of functions that satisfy the Pl\"ucker relations
 \begin{equation*}
  \sum_{k=0}^s \ \rho(\alpha -\epsilon_{i_k}+\epsilon_{i_0}+\dotsb+\epsilon_{i_s}) \cdot \rho(\alpha +\epsilon_{i_k}+\epsilon_{j_2}+\dotsb +\epsilon_{j_s})  \in  N_{F}
 \end{equation*}
 for all $2\leq s\leq r$, $\alpha\in\Delta^{r-s}_n$, and $i_0,\dotsb,i_s,j_2,\dotsc,j_s\in[n]$.
\end{df}
As a consequence of \autoref{lemma: extended Plucker relations over idempotent F} and \autoref{remark: pluecker relations imply m-convex support}, $\PolyGr(r,n)(F)$ is the union of the thin Schubert cells $\Gr_J(F)$ for the various M-convex subsets $J\subseteq\Delta^r_n$. More precisely, the association $\brho\mapsto\rho$ established by \autoref{lemma: extended Plucker relations over idempotent F} defines a bijection
\[
 \coprod_{\substack{J\subseteq\Delta^r_n\\ \text{M-convex}}} \Gr_J(F) \ \longrightarrow \ \PolyGr(r,n)(F).
\]
In particular, $\PolyGr(r,n)(\K)$ is canonically in bijection with the collection of all M-convex subsets $J\subseteq\Delta^r_n$.

\begin{rem}
 In fact, the Pl\"ucker relations endow the Polygrassmannian with the structure of a band scheme $\PolyGr(r,n)$ whose $F$-rational points are naturally identified with $\PolyGr(r,n)(F)$. This perspective can be extended to an interpretation of $\PolyGr(r,n)$ as the fine moduli space of polymatroids, in the vein of \cite{Baker-Lorscheid21b}, in terms of a suitable theory of polymatroid bundles for band schemes.
 
 Note that we refrain from making an analogous definition for the \emph{weak} Polygrassmannian, because the weak Pl\"ucker relations \ref{WR3} do not by themselves imply that the support of $\rho$ is a polymatroid. This leads to a less satisfactory picture from an algebro-geometric point of view: the space of weak $F$-representations of M-convex subsets of $\Delta^r_n$ is not represented by a band scheme. (A similar phenomenon occurs already for matroids, so there is nothing unique in this regard about the case of polymatroids.)
\end{rem}

Let $\theta\colon \Z^{rn} \to \Z^n$ be the function given by $\theta(\alpha)_i = \sum_{j=1}^r \alpha_{r(i-1)+j}$ with $i\in[n]$. For a function $\rho\colon \Delta^r_n \to F$, we define the associated function $\Up\rho\colon \Delta^r_{rn} \to F$ by 
\[
  \Up\rho(\alpha) =
  \begin{cases}
    \rho(\theta\alpha) & \text{if } \alpha \leq \1, \\
    0 & \text{otherwise}.
  \end{cases}
\]
This is a direct analogue of the up operator defined in \autoref{subsection: the up operator} when $F$ is a near-idempotent tract. If the support $J$ of $\rho$ is a polymatroid, the support $N = \theta^{-1}(J) \cap \{0,1\}^{[rn]}$ of $\Up\rho$ is a matroid. Note that $N$ is the natural matroid of $J$ iff the rank of each element $i$ of $J$ is $r$.

\begin{lemma}\label{lemma: extended Plucker relations are preserved under the up operator}
 Let $F$ be an idempotent fusion tract. A function $\rho\colon \Delta^r_n \to F$ satisfies the Pl\"ucker relations  
 \begin{equation*}
  \sum_{k=0}^s \ \rho(\alpha -\epsilon_{i_k}+\epsilon_{i_0}+\dotsb+\epsilon_{i_s}) \cdot \rho(\alpha +\epsilon_{i_k}+\epsilon_{j_2}+\dotsb +\epsilon_{j_s})  \in  N_{F}
 \end{equation*}
 for all $2\leq s\leq r$, $\alpha\in\Delta^{r-s}_n$, and $i_0,\dotsb,i_s,j_2,\dotsc,j_s\in[n]$ if and only if $\Up \rho\colon \Delta^r_{rn} \to F$ satisfies the Pl\"ucker relations 
 \begin{equation*}
  \sum_{k=0}^s \ \Up\rho(\beta -\epsilon_{i_k}+\epsilon_{i_0}+\dotsb+\epsilon_{i_s}) \cdot \Up\rho(\beta +\epsilon_{i_k}+\epsilon_{j_2}+\dotsb +\epsilon_{j_s})  \in  N_{F}
 \end{equation*}
 for all $2\leq s\leq r$, $\beta\in\Delta^{r-s}_{rn}$, and $i_0,\dotsb,i_s,j_2,\dotsc,j_s\in[rn]$.
\end{lemma}
\begin{proof}
  We denote the support of $\rho$ by $J$ and the support of $\Up\rho$ by $N$. 

  Assume that $\rho$ satisfies all Pl\"ucker relations. By \autoref{lemma: extended Plucker relations over idempotent F}, it suffices to show that $\Up\rho$ satisfies $\sum_{k=0}^s \ \Up\rho(\beta -\epsilon_{i_k}+\epsilon_{i_0}+\dotsb+\epsilon_{i_s}) \cdot \Up\rho(\beta +\epsilon_{i_k}+\epsilon_{j_2}+\dotsb +\epsilon_{j_s})  \in  N_{F}$
  for all $2\leq s\leq r$, $\beta\in\Delta^{r-s}_{rn}$, and $i_0,\dotsb,i_s,j_2,\dotsc,j_s\in[rn]$ such that $\beta + \epsilon_{i_0} + \dotsb + \epsilon_{i_s}+ \epsilon_{j_2} + \dotsb + \epsilon_{j_s} \leq \delta^+_N = \1$. In this case, $\Up\rho(\beta - \epsilon_{i_k}+\epsilon_{i_0}+\dotsb+\epsilon_{i_s}) = \rho(\theta(\beta - \epsilon_{i_k}+\epsilon_{i_0}+\dotsb+\epsilon_{i_s}))$ and $\Up\rho(\beta + \epsilon_{i_k}+\epsilon_{j_2}+\dotsb +\epsilon_{j_s}) = \rho(\theta(\beta + \epsilon_{i_k}+\epsilon_{j_2}+\dotsb +\epsilon_{j_s}))$. Therefore, the above Pl\"ucker relations for $\Up\rho$ follow from the Pl\"ucker relations for $\rho$.

  Conversely, assume that $\Up\rho$ satisfies all Pl\"ucker relations. By \autoref{lemma: extended Plucker relations over idempotent F}, it suffices to show that $\rho$ satisfies $\sum_{k=0}^s \ \rho(\alpha -\epsilon_{i_k}+\epsilon_{i_0}+\dotsb+\epsilon_{i_s}) \cdot \rho(\alpha +\epsilon_{i_k}+\epsilon_{j_2}+\dotsb +\epsilon_{j_s})  \in  N_{F}$ for all $2\leq s\leq r$, $\alpha\in\Delta^{r-s}_{n}$, and $i_0,\dotsb,i_s,j_2,\dotsc,j_s\in[n]$ such that $\alpha + \epsilon_{i_0} + \dotsb + \epsilon_{i_s}+ \epsilon_{j_2} + \dotsb + \epsilon_{j_s} \leq \delta^+_J \leq r\1$. There are $\beta \in \Delta^{r-s}_{rn}$ and $i'_0,\dotsb,i'_s,j'_2,\dotsc,j'_s\in[rn]$ such that $\theta(\beta) = \alpha$, $\theta(i'_k) = i_k$ for all $k$, $\theta(j'_l) = j_l$ for all $l$, and $\beta + \epsilon_{i'_0} + \dotsb + \epsilon_{i'_s}+ \epsilon_{j'_2} + \dotsb + \epsilon_{j'_s} \leq \1$. Thus the above Pl\"ucker relations for $\rho$ follow from the Pl\"ucker relations for $\Up\rho$.
\end{proof}

As a consequence of \autoref{lemma: extended Plucker relations are preserved under the up operator}, we obtain an embedding
\[
 \PolyGr(r,n)(F) \ \longrightarrow \ \Gr(r,rn)(F)
\]
provided that $F$ is an idempotent fusion tract.

\begin{rem}
  Let \begin{equation}\label{eq:binomialbasis}
  g(w_1,w_2,\dots,w_n) = \sum_{\alpha \in \Delta^r_n} \ \binom{r}{\alpha} \ c_\alpha w^\alpha \end{equation} be a degree $r$ homogeneous polynomial in $n$ variables with nonnegative coefficients, where $w^\alpha = \prod_{i=1}^n w_i^{\alpha_i}$ and $\binom{r}{\alpha} = \prod_{i=1}^n \binom{r}{\alpha_i}$. The \emph{polarization} of $g$ is the degree $r$ homogeneous polynomial in $rn$ variables $w_{ij}$ with $i\in[n]$ and $j\in[r]$ given by \[ \Up g := \sum_{\alpha \in \Delta^r_n} c_\alpha \prod_{i=1}^n e_{\alpha_i}(w_{i1},w_{i2},\dots,w_{ir}), \] where $e_k(z_1,\dots,z_r)$ denotes the elementary symmetric polynomial of degree $k$. Br\"and\'en and Huh \cite[Prop.\ 3.1 and Thm.\ 3.14]{Branden-Huh20} proved that the polarization operator preserves the Lorentzian property of polynomials and that $(c_\alpha)_{\alpha\in\Delta^r_n}$ is a $\T_0$-representation of a polymatroid if and only if the polynomial $\sum_{\alpha \in \Delta^r_n} \binom{r}{\alpha} c_\alpha^q w^\alpha$ is Lorentzian for all $q \ge 0$. 
  We remark that for $q\geq0$, if we assign to a polynomial $g$ written as in 
  (\ref{eq:binomialbasis}) the function $\tilde\rho_g\colon \Delta^r_n \to \T_q$ given by $\tilde\rho_g(\alpha)=c_\alpha$, then we have $\Up\tilde\rho_g=\tilde\rho_{\Up g}$.
  However, note that $\tilde\rho_g$ differs from the  function  $\rho_g\colon \Delta^r_n \to \T_q$ that is assigned to a polynomial $g$ in \cite{BHKL1}, for which in general $\Up\rho_g\neq\rho_{\Up g}$.
\end{rem}

\subsubsection{The Polydressian} 

 Let $\overline{\R}=\R\cup\{\infty\}$ be the min-plus-algebra. Taking coefficientwise inverse logarithms defines the bijection $-\log\colon\P^N(\T_0)\to\overline{\R}^N/\overline{\R}$, where $N=\#\Delta^r_n-1$. We define the \emph{Polydressian} as the image $\PolyDr(r,n)=-\log(\PolyGr(r,n)(\T_0))$ of the Polygrassmannian over $\T_0$ under this map. The Polydressian $\PolyDr(r,n)$ contains the Dressian $\Dr(r,n)$, which is the union of the local Dressians $\Dr_J=-\log\big(\Gr_J(\T_0)\big)$ for which $J$ is a matroid. It follows from the results in \autoref{subsection: M-convex functions as representations over the tropical hyperfield} that the Polydressian $\PolyDr(r,n)$ is the set of M-convex functions $\Delta_n^r\to\R\cup\{\infty\}$ modulo constant functions.
 
 In rank $r=1$, every polymatroid is a matroid and thus $\Dr(1,n)=\PolyDr(1,n)$. For rank $r\geq2$ and $n\geq1$, the Polydressian is strictly larger. For example, $\Dr(r,1)$ is empty for $r\geq2$ since there is no matroid of rank $r\geq2$ on $[1]$, but $\PolyDr(r,1)$ is a singleton that contains the class of the trivial valuation $v\colon\Delta^r_1\to\overline\R$ with $v(r,0)=0$.
 
 A richer example is $\PolyDr(2,2)$. The Dressian $\Dr(2,2)$ is a singleton that consists of the class of the trivial valuation $v:U_{2,2}\to\overline\R$ with $v(1,1)=0$, where $U_{2,2}=\{(1,1)\}$ is the uniform matroid of rank $2$ on $[2]$. The Polydressian $\PolyDr$ has six strata $\Dr_J=-\log(\Gr_J(\T_0))$, labeled by the six M-convex subsets 
 \[
  \{(2,0)\}, \quad \{(1,1)\}, \quad \{(0,2)\}, \quad \{(2,0),(1,1)\}, \quad \{(2,0),(1,1)\}, \quad \Delta^2_2
 \]
 of $\Delta^2_2$. The (logarithmic) Pl\"ucker relations are trivially satisfied for all valuations of these M-convex sets, with the unique exception of $\Delta^2_2$, whose valuations must satisfy the condition that the minimum among
 \[
  v(2,0)+v(0,2), \qquad v(1,1)+v(1,1), \qquad \text{and} \qquad v(1,1)+v(1,1) 
 \]
 occurs at least twice, which means that $2v(1,1)\leq v(2,0)+v(0,2)\neq\infty$. See \autoref{fig: PolyDr22} for an illustration of $\PolyDr(2,2)$.

\begin{figure}[ht]
 \includegraphics{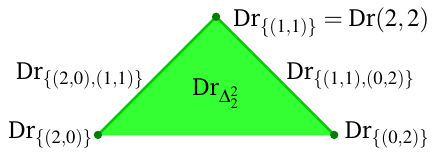}
 \caption{The Polydressian $\PolyDr(2,2)$}
 \label{fig: PolyDr22}
\end{figure}

\subsection{The Tutte rank}
\label{subsection: the Tutte rank}

If $M$ is a matroid, then $P_M^\times$ is canonically isomorphic to the Tutte group of $M$; cf.\ \cite[Thm.\ 6.27]{Baker-Lorscheid21b}. This justifies the following definition of the Tutte group of a polymatroid.

\begin{df}
 Let $J$ be an M-convex set. The \emph{Tutte group of $J$} is the abelian group $P_J^\times$, and the \emph{Tutte rank of $J$} is its free rank $\tau(J)=\rk P_J^\times$.
\end{df}

As an immediate consequence of \autoref{thm: bijection between the universal pasture and the universal tract}, the Tutte group is also canonically isomorphic to $T_J^\times$. The Tutte rank $\tau(J)$ is finite since $P_J^\times$ is finitely generated. 

 \begin{cor}\label{cor: Tutte rank as dimension of the thin Schubert cell over the degenerate triangular hyperfield}
 Let $J\subseteq\Delta^r_n$ be an M-convex set. Then $-\log\big(\Gr^w_J(\T_\infty)\big)$ is a real vector space whose dimension equals the Tutte rank $\tau(J)$ of $J$.
\end{cor}

\begin{proof}
 Since $\T_\infty$ is degenerate and every group homomorphism $f\colon\widehat P_J\to\T_\infty^\times$ maps the torsion element $-1$ to $1=-1$ in $\T_\infty^\times=\R_{>0}$, \autoref{prop: characterization of the degeneracy locus as group homomorphisms from the extended universal pasture} implies that 
 \[
  \upR^w_J(\T_\infty) \ = \ \upD_J(\T_\infty) \ = \ \Hom_\Z(\widehat P_J^\times,\ \R_{>0}).
 \]
 In conclusion $-\log\big(\Gr^w_J(\T_\infty)\big)=\Hom_\Z(P_J^\times,\ \R)$ is a real vector space of dimension $\tau(J)=\rk(P_J^\times)$. 
\end{proof}

\begin{rem}\label{rem:lordim}
It is shown in \cite{BHKL1} that for every M-convex set $J$, the dimension of the space of Lorentzian polynomials with support $J$ is equal to $\tau(J)$ + 1.
\end{rem}

For nonnegative integers $\kappa_1,\dots,\kappa_n$, we define the polymatroid \[ \Delta^r_{(\kappa_1,\dots,\kappa_n)} := \ \big\{ \alpha \in \Delta^r_n \mid \alpha_i \leq \kappa_i \text{ for all } i \big\}. \] Note that $\Delta^r_{(1,1,\dots,1)}$ is the uniform matroid $U_{r,n}$ and $\Delta^r_{(r,r,\ldots,r)} = \Delta^r_n$.
It was shown in \cite{Branden21} that the  space of Lorentzian polynomials with support $\Delta^r_{(\kappa_1,\dots,\kappa_n)}$ is full-dimensional. Thus, by \autoref{rem:lordim}, its  Tutte rank is equal to the number of bases minus one. 
We give a direct proof of this formula.

\begin{lemma}
  The Tutte rank of $\Delta^r_{(\kappa_1,\dots,\kappa_n)}$ is equal to $b-1$, where $b$ denotes the number of bases of $\Delta^r_{(\kappa_1,\dots,\kappa_n)}$.
\end{lemma}

\begin{proof}
  We write $J$ for $\Delta^r_{(\kappa_1,\dots,\kappa_n)}$. Every $3$-term Pl\"ucker relation in the universal pasture $\widehat P_J$ is non-degenerate. Thus the unit group $\widehat P^\times_J$ is the free abelian group generated by symbols $-1$ and $x_{\alpha}$ with $\alpha \in J$ modulo the relation $(-1)^2 = 1$ if $J$ is a translate of a matroid. Otherwise, $\widehat P^\times_J$ is the same free abelian group modulo the relation $-1=1$, which follows from the idempotency principle (\autoref{prop: idempotency principle}). Thus $P^\times_J = \ker(\deg) \simeq (\Z/2\Z) \times \Z^{b-1}$ or $\Z^{b-1}$.
\end{proof}

\begin{ex}
  The Tutte rank of $\Delta^2_2$ is two, and the Tutte rank of $U^+_{2,3} = \Delta^2_{(2,1,1)}$ is three.
\end{ex}

\section{The canonical torus embedding for realization spaces}
\label{section: the canonical torus embedding for realization spaces}

Let $J\subseteq\Delta^r_n$ be M-convex, $\Jbar$ its reduction, $\rbar$ its effective rank, and $\bJ=\{\balpha\in[n]^\rbar\mid\sum\balpha\in\Jbar\}$. The non-degenerate cross ratios of an $F$-representation $\brho\colon[n]^\rbar\to F$ of $J$ determine the map
\[
 \begin{array}{cccl}
  \varpi_J: & \ulineGr^w_J(F) & \longrightarrow & (F^\times)^{\Omega^\diamond_J} \\[3pt]
            & [\brho]       & \longmapsto     & \big[(\alpha,i,j,k,l)\mapsto\cross ijkl{\alpha,\brho}\big].
 \end{array}
\]
Let $\eta_J\colon(F^\times)^\bJ\to(F^\times)^{\Omega_J^\diamond}$ be the group homomorphism that maps a function $\brho\colon\bJ\to F^\times$ to the function $\eta_J(\brho)\colon\Omega_J\to F^\times$ given by
\[
 (\alpha,i,j,k,l) \ \longmapsto \ \cross ijkl{\alpha,\brho} \ = \ \frac{\brho(\balpha ik)\cdot\brho(\balpha jl)}{\brho(\balpha il)\cdot\brho(\balpha jk)},
\]
where $\balpha\in[n]^{\rbar-2}$ such that $\sum\balpha=\alpha$. Let $\pl_J\colon\upR^w_J\to (F^\times)^\bJ$ be the canonical torus embedding (cf.\ \autoref{section: the torus embedding of the representation space}) and $\pi_J\colon\upR^w_J(F)\to\ulineGr^w_J(F)$ the quotient map.

\begin{prop}\label{prop: torus embedding of the realization space}
 The map $\varpi_J$ is injective and 
 \[
  \begin{tikzcd}[column sep=60, row sep=20]
   \upR^w_J(F) \ar[r,"\pl_J"] \ar[d,"\pi_J"'] & (F^\times)^\bJ \ar[d,"\eta_J"] \\
   \ulineGr^w_J(F) \ar[r,"\varpi_J"] & (F^\times)^{\Omega_J^\diamond}
  \end{tikzcd}
 \]
 commutes.
\end{prop}

\begin{proof}
 By \autoref{prop: universal property of the foundation}, a class $[\brho]\in\upR^w_J(F)$ is uniquely determined by the associated morphism $f_\brho\colon F_J\to F$ (cf.\ \autoref{thm: generators and relations for the foundation}), which in turn is uniquely determined by the image of the non-degenerate cross ratios $f_\rho(\cross ijkl\alpha)=\cross ijkl{\alpha,\brho}$ since $F_J$ is generated by the non-degenerate cross ratios (\autoref{thm: generators and relations for the foundation}). Thus $\varpi_J$ is injective.
 
 The diagram commutes since 
 \begin{multline*}\textstyle
  (\varpi_J\circ\pi_J)(\brho) \ = \ \varpi_J([\brho]) \ = \ \big\{\cross ijkl{\alpha,\brho} \, \big| \, (\alpha,i,j,k,l)\in\Omega_J^\diamond\big\} \\
  = \ \eta_J\big(\{\brho(\bbeta)\mid\bbeta\in\bJ\}\big) \ = \ (\eta_J\circ\pl_J)(\brho)
 \end{multline*}
 for all $\brho\in\upR^w_J(F)$.
\end{proof}

\subsection{A decomposition of the representation space}
\label{subsection: decomposition of the representation space}

By \autoref{lemma: the universal tract is freely generated over the foundation}, the extended universal pasture is a free algebra over the foundation. The choice of an isomorphism $\widehat P_J\simeq F_J(x_1,\dotsc,x_s)$ yields a bijection
\begin{multline*}
 \Gr^w_J(F) \ = \ \Hom(P_J,F) \ \simeq \ \Hom(F_J(x_1,\dotsc,x_s),F) \\
 = \ \Hom(F_J,F) \ \times \ \Maps\big(\{x_1,\dotsc,x_s\},F^\times\big) \ = \ \ulineGr^w_J(F)\times (F^\times)^s
\end{multline*}
which is functorial in $F$ and which commutes with the projections onto $\ulineGr^w_J(F)$. More precisely, the following holds.

\begin{thm}\label{thm: decomposition of the representation space}
 Let $J\subseteq\Delta^r_n$ be M-convex and $J=\bigoplus_{i=1}^{c(J)} J_i$ a decomposition into indecomposable M-convex sets $J_i\subseteq\Delta^{r_i}_{n_i}$. Let $\brho\colon[n]^r\to F$ be an $F$-representation of $J$. Then the stabilizer of $[\brho]$ under the action of $T(F)=(F^\times)^n$ on $\Gr^w_J(F)$ is
 \[
  \Stab_{T(F)}([\brho]) \ = \ \big\{ t\in T(F) \, \big| \, t_{n_{i-1}+1}=\dotsc=t_{n_i}\text{ for all }i=1,\dotsc,c(J) \big\},
 \]
 where $n_0=0$, and the orbit $T(F).[\brho]$ is in bijective correspondence with $(F^\times)^{n-c(J)}$. 
 
 In particular, $P_J\simeq F_J(x_1,\dotsc,x_{n-c(J)})$. If $F$ is idempotent, then $\upLin_J(F)\simeq(F^\times)^{n-c(J)}$.
\end{thm}

\begin{proof}
 We begin with the first claim. An element $t\in T(F)$ is in the stabilizer of $[\brho]$ iff $[t.\brho]=[\brho]$, i.e., iff there is an element $c\in F^\times$ such that $(t.\brho)(\bbeta)=c\cdot\brho(\bbeta)$ for all $\bbeta\in J$, which means that $\prod_{j=1}^n t_j^{\bbeta_j}=c$.
 
 By \autoref{thm: representations of direct sums}, $\brho(\bbeta)=\prod_{i=1}^{c(J)}\brho_i(\bbeta\vert_{J_i})$ for certain $F$-representations $\brho_i\colon[n_i]^{r_i}\to F$ of $J_i$ and $\bbeta\vert_{J_i}=(\bbeta_{n_{i-1}+1},\dotsc,\bbeta_{n_i})$.
 
 Assume that $t_{n_{i-1}+1}=\dotsc=t_{n_i}$ for all $i=1,\dotsc,c(J)$ and define $c=\prod_{i=1}^{c(J)} t_{n_i}^{r_i}$. Then $\prod_{j=1}^n t_j^{\bbeta_j}=c$ for all $\bbeta\in J$ since $\norm{\bbeta\vert_{J_i}}=r_i$, and thus $t\in\Stab_{T(F)}([\brho])$.
 
 Conversely, assume that $t\in\Stab_{T(F)}([\brho])$. We want to show that $t_{n_{i-1}+1}=\dotsc=t_{n_i}$ for each $i=1,\dotsc,c(J)$, which can be established separately for each component $J_i$. After replacing $J$ by $J_i$, this allows us to assume that $J$ is indecomposable for simplicity. We establish the claim by showing through an induction on $s=1,\dotsc,n$ that there is a subset $S\subseteq[n]$ of cardinality $s$ that contains $t_n$ and such that $t_i=t_n$ for all $i\in S$.
 
 If $s=1$, then $S=\{t_n\}$ satisfies the claim. In order to establish the inductive step, assume that $S$ is a proper subset of $[n]$ that contains $t_n$ and such that $t_i=t_n$ for all $i\in S$. By \autoref{lemma: characterization of indecomposable polymatroids}, there are $\beta,\gamma\in J$ with $\beta_S<\gamma_S$. Applying the exchange axiom repeatedly yields a sequence $\beta=\beta^{(0)},\dotsc,\beta^{(m)}=\gamma$ of elements in $J$ with $\sum_{i\in[n]}\norm{\beta^{(j)}_i-\beta^{(j-1)}_i}=2$ for all $j=1,\dotsc,m$. In particular, there is a $j$ for which $\beta^{(j)}_S=\beta^{(j-1)}_S+1$. 
 
 After replacing $\beta$ by $\beta^{(j-1)}$ and $\gamma$ by $\beta^{(j)}$, this means that there are $k\in S$ and $\ell\in[n]-S$ such that $\gamma_k=\beta_k+1$, $\gamma_\ell=\beta_\ell-1$, and $\gamma_i=\beta_i$ for all $i\neq k,\ell$. Since $t\in\Stab_{T(F)}(\brho)$, we have $\prod_{i\in[n]} t_i^{\beta_i}=\prod_{i\in[n]} t_i^{\gamma_i}$, and thus $t_\ell=t_k$. This shows that $S\cup\{\ell\}$ satisfies the claim of the induction, which establishes the inductive step. This completes the proof of the first claim of the proposition.

 The orbit $T(F).[\brho]$ is in bijection with $T(F)/\Stab_{T(F)}(\brho)$, which is isomorphic (as a group) to $(F^\times)^{n-c(J)}$, which establishes the second claim. For a finite proper extension of $\K$ (e.g.\ $F=\H\otimes\K$), counting the elements of $\ulineGr^w_J(F)\times(F^\times)^{n-c(J)}$$=\ulineGr^w_J(F)\times(F^\times)^{s}$ implies that $s=n-c(J)$, where $s\in\N$ is chosen so that $\widehat P_J\simeq F_J(x_1,\dotsc,x_s)$. Thus $P_J\simeq F_J(x_1,\dotsc,x_{n-c(J)})$. If $F$ is idempotent, we can choose $\brho$ to be the trivial $F$-representation $\bchi_{J,F}\colon[n]^r\to F$, which yields $\upLin_J(F)=T(F).[\bchi_{J,F}]\simeq(F^\times)^{n-c(J)}$. This establishes the last claim.
\end{proof}

\section{Topologies for representation spaces and realization spaces}
\label{subsection: topologies for representation spaces and realization spaces}

If a tract $F$ comes with a topology, then this induces a topology on the (weak) representation and realization spaces. If the topology of $F$ is sufficiently nice, then this topology has several equivalent characterizations. 
The following definition is analogous to the notion of a topological idyll (cf.\ \cite{Baker-Jin-Lorscheid24}).

\begin{df}
 A \emph{topological tract} is a tract $F$ together with a topology such that the following holds:
 \begin{enumerate}[label = (TP\arabic*)]
  \item\label{TP1} $\{0\}\subseteq F$ is a closed subset.
  \item\label{TP2} The multiplication $m\colon F\times F\to F$ is continuous (where $F\times F$ carries the product topology).
  \item\label{TP3} The inversion $i\colon F^\times\to F^\times$ is continuous (where $F^\times\subseteq F$ carries the subspace topology).
  \item\label{TP4} For all $n\in\N$, the set $\{(a_1,\ldots,a_n)\in F^n\mid a_1+\ldots+a_n\in N_F\}$ is a closed subset of $F^n$ (where $F^n$ carries the product topology).
 \end{enumerate}
\end{df}

Note that \ref{TP1} is equivalent to $F^\times$ being an open subset of $F$. Axioms \ref{TP2} and \ref{TP3} imply that $F^\times$ is a topological group. Examples of topological tracts are topological fields and the triangular hyperfields $\T_q$ (for all $0\leq q<\infty$) with the euclidean topology for $\T_q=\R_{\geq0}$. Every tract is a topological tract with respect to the discrete topology.

If $F$ is a topological tract and $F'$ an arbitrary tract, we endow
$\Hom(F',F)$
with the compact-open topology where we equip $F'$ with the discrete topology. 
In other words, $\Hom(F',F)$ is endowed with the coarsest topology such that for every $a\in F'$, the map
\[
 \begin{array}{cccc}
 \ev_a\colon & \Hom(F',\ F) & \longrightarrow & F \\
        & [f\colon F'\to F]  & \longmapsto     & f(a)
 \end{array}
\]
is continuous. In this way, we obtain a topology on $\ulineGr^w_J(F)=\Hom(F_J,F)$, and similarly on (weak) representation spaces and (weak) thin Schubert cells. 
(For the sake of brevity, we restrict our discussion to the weak setup.)

\begin{prop}\label{prop: representation spaces for topological tracts}
  Let $F$ be a topological tract. Then $\upR^w_J(F)$ is a closed subspace of $(F^\times)^J$ (considered with the product topology).
\end{prop}

\begin{proof}
    For every $\alpha\in J$, let $x_\alpha$ be the generator of $\widehat P_J$ indexed by the element $\balpha\in[n]^r$ with $\Sigma\balpha=\alpha$ whose entries are ordered increasingly. Then the bijection $\Hom(\widehat P_J,F)\to\upR^w_J(F)$ is given by $\varphi\mapsto(\ev_{x_\alpha}(\varphi))_{\alpha\in J}$, and thus is continuous. Because $-1$ and the $x_\alpha$ generate the unit group of $\widehat P_J$, every value of $\varphi\in\Hom(\widehat P_J,F)$ can be expressed in terms of products and inverses of constants and the $\ev_{x_\alpha}(\varphi)$. Since $F^\times$ is a topological group, this shows that the map inverse to $\Hom(\widehat P_J,F)\to\upR^w_J(F)$ is also continuous.
    The subspace $\upR^w_J(F)$ being closed in $(F^\times)^J$ follows directly from the definition of a topological tract.
\end{proof}

\begin{rem}\label{rem: representation space for the degenerate triangular hyperfield}
 Note that also $\upR^w_J(\T_\infty)=\upD_J(\T_\infty)$ is a closed subset of $(\T_\infty^\times)^J=\R_{>0}^J$, even though the order topology for $\T_\infty=\R_{\geq0}$ fails to satisfy \ref{TP4}.
\end{rem}

\begin{prop}
 Let $T$ be a topological tract and $F=\colim\cF$ for a finite diagram $\cF$ of tracts, then the canonical bijection $\Hom(F,T)\to\lim\Hom(\cF,T)$ is a homeomorphism.
\end{prop}

\begin{proof}
 This follows from \cite[Thm.\ 3.5]{Baker-Jin-Lorscheid25}.
\end{proof}

\begin{cor}\label{cor:product}
 Let $F$ be a topological tract, $J\subseteq\Delta^r_n$ an M-convex set with $c(J)$ indecomposable components and $s=n-c(J)$. Then there are (non-canonical) homeomorphisms $\upR^w_J(F)\to\ulineGr_J^w(F)\times(F^\times)^{s+1}$ and $\Gr^w_J(F)\to\ulineGr_J^w(F)\times(F^\times)^{s}$ such that the following diagram commutes:
 \begin{equation}\label{eq:productdiagram}
    \begin{tikzcd}
        \upR^w_J(F)\arrow{d}\arrow{r}&\Gr^w_J(F)\arrow{d}\arrow{r}&\ulineGr_J^w(F)\arrow{d}\\
        \ulineGr_J^w(F)\times(F^\times)^{s+1} \arrow{r}& \ulineGr_J^w(F)\times(F^\times)^{s}  \arrow{r}& \ulineGr_J^w(F)
    \end{tikzcd}.
 \end{equation}
 Here the maps in the first row take elements to their equivalence classes and the maps in the second row are coordinate projections.
\end{cor}
\begin{proof}
    By \autoref{lemma: the universal tract is freely generated over the foundation} we have
    \begin{equation*}
        P_J\simeq F_J(x_1,\ldots,x_s)\simeq F_J\otimes \Funpm(x_1,\ldots,x_s).
    \end{equation*}
    Thus by the preceding proposition we have a homeomorphism
    \begin{equation*}
        \Gr_J^w(F)=\Hom(P_J,F)\simeq\Hom(F_J,F)\times\Hom(\Funpm(x_1,\ldots,x_s),F)=\ulineGr_J^w(F)\times(F^\times)^{s},
    \end{equation*}
    and similarly for $\widehat P_J$.
\end{proof}
\begin{cor}
 Let $F$ be a topological tract and $J\subseteq\Delta^r_n$ an M-convex set.  The spaces $\Gr_J^w(F)$ and $\ulineGr_J^w(F)$ carry the quotient topology of $\upR^w_J(F)$.
\end{cor}
\begin{proof}
    It follows from \autoref{cor:product} that the natural maps $\upR^w_J(F)\to\Gr_J^w(F)$ and $\upR^w_J(F)\to\ulineGr_J^w(F)$ are open and continuous, which implies the claim.
\end{proof}
\begin{cor}
 Let $F$ be a topological tract, $J\subseteq\Delta^r_n$ an M-convex set with $c(J)$ indecomposable components, and $s=n-c(J)$. Let $\widehat T(F)=F^\times\times T(F)$ be the extended torus acting on $\upR^w_J(F)$. Then the orbits of $F^\times$ and $\widehat T(F)$ are closed in $\upR^w_J(F)$, and the orbits are homeomorphic to $F^\times$ and $(F^\times)^{s+1}$, respectively.   
\end{cor}
\begin{proof}
    This follows from \autoref{cor:product} because the fibers of the maps $\ulineGr_J^w(F)\times(F^\times)^{s+1} \to \ulineGr_J^w(F)\times(F^\times)^{s}$ and  $\ulineGr_J^w(F)\times(F^\times)^{s+1}\to \ulineGr_J^w(F)$ in the second row of \autoref{eq:productdiagram} are closed and homeomorphic to $F^\times$ and $(F^\times)^{s+1}$, respectively.
\end{proof}

\begin{small}
 \bibliographystyle{plain}
 \bibliography{lorentzian}

\begin{thebibliography}{10}

\bibitem{Adiprasito-Huh-Katz18}
Karim Adiprasito, June Huh, and Eric Katz.
\newblock Hodge theory for combinatorial geometries.
\newblock {\em Ann. of Math. (2)}, 188(2):381--452, 2018.

\bibitem{Amini-Esteves24}
Omid Amini and Eduardo Esteves.
\newblock Tropicalization of linear series and tilings by polymatroids.
\newblock Preprint, \arxiv{2405.04306}, 2024.

\bibitem{Baker-Bowler19}
Matthew Baker and Nathan Bowler.
\newblock Matroids over partial hyperstructures.
\newblock {\em Adv. Math.}, 343:821--863, 2019.

\bibitem{BHKL1}
Matthew Baker, June Huh, Mario Kummer, and Oliver Lorscheid.
\newblock Lorentzian polynomials and matroids over triangular hyperfields {I}.
\newblock Preprint, \arxiv{arXiv:2508.02907}, 2025.

\bibitem{BHKL2}
Matthew Baker, June Huh, Mario Kummer, and Oliver Lorscheid.
\newblock Lorentzian polynomials and matroids over triangular hyperfields.
  {P}art 2: Sandwich theorems.
\newblock in preparation.

\bibitem{Baker-Jin-Lorscheid24}
Matthew Baker, Tong Jin, and Oliver Lorscheid.
\newblock New building blocks for {$\mathbb{F}_1$}-geometry: bands and band
  schemes.
\newblock {\em J. Lond. Math. Soc. (2)}, 111(4):Paper No. e70125, 62, 2025.

\bibitem{Baker-Jin-Lorscheid25}
Matthew Baker, Tong Jin, and Oliver Lorscheid.
\newblock Tutte's homotopy theorem from a modern perspective.
\newblock Preprint, in preparation, 2025.

\bibitem{Baker-Lorscheid21b}
Matthew Baker and Oliver Lorscheid.
\newblock The moduli space of matroids.
\newblock {\em Adv. Math.}, 390:Paper No. 107883, 118, 2021.

\bibitem{Baker-Lorscheid20}
Matthew Baker and Oliver Lorscheid.
\newblock Foundations of matroids, {P}art 1: {M}atroids without large uniform
  minors.
\newblock {\em Mem. Amer. Math. Soc.}, 305(1536):v+84, 2025.

\bibitem{Baker-Lorscheid21}
Matthew Baker and Oliver Lorscheid.
\newblock Lift theorems for representations of matroids over pastures.
\newblock {\em J. Combin. Theory Ser. B}, 170:1--55, 2025.

\bibitem{Baker-Lorscheid-Zhang24}
Matthew Baker, Oliver Lorscheid, and Tianyi Zhang.
\newblock Foundations of matroids, {P}art 2: {F}urther theory, examples, and
  computational methods.
\newblock {\em Comb. Theory}, 5(1):Paper No. 1, 77, 2025.

\bibitem{Baker-Zhang23}
Matthew Baker and Tianyi Zhang.
\newblock Fusion rules for pastures and tracts.
\newblock {\em European J. Combin.}, 108:Paper No. 103628, 15, 2023.

\bibitem{Bernardi-Kalman-Postnikov22}
Olivier Bernardi, Tam\'as K\'alm\'an, and Alexander Postnikov.
\newblock Universal {T}utte polynomial.
\newblock {\em Adv. Math.}, 402:Paper No. 108355, 74, 2022.

\bibitem{Bonin-Chun-Fife23}
Joseph~E. Bonin, Carolyn Chun, and Tara Fife.
\newblock The natural matroid of an integer polymatroid.
\newblock {\em SIAM J. Discrete Math.}, 37(3):1751--1770, 2023.

\bibitem{Branden07}
Petter Br{\"a}nd{\'e}n.
\newblock Polynomials with the half-plane property and matroid theory.
\newblock {\em Adv. Math.}, 216(1):302--320, 2007.

\bibitem{Branden10}
Petter Br\"and\'en.
\newblock Discrete concavity and the half-plane property.
\newblock {\em SIAM J. Discrete Math.}, 24(3):921--933, 2010.

\bibitem{Branden11}
Petter Br{\"a}nd{\'e}n.
\newblock Obstructions to determinantal representability.
\newblock {\em Adv. Math.}, 226(2):1202--1212, 2011.

\bibitem{Branden21}
Petter Br\"and\'en.
\newblock Spaces of {L}orentzian and real stable polynomials are {E}uclidean
  balls.
\newblock {\em Forum Math. Sigma}, 9:Paper No. e73, 8, 2021.

\bibitem{Branden-Huh20}
Petter Br\"{a}nd\'{e}n and June Huh.
\newblock Lorentzian polynomials.
\newblock {\em Ann. of Math. (2)}, 192(3):821--891, 2020.

\bibitem{Brandt-Speyer22}
Madeline Brandt and David~E. Speyer.
\newblock Computation of {Dressians} by dimensional reduction.
\newblock {\em Adv. Geom.}, 22(3):409--420, 2022.

\bibitem{Buch00}
Anders~Skovsted Buch.
\newblock The saturation conjecture (after {A}.\ {K}nutson and {T}.\ {T}ao).
\newblock {\em Enseign. Math. (2)}, 46(1-2):43--60, 2000.
\newblock With an appendix by William Fulton.

\bibitem{Choe-Oxley-Sokal-Wagner04}
Young-Bin Choe, James~G. Oxley, Alan~D. Sokal, and David~G. Wagner.
\newblock Homogeneous multivariate polynomials with the half-plane property.
\newblock {\em Adv. Appl. Math.}, 32(1-2):88--187, 2004.

\bibitem{Crowley-Huh-Larson-Simpson-Wang22}
Colin Crowley, June Huh, Matt Larson, Connor Simpson, and Botong Wang.
\newblock The {B}ergman fan of a polymatroid.
\newblock Preprint, \arxiv{2207.08764}, 2022.

\bibitem{Crowley-Simpson-Wang24}
Colin Crowley, Connor Simpson, and Botong Wang.
\newblock Polymatroid {S}chubert varieties.
\newblock Preprint, \arxiv{2410.10552}, 2024.

\bibitem{Derksen-Fink10}
Harm Derksen and Alex Fink.
\newblock Valuative invariants for polymatroids.
\newblock In {\em 22nd {I}nternational {C}onference on {F}ormal {P}ower
  {S}eries and {A}lgebraic {C}ombinatorics ({FPSAC} 2010)}, volume~AN of {\em
  Discrete Math. Theor. Comput. Sci. Proc.}, pages 271--282. Assoc. Discrete
  Math. Theor. Comput. Sci., Nancy, 2010.

\bibitem{Edmonds70}
Jack Edmonds.
\newblock Submodular functions, matroids, and certain polyhedra.
\newblock In {\em Combinatorial {S}tructures and their {A}pplications ({P}roc.
  {C}algary {I}nternat. {C}onf., {C}algary, {A}lta., 1969)}, pages 69--87.
  Gordon and Breach, New York-London-Paris, 1970.

\bibitem{Eur-Larson24}
Christopher Eur and Matt Larson.
\newblock Intersection theory of polymatroids.
\newblock {\em Int. Math. Res. Not. IMRN}, (5):4207--4241, 2024.

\bibitem{Farras-Marti-Farre-Padro12}
Oriol Farr\`as, Jaume Mart\'i-Farr\'e, and Carles Padr\'o.
\newblock Ideal multipartite secret sharing schemes.
\newblock {\em J. Cryptology}, 25(3):434--463, 2012.

\bibitem{Fulton98}
William Fulton.
\newblock Eigenvalues of sums of {H}ermitian matrices (after {A}. {K}lyachko).
\newblock Number 252, pages Exp. No. 845, 5, 255--269. 1998.
\newblock S\'eminaire Bourbaki. Vol.\ 1997/98.

\bibitem{Gelfand-Goresky-MacPherson-Serganova87}
I.~M. Gelfand, R.~M. Goresky, R.~D. MacPherson, and V.~V. Serganova.
\newblock Combinatorial geometries, convex polyhedra, and {S}chubert cells.
\newblock {\em Adv. in Math.}, 63(3):301--316, 1987.

\bibitem{Gelfand-Serganova87}
I.~M. Gelfand and V.~V. Serganova.
\newblock Combinatorial geometries and the strata of a torus on homogeneous
  compact manifolds.
\newblock {\em Uspekhi Mat. Nauk}, 42(2(254)):107--134, 287, 1987.

\bibitem{Herzog-Hibi02}
J\"urgen Herzog and Takayuki Hibi.
\newblock Discrete polymatroids.
\newblock {\em J. Algebraic Combin.}, 16(3):239--268, 2002.

\bibitem{Khan-Maclagan24}
Bivas Khan and Diane Maclagan.
\newblock Tropical vector bundles.
\newblock Preprint \arxiv{2405.03505}, 2024.

\bibitem{Klyachko98}
Alexander~A. Klyachko.
\newblock Stable bundles, representation theory and {H}ermitian operators.
\newblock {\em Selecta Math. (N.S.)}, 4(3):419--445, 1998.

\bibitem{Knutson-Tao99}
Allen Knutson and Terence Tao.
\newblock The honeycomb model of {${\rm GL}_n({\bf C})$} tensor products. {I}.
  {P}roof of the saturation conjecture.
\newblock {\em J. Amer. Math. Soc.}, 12(4):1055--1090, 1999.

\bibitem{Murota03}
Kazuo Murota.
\newblock {\em Discrete convex analysis}.
\newblock SIAM Monographs on Discrete Mathematics and Applications. Society for
  Industrial and Applied Mathematics (SIAM), Philadelphia, PA, 2003.

\bibitem{Oxley-Semple-Whittle16}
James Oxley, Charles Semple, and Geoff Whittle.
\newblock A wheels-and-whirls theorem for 3-connected 2-polymatroids.
\newblock {\em SIAM J. Discrete Math.}, 30(1):493--524, 2016.

\bibitem{Oxley-Whittle93a}
James Oxley and Geoff Whittle.
\newblock A characterization of {T}utte invariants of {$2$}-polymatroids.
\newblock {\em J. Combin. Theory Ser. B}, 59(2):210--244, 1993.

\bibitem{Oxley-Whittle93}
James Oxley and Geoff Whittle.
\newblock Some excluded-minor theorems for a class of polymatroids.
\newblock {\em Combinatorica}, 13(4):467--476, 1993.

\bibitem{Pagaria-Pezzoli23}
Roberto Pagaria and Gian~Marco Pezzoli.
\newblock Hodge theory for polymatroids.
\newblock {\em Int. Math. Res. Not. IMRN}, (23):20118--20168, 2023.

\bibitem{Pendavingh-vanZwam10b}
Rudi~A. Pendavingh and Stefan H.~M. van Zwam.
\newblock Lifts of matroid representations over partial fields.
\newblock {\em J. Combin. Theory Ser. B}, 100(1):36--67, 2010.

\bibitem{Postnikov-Reiner-Williams08}
Alex Postnikov, Victor Reiner, and Lauren Williams.
\newblock Faces of generalized permutohedra.
\newblock {\em Doc. Math.}, 13:207--273, 2008.

\bibitem{Postnikov09}
Alexander Postnikov.
\newblock Permutohedra, associahedra, and beyond.
\newblock {\em Int. Math. Res. Not. IMRN}, (6):1026--1106, 2009.

\bibitem{Semple-Whittle96}
Charles Semple and Geoff Whittle.
\newblock Partial fields and matroid representation.
\newblock {\em Adv. in Appl. Math.}, 17(2):184--208, 1996.

\bibitem{Tutte58a}
William~T. Tutte.
\newblock A homotopy theorem for matroids, {I}.
\newblock {\em Trans. Amer. Math. Soc.}, 88:144--160, 1958.

\bibitem{Tutte58b}
William~T. Tutte.
\newblock A homotopy theorem for matroids, {II}.
\newblock {\em Trans. Amer. Math. Soc.}, 88:161--174, 1958.

\bibitem{Viro10}
Oleg Viro.
\newblock Hyperfields for tropical geometry {I}. {H}yperfields and
  dequantization, 2010.
\newblock Preprint, \arxiv{1006.3034}.

\bibitem{Vladoiu06}
Marius Vl\u{a}doiu.
\newblock Discrete polymatroids.
\newblock {\em An. \c Stiin\c t. Univ. ``Ovidius'' Constan\c ta Ser. Mat.},
  14(2):97--120, 2006.

\bibitem{Wenzel1996}
Walter Wenzel.
\newblock Maurer's homotopy theory and geometric algebra for even
  {$\Delta$}-matroids.
\newblock {\em Adv. in Appl. Math.}, 17(1):27--62, 1996.

\bibitem{Whittle92}
Geoff Whittle.
\newblock Duality in polymatroids and set functions.
\newblock {\em Combin. Probab. Comput.}, 1(3):275--280, 1992.

\end{thebibliography}
\end{small}

\end{document}